\documentclass[a4paper]{amsart}

\usepackage{amsmath}
\usepackage{url}
\usepackage{amsthm}
\usepackage{amsfonts}
\usepackage{appendix}
\usepackage{mathrsfs}
\usepackage{color}
\usepackage{enumerate}
\usepackage{amssymb}
\usepackage{tikz} 
\usepackage{fancyhdr}
\usetikzlibrary{calc}

\newtheorem{theorem}{Theorem}[section]

\newtheorem{lemma}[theorem]{Lemma}
\newtheorem{proposition}[theorem]{Proposition}
\newtheorem{corollary}[theorem]{Corollary}
\newtheorem{mtheorem}{Main Theorem}
\theoremstyle{definition}
\newtheorem{definition}[theorem]{Definition}

\theoremstyle{remark}

\numberwithin{equation}{section}

\newcommand{\id}{\textup{id}}

\numberwithin{equation}{section}

\makeatletter
\newcommand*\owedge{\mathpalette\@owedge\relax}
\newcommand*\@owedge[1]{%
  \mathbin{%
    \ooalign{%
      $#1\m@th\bigcirc$\cr
      \hidewidth$#1\m@th\wedge$\hidewidth\cr
    }%
  }%
}
\makeatother

\makeatother
\numberwithin{equation}{section}

\title[Ricci flow on manifolds with boundary]{Ricci flow on manifolds with boundary with arbitrary initial metric}

\author{Tsz-Kiu Aaron Chow}

\address{Department of Mathematics, Columbia University, New York, NY}

\email{achow@math.columbia.edu}

\begin{document}

\maketitle

\begin{abstract}
In this paper, we study the Ricci flow on manifolds with boundary. In the paper, we substantially improve Shen's result \cite{Sh} to manifolds with arbitrary initial metric. We prove short-time existence and uniqueness of the solution, in which the boundary becomes instantaneously totally geodesic for positive time. Moreover, we prove that the flow we constructed preserves natural boundary conditions. More specifically, if the initial metric has a convex boundary, then the flow preserves positive curvature operator and the PIC1, PIC2 conditions. Moreover, if the initial metric has a two-convex boundary, then the flow preserves the PIC condition.
\end{abstract}

\tableofcontents 

\section{Introduction}
\bigskip

In this article, we study the deformation of Riemannian metrics on compact smooth manifolds with boundary by Ricci flow. The study of Ricci flow dates back to the work of Hamilton \cite{Ham82}, who applied it to prove that every three dimensional closed compact Riemannian manifold with positive Ricci curvature admits a metric of constant positive sectional curvature. Since then, many significant results on the interaction between geometry and topology were established using Ricci flow. One important geometric property Ricci flow enjoys is that it preserves various conditions curvatures. A crucial ingredient in \cite{Ham82} is the preservation of positive Ricci curvature along Ricci flow on closed compact Riemannian manifolds in dimension three. In dimension four, Hamilton showed that the Ricci flow on closed compact Riemannian manifolds preserves positive curvature operator \cite{Ham86} and positive isotropic curvature \cite{Ham97}, which were used in subsequent convergence results of the Ricci flow. In higher dimensions, the natural curvature conditions PIC, PIC1, PIC2 are preserved by the Ricci flow on closed compact Riemannian manifolds, as observed by Brendle \cite{Bre08}, Brendle and Schoen \cite{BS}, Nguyen \cite{Ngu}. These curvature conditions play vital roles in general convergence of Ricci flow in higher dimensions, as in the results of Brendle \cite{Bre08, Bre19}.  

A natural question to ask is whether these results can be generalized to smooth manifolds with boundary. More specifically, we would like to study short-time existence problem for Ricci flow on manifold with boundary. Moreover, we would like the flow to preserve natural curvature conditions as its counterpart in the case of closed manifolds. The challenge to the short-time existence problem is the diffoemorphism invariance of the Ricci curvature. One would need to study carefully the geometry of the boundary in order to set up a well-posed boundary value problem. The challenge to the preservation of curvature conditions is the failure of usual parabolic maximum principle on the boundary, if one asks for an arbitrary initial metric. 

The first work in this direction was done by Shen \cite{Sh}, where he proved short-time existence to the Ricci flow on smooth compact manifold with umbilic boundary. One might ask for deformation of a more arbitrary metric. Later on, short-time existence results for Ricci flow on smooth compact manifold with boundary with an arbitrary initial metric were established in the works of Pulemotov \cite{Pul} and Gianniotis \cite{Gian}. In particular, Gianniotis \cite{Gian} has proven both short-time existence and uniqueness results, where he set up the evolution equations by prescribing the mean curvature and the conformal class of the boundary. However, there are no results so far on preserving curvature conditions in this setting. In dimension 2, more results were contributed by Brendle \cite{Bre02}, Cortissoz and Murcia \cite{Cort}.

Our work substantially improves the result of \cite{Sh}, where we prove short-time existence and uniqueness of solutions to Ricci flow on manifold with boundary in which the boundary become instantaneously umbilic for positive time. We remark that our result does not require the initial metric to have a umbilic boundary. We approach the problem via doubling of the manifold. Extending the initial metric to the doubled manifold by reflection, we obtain an extended metric which is merely H\"older continuous. We seek a solution to Ricci flow on the doubled manifold with a H\"older continuous initial metric that is smooth for positive time. To that end, we proved:

\bigskip
\begin{mtheorem}
	Let $M$ be a closed compact smooth manifold and $g_0\in C^{\alpha}(M)$ be a Rimannian metric for some $\alpha\in (0,1)$. Let $k\geq 2$, $\gamma\in (0,\alpha)$ and $\beta\in (\gamma,\alpha)$ be given. Then there exists a $C^{1,\beta}$ diffeomorphism $\psi$ and $T = T(M, \|g_0\|_{\alpha})$, $K = K(M, k,  \|g_0\|_{\alpha})$ such that the following holds:\\
	
	There is a solution $g(t)\in\mathcal{X}_{k,\gamma}^{(\beta)}(M\times[0,T])$ to the Ricci flow
	\begin{align*}
		\frac{\partial}{\partial t}g(t)=-2Ric(g(t))\quad&\text{on}\quad \tilde{M}\times(0,T]
	\end{align*}
	such that $g(0) = \psi^*g_0$ and
 	$$\|g\|_{\mathcal{X}_{k,\gamma}^{(\beta)}(M\times[0,T])}\leq K. $$	

\end{mtheorem}
\bigskip
Here the Banach spaces $\mathcal{X}_{k,\gamma}^{(\beta)}(M\times (0,T])$ will be defined in the beginning of section 4. We also show that the above theorem gives a canonical solution to the Ricci flow, in the sense of the following uniqueness theorem:

\bigskip

\begin{mtheorem}
 Let $\alpha\in (0,1)$ be given.	Let $M$ be a closed compact smooth manifold and $g_0\in C^{\alpha}(M)$ be a Riemannian metric on $M$. Suppose that the pairs $(g^1(t), \psi^1)$ and $(g^2(t), \psi^2)$ satisfy the conclusion of Main Theorem 1. Then there exists a $C^{k+1}$ diffeomorphism $\varphi:M\to M$ such that
 	 $$g^2(t) = \varphi^*(g^1(t)),$$
 	 where $\psi^2 = \psi^1\circ\varphi$. In particular, $((\psi^1)^{-1})^*g^1(t) = ((\psi^2)^{-1})^*g^2(t)$.
\end{mtheorem}

\bigskip

Via doubling the above results imply the existence and uniqueness results for Ricci flow on manifolds with boundary:

\bigskip

\begin{mtheorem}
	Let $(M,g_0)$ be a compact smooth Riemannian manifold with boundary. Let $k\geq 2$, $\beta\in (0,1)$ and $\varepsilon\in (0, 1-\beta)$ be given. Then there exists a $C^{1+\beta}$ diffeomorphism $\psi$ and $T = T(M, \hat{g}, \|g_0\|_{\beta+\varepsilon})$, $K = K(M, k, \hat{g},  \|g_0\|_{\beta+\varepsilon})$ such that the following holds:\\
	
	There is a solution $g(t)\in\mathcal{X}_{k,\gamma}^{(\beta)}(M\times[0,T])$ to the Ricci flow on manifold with boundary
\begin{align}
    \begin{cases}
    \frac{\partial}{\partial t}g(t)=-2Ric(g(t))\quad&\text{on}\quad M\times(0,T]\\
    A_{g(t)}=0\quad\quad\quad\quad\quad &\text{on}\quad \partial M\times (0, T]  
    \end{cases}
\end{align}
	such that $g(0) = \psi^*g_0$ and
	 $$\|g\|_{\mathcal{X}_{k,\gamma}^{(\beta)}(M\times[0,T])}\leq K.$$	
Here $A_{g(t)}$ stands for the second fundamental form of the boundary with respect to the metric $g(t)$. For each $t>0$, the metric $g(t)$ extends smoothly to the doubled manifold $\tilde{M}$ of $M$, and the doubled metric lies in $\mathcal{X}_{k,\gamma}^{(\beta)}(\tilde{M}\times[0,T])$. The diffeomorphism $\psi$ also extends to a $C^{1+\beta}$ diffeomorphism on the doubled manifold. 
\end{mtheorem}

\vskip 0.2cm

\begin{mtheorem}
	Let $(M,g_0)$ be a compact smooth Riemannian manifold with boundary. Suppose that the pairs $(g^1(t), \psi^1)$ and $(g^2(t), \psi^2)$ satisfy the conclusion of Main Theorem 3. Then there exists a $C^{k+1}$ diffeomorphism $\varphi:M\to M$ such that $\varphi$ extends to a $C^{k+1}$ diffeomorphism on the doubled manifold and
 	 $$g^2(t) = \varphi^*(g^1(t)),$$
 	 where $\psi^2 = \psi^1\circ\varphi$. In particular, $((\psi^1)^{-1})^*g^1(t) = ((\psi^2)^{-1})^*g^2(t)$.
\end{mtheorem}
\bigskip

Next, we show that the canonical solution constructed above preserves natural curvature conditions, provided that the geometry of the boundary is controlled. More precisely, we prove that the flow preserves positive curvature operator, PIC1 and PIC2 conditions, provided that the boundary is convex with respect to the initial metric. Moreover, if the initial metric has a two-convex boundary, we prove that the flow preserves the PIC conditions. We recall the definitions of various curvature conditions we are interested in:
\bigskip
\begin{definition}\
 \begin{itemize}
    \item[(i)]  We say that $(M,g)$
    has positive curvature operator if
    $R(\varphi,\varphi)>0$
    for all nonzero two-vectors $\varphi\in \bigwedge^2T_pM$ and all $p\in M$.\\
    \item [(ii)]   We say that $(M,g)$ is PIC if $R(z,w,\bar{z}, \bar{w}) > 0$
    for all linearly independent vectors $z,w\in T_pM\otimes \mathbb{C}$ such that $g(z,z)=g(w,w)=g(z,w)=0$.\\
    \item[(iii)]    We say that $(M,g)$ is PIC1 if $R(z,w,\bar{z},\bar{w})> 0$
    for all linearly independent vectors $z,w\in T_pM\otimes \mathbb{C}$ such that     $g(z,z)g(w,w)-g(z,w)^2=0$.\\
    \item[(iv)]    We say that $(M,g)$ is PIC2 if $R(z,w,\bar{z},\bar{w})> 0$
    for all linearly independent vectors $z,w\in T_pM\otimes \mathbb{C}$. That means $(M,g)$ has positive complex sectional curvature.
\end{itemize}	
\end{definition}

We now state the result for curvature preservation:

\begin{mtheorem}
Suppose that $g(t)$ is a canonical solution to the Ricci flow on manifold with boundary  on $M\times [0,T]$ given by the Main Theorem 3. Then the following holds:\\
 If $(M,g_0)$ has a convex boundary, then
\begin{itemize}
    \item[(i)] $(M,g_0)$ has positive curvature operator $\implies$ $(M,g(t))$ has positive curvature operator;
    \item[(ii)] $(M,g_0)$ is PIC1 $\implies$ $(M,g(t))$ is PIC1;
    \item[(iii)] $(M,g_0)$ is PIC2 $\implies$ $(M,g(t))$ is PIC2.
\end{itemize}
If $(M,g_0)$ has a two-convex boundary, then
\begin{itemize}
    \item[(iv)] $(M,g_0)$ is PIC $\implies$ $(M,g(t))$ is PIC.
\end{itemize}
Moreover, if $(M,g_0)$ has a mean-convex boundary, then
\begin{itemize}
    \item[(v)] $(M,g_0)$ has positive scalar curvature $\implies$ $(M,g(t))$ has positive scalar curvature.
\end{itemize}
\end{mtheorem}

It is worth noting that a similar result to the Main Theorem 5 was proved by Schlichting in his Ph.D. thesis \cite{Sch} which concerns with curvature preservation along the Ricci-DeTurck flow under convexity assumptions on the boundary. Nevertheless, our result allows the boundary to be two-convex for preserving positive isotropic curvature. The proof in Schlichting's thesis is different from ours, whereas in our proof of Main Theorem 5 additional properties of the approximated metrics constructed in \cite{Chow} are used.

Since the canonical solution we obtained preserves natural curvature conditions under suitable assumptions of the boundary data and can be extended smoothly to the doubled manifold, many results in Ricci flow on closed compact manifolds can be applied to our case. For examples, if the initial metric is PIC1 and has a convex boundary, our results and \cite{Bre08} then imply that the Ricci flow converges to a metric of constant curvature with a totally geodesic boundary after rescaling. From \cite{Bre19}, we can also obtain a topological classification of all compact manifolds with boundary of dimension $n\geq 12$ which admit metrics that are PIC and have two-convex boundary and do not contain non-trivial incompressible $(n-1)$-dimensional space forms.

The proof of the Main Theorems will occupy section 6. In section 3, we study a linear parabolic equation on vector bundles over $M$. This result will be applied to prove the existence and uniqueness of solutions to the Ricci-DeTurck flow and the harmonic map heat flow in section 4 and 5 using the Banach fixed point argument.

\vskip 0.5cm
\noindent\textbf{Acknowledgement:}
The author would like to express his gratitude to his advisor Professor Simon Brendle for his continuing support, his guidance and many inspiring discussions. The author would also like to thank the anonymous referee for his/her very careful reading and insightful comments, and for his/her valuable suggestions to improve the exposition of the previous version.

\bigskip
\section{Definition and Notation}
Throughout Sections 2, 3, 4, and 5, we assume, unless said otherwise, that $M$ is a closed compact manifold, $\pi:E\to M$ is a smooth vector bundle, and $\hat{g}$ a smooth metric on $M$. For sections $\eta, \zeta$ of $E$, the expression $ \eta*\zeta$ means a bilinear combination with respect to $\hat{g}$. We fix a time $T\leq 1$, which we will make small according to our needs. We fix a finite set of coordinate charts which simultaneously serve as trivializations of the vector bundle $\{U_s, \varphi_s, \tilde{\varphi}_s\}_{s=1,\dots,m}$ so that $\varphi_s: U_s\to \mathbb{R}^n$ is a diffeomorphism and $\tilde{\varphi}_s:\pi^{-1}(U_s)\to U_s\times\mathbb{R}^N$ is a trivialization. We may also assume the charts trivialize the vector bundle $\textup{Sym}^2(T^*M)$. Our trivializations give local frames $\bold{e}_r^s$ for $E$ and local frames $dx^idx^j$ for $\textup{Sym}^2(T^*M)$ on $U_s$ (note the coordinates $x_1,\dots, x_n$ depend on $s$). \\

We want to work on the parabolic H\"older spaces of tensors on $M$. Consider any tensor bundle of the form $E = T^{(p,q)}M$. If $S$ is a section of $E$, then in the open set $U_s$ it can be written as $S=\sum_{r=1}^NS^{r}_s\bold{e}_r^s$, where $\{\bold{e}_r^s\}$ is a local frame on $E|_{\pi^{-1}(U_s)}$. For any $\alpha\in (0,1)$, we define by $C^{ \alpha,\alpha/2}(M\times[t_1,t_2]; E)$ the space of maps $\eta: M\times [t_1, t_2]\to E$ such that $\eta(t)\in C^{ \alpha,\alpha/2}(M;E)$ is a $\alpha$-H\"older continuous section of $E$ for each $t\in [t_1, t_2]$. Given any map $\eta\in C^{ \alpha,\alpha/2}(M\times[t_1,t_2]; E)$,  we define the associated $C^0$ norm to be
\begin{align*}
	|\eta|_{0; M\times [t_1,t_2]}:=\sum_{s}\sum_{r=1}^N|\eta^{r}_s|_{0; \varphi_s(U_s)\times[t_1,t_2]},
\end{align*}
and we also define the associated parabolic H\"older semi-norm to be 
\begin{align*}
    [\eta]_{\alpha,\alpha/2; M\times [t_1,t_2]}:=\sum_{s}\sum_{r=1}^N[\eta^{r}_s]_{\alpha,\alpha/2; \varphi_s(U_s)\times[t_1,t_2]}.
\end{align*}
Subsequently, we define the parabolic H\"older norm on the space $C^{ \alpha,\alpha/2}(M\times[t_1,t_2]; E)$ to be
	\[ \|\eta\|_{\alpha,\alpha/2; M\times [t_1,t_2]} := |\eta|_{0; M\times [t_1,t_2]} +  [\eta]_{\alpha,\alpha/2; M\times [t_1,t_2]}.  \]
Next, for any nonnegative integer $k\geq 0$, the space of maps $C^{ k+\alpha,(k+\alpha)/2}(M\times[t_1,t_2]; E)$ can be defined similarly using the same charts $\{U_s\}$, where the derivatives are taken with respect to the fixed connection $\hat{\nabla}$. For instance, the space $C^{ k+\alpha,(k+\alpha)/2}(M\times[t_1,t_2]; E)$  consists of maps $\eta: M\times[t_1,t_2]\to E$ such that $\eta(t)\in C^k(M;E)$ is a $C^k$-differentiable section of $E$ for each $t$ with respect to $\hat{\nabla}$, and $\hat{\nabla}^k\eta(t)\in C^{ \alpha,\alpha/2}(M;(T^*M)^{\otimes k}\otimes E)$  is a $\alpha$-H\"older continuous section for each $t$. Then the associated norm can be defined similarly by 
	\[ \|\eta\|_{k+\alpha,\frac{k+\alpha}{2}; M\times [t_1,t_2]} := |\eta|_{0; M\times [t_1,t_2]} +\cdots + |\hat{\nabla}^k\eta|_{0; M\times [t_1,t_2]} +  [\hat{\nabla}^k\eta]_{\alpha,\alpha/2; M\times [t_1,t_2]}.  \]
 The elliptic H\"older space $C^{ k+\alpha}(M\times[t_1,t_2]; E)$ can also be defined similarly. Note that the H\"older norms defined this way are equivalent in different atlas. \\

We wish to work on weighted parabolic H\"older spaces in order to prove existence results. In the sequel, we define weighted norms on $M\times (0,T]$. Given a nonnegative integer $k\geq 0$, a H\"older exponent $\alpha\in (0,1)$ and a positive real number $\gamma$, we define:
\begin{align}
    \|\eta\|_{\mathcal{C}^{k,\alpha}_{\gamma}(M\times(0,T])}
    :=\sum_{i=0}^k\sup_{\sigma\in (0,T]}\sigma^{\gamma+\frac{i}{2}}|\hat{\nabla}^i\eta|_{0; M\times [\frac{\sigma}{2},\sigma]} + \sum_{i=0}^k\sup_{\sigma\in (0,T]}\sigma^{\gamma+\frac{\alpha}{2}+\frac{i}{2}}[\hat{\nabla}^i\eta]_{\alpha,\alpha/2; M\times [\frac{\sigma}{2},\sigma]},
\end{align}
Throughout the note, we will work on the Banach spaces that we define below. 

\begin{definition}
Let $M$ be a smooth closed compact manifold, $E$ a vector bundle on $M$, and $\hat{g}$ a smooth background metric on $M$. Given a nonnegative integer $k\geq 0$, a H\"older exponent $\alpha\in (0,1)$ and a positive real number $\gamma$, we define a weighted parabolic H\"older space on the space of  sections by
\begin{align}
    &\mathcal{C}^{k,\alpha}_{\gamma}(M\times(0,T];E)\\
    &\notag := \{\eta: M\times (0,T]\to E|\ \eta(t)\in C^k(M;E)\ \text{for each}\ t\in (0,T],\ \|\eta\|_{\mathcal{C}^{k,\alpha}_{\gamma}(M\times(0,T])}
<\infty\}.
\end{align}
\end{definition}

We now state some easy consequences from the definition of these spaces:

\begin{lemma}{\label{properties of weighted space}}\
\begin{itemize}
	\item[(1)] $\eta\in \mathcal{C}^{k,\alpha}_{\gamma}(M\times (0,T])$ implies $\hat{\nabla}^j\eta\in \mathcal{C}^{k-j,\alpha}_{\gamma+\frac{j}{2}}(M\times (0,T])$ for $j\leq k$;\\
	\item[(2)] For any $\delta > 0\ $, $\|\eta\|_{\mathcal{C}^{k,\alpha}_{\gamma+\delta}(M\times (0,T])}\leq T^{\delta}\|\eta\|_{\mathcal{C}^{k,\alpha}_{\gamma}(M\times (0,T])}$. In particular, this implies that
		\[\mathcal{C}^{k,\alpha}_{\gamma}(M\times (0,T])\subset \mathcal{C}^{k,\alpha}_{\gamma+\delta}(M\times (0,T]). \]
	\item [(3)] Suppose that $\eta\in \mathcal{C}^{k,\alpha}_{\gamma}(M\times (0,T])$, then $\eta\in \mathcal{C}^{k,\beta}_{\gamma}(M\times (0,T])$ for any $\beta<\alpha$;\\

	\item [(4)]Suppose that $\eta\in \mathcal{C}^{k,\alpha}_{\gamma}(M\times (0,T])$ and $\zeta\in \mathcal{C}^{k,\alpha}_{\delta}(M\times (0,T])$. Then
	$$ \eta*\zeta \in \mathcal{C}^{k,\alpha}_{\gamma+\delta}(M\times (0,T]), $$
	where $ \eta*\zeta$ means a bilinear combination with respect to $\hat{g}$.
	Moreover,
	$$ \|\eta*\zeta\|_{\mathcal{C}^{k,\alpha}_{\gamma + \delta}(M\times (0,T])}\leq K(\hat{g})\ \|\eta\|_{\mathcal{C}^{k,\alpha}_{\gamma}(M\times (0,T])}\|\zeta\|_{\mathcal{C}^{k,\alpha}_{\delta}(M\times (0,T])}. $$
	
\end{itemize}
\end{lemma}

\begin{proof} 
	Statements (1) and (2) follow from the definition. Statement (3) follows from the properties of parabolic H\"older spaces. Now we prove statement (4). 
	Fix $\sigma\in (0,T]$. For any $j\leq k$, we have
	$$ \hat{\nabla}^j(\eta*\zeta) = \sum_{j_1+j_2=j}\hat{\nabla}^{j_1}\eta*\hat{\nabla}^{j_2}\zeta.	$$
	Let us denote $M_{\sigma} = M\times [\frac{\sigma}{2},\sigma]$ for simplicity, we have
	\begin{align*}
		\sigma^{\gamma+\delta+\frac{j}{2}}|	\hat{\nabla}^j(\eta*\zeta) |_{0;M_{\sigma}} &\leq K(\hat{g})\sum_{j_1+j_2=j}\sigma^{\gamma+\frac{j_1}{2}}|	\hat{\nabla}^{j_1}\eta |_{0;M_{\sigma}}\sigma^{\delta+\frac{j_2}{2}}|	\hat{\nabla}^{j_2}\zeta |_{0;M_{\sigma}}\\
		&\leq K(\hat{g})\ \|\eta\|_{\mathcal{C}^{k,\alpha}_{\gamma}(M\times (0,T])}\|\zeta\|_{\mathcal{C}^{k,\alpha}_{\delta}(M\times (0,T])}
	\end{align*}
	and
	\begin{align*}
		&\sigma^{\gamma+\delta+\frac{\alpha}{2}+\frac{j}{2}}[\hat{\nabla}^j(\eta*\zeta) ]_{\alpha,\frac{\alpha}{2};M_{\sigma}} \\
		&\leq K(\hat{g})\sum_{j_1+j_2=j}\Big(\sigma^{\gamma+\frac{j_1}{2}}|	\hat{\nabla}^{j_1}\eta |_{0;M_{\sigma}}\sigma^{\delta+\frac{\alpha}{2}+\frac{j_2}{2}}[	\hat{\nabla}^{j_2}\zeta ]_{\alpha,\frac{\alpha}{2};M_{\sigma}} + \sigma^{\gamma+\frac{\alpha}{2}+\frac{j_1}{2}}[	\hat{\nabla}^{j_1}\eta ]_{\alpha,\frac{\alpha}{2};M_{\sigma}}\sigma^{\delta+\frac{j_2}{2}}|	\hat{\nabla}^{j_2}\zeta |_{0;M_{\sigma}}\Big)\\
		&\leq K(\hat{g})\ \|\eta\|_{\mathcal{C}^{k,\alpha}_{\gamma}(M\times (0,T])}\|\zeta\|_{\mathcal{C}^{k,\alpha}_{\delta}(M\times (0,T])}.
	\end{align*}
	From these, statement (4) follows.	
\end{proof}

\newpage

\section{Linear Parabolic Equation with $C^{\alpha}$ Initial Data}
Although our goal is to solve a boundary-value problem for PDE on a manifold with boundary, it is equivalent to work with PDE with rough initial data on a closed manifold. Fix a real number $I\in (0,1)$. Let $w(x,t), t\in [0,I]$ be a continuous family of Riemannian metrics on $M$. Given a section $\eta_0\in C^{\alpha}(M; E)$, we consider the following parabolic system on vector bundle:

\begin{align}
    \begin{cases}
    \frac{\partial}{\partial t}\eta(x,t) - tr_{w}\hat{\nabla}^2\eta(x,t) = F(x,t) \quad &\text{on}\quad M\times (0,T]\\
    \eta(x,0) = \eta_0(x) \quad &\text{on}\quad M,
    \end{cases}
\end{align}
where $T\leq I$ and $F\in \Gamma(M\times (0,T]; E),\ w\in \Gamma(M\times [0,T]; \textup{Sym}^2(T^*M))$. Here $\hat{\nabla}$ is the Levi-Civita connection with respect to the background metric $\hat{g}$. Our goal is to prove solvability of (3.1). To do that, we need the uniform parabolicity assumption on $w$: there is a $\lambda>0$ such that
\begin{align}
    \lambda|\xi|_{\hat{g}}^2(x)\geq w^{kl}(x,t)\xi_k(x)\xi_l(x)\geq\frac{1}{\lambda}|\xi|_{\hat{g}}^2(x)
\end{align}
for any $(x,t)\in M\times [0,I]$ and any $\xi\in \Gamma(TM)$.

We now state the main result of this section. It will be utilized to study the existence of Ricci flow on manifold with boundary in the next section.\\

\begin{theorem}{\label{theorem linear equation}}
Let $\alpha, \gamma\in (0,1)$ be given such that $\alpha >\gamma$.  Let $k\geq 0$ be an non-negative integer. Suppose that 
\begin{itemize}
	\item[(1)] $\eta_0\in C^{\alpha}(M;E)$;
	\item[(2)] $w$ satisfies the uniform parabolicity condition (3.2);
    \item[(3)] $\|w\|_{\gamma,\frac{\gamma}{2}; M\times [0,I]}  + \|\hat{\nabla}w\|_{\mathcal{C}^{k-1,\gamma}_{\frac{1}{2}}(M\times (0,I])} \leq A$ when $k\geq 1$; \\ or $\|w\|_{\gamma,\frac{\gamma}{2}; M\times [0,I]}   \leq A$ when $k = 0$.
\end{itemize}
Then there exits a positive constant $K = K(M, k,  \hat{g}, A)$ such that the following holds:

For each $T\leq I$, if $F\in\mathcal{C}^{k,\gamma}_{1-\frac{\alpha}{2}}(M\times (0,T];E)$, then there is an unique solution $\eta$ to the system (3.1) such that 
	$$\eta\in C^{\alpha,\frac{\alpha}{2}}(M\times [0,T];E),\ \hat{\nabla}\eta\in \mathcal{C}^{k+1,\gamma}_{\frac{1}{2}-\frac{\alpha}{2}}(M\times (0,T];E).$$
Moreover, $\eta$ satisfies the estimate
\begin{align*}
	&\|\eta\|_{\alpha,\frac{\alpha}{2}; M\times [0,T]}  + \|\hat{\nabla}\eta\|_{\mathcal{C}^{k+1,\gamma}_{\frac{1}{2}-\frac{\alpha}{2}}(M\times (0,T])} \leq K	(\|F\|_{\mathcal{C}^{k,\gamma}_{1-\frac{\alpha}{2}}(M\times (0,T])} + \|\eta_0\|_{\alpha; M}).
\end{align*}
\end{theorem}

We will prove this theorem in the remainder of the section.

\bigskip

\subsection{Formulation of the Proof}\

\

For every $T\in (0,I]$, we define the Banach space 
	$$\mathcal{W}_k(M\times (0,T];E):=\mathcal{C}^{k,\gamma}_{1-\frac{\alpha}{2}}(M\times (0,T];E)\times C^{\alpha}(M;E) $$ 
whose elements are the pairs of sections $h=(F,\eta_0)$, where $F\in\mathcal{C}^{k,\gamma}_{1-\frac{\alpha}{2}}(M\times (0,T];E)$ and $\eta_0\in C^{\alpha}(M;E)$. We equip $\mathcal{W}_k$ with the norm
$$||h||_{\mathcal{W}_k}:= ||F||_{\mathcal{C}^{k,\gamma}_{1-\frac{\alpha}{2}}(M\times(0,T])}+||\eta_0||_{\alpha;M}.$$
Moreover, we define the norm $\|\cdot\|_{\mathcal{X}_k(M\times[0,T];E)}$ by
$$||\eta||_{\mathcal{X}_k}:= ||\eta||_{\alpha,\frac{\alpha}{2};M\times [0,T]}  + ||\hat{\nabla}\eta||_{\mathcal{C}^{k-1,\gamma}_{\frac{1}{2}-\frac{\alpha}{2}}(M\times (0,T])}.$$
We define the associated Banach space $\mathcal{X}_k(M\times [0,T];E)$ by
\begin{align*}
    \mathcal{X}_k(M\times [0,T];E) := \{\eta: M\times [0,T]\to E |\ \|\eta\|_{\mathcal{X}_k} <\infty \}.
\end{align*}
Subsequently $\mathcal{X}_{k+2}$ will serve as the solution space. We basically adapt the method in Chapter IV of \cite{Lad} to the case of vector bundles. The idea of the proof is as follows: \\

Let $H:\mathcal{X}_{k+2}\to\mathcal{W}_k$ to be the linear operator that associates any $\eta\in\mathcal{X}_{k+2}$ to 
$$H\eta = (L\eta, \eta(\cdot,0)),$$
where $L\eta = \frac{\partial}{\partial t}\eta - tr_{w}\hat{\nabla}^2\eta$. Then Theorem \ref{theorem linear equation} can be interpreted to the solvability of 
$$H\eta=h$$
for any $h\in\mathcal{W}_k(M\times (0,T];E)$. It is equivalent to prove the existence of a bounded inverse operator $H^{-1}$. The key is to construct an operator $R:\mathcal{W}_k\to\mathcal{X}_{k+2}$ which satisfies
\begin{align*}
    \begin{cases}
    &HRh=h+Sh\\
    &RH\eta=\eta+G\eta
    \end{cases}
\end{align*}
for some bounded operators $S:\mathcal{W}_k\to\mathcal{W}_k$ and $G:\mathcal{X}_{k+2}\to\mathcal{X}_{k+2}$. If their norms can be controlled such that $||S||, ||G||<1$, then it follows from an elementary argument that $H^{-1}$ exists. 

\bigskip

\subsection{Construction of an approximated solution}\
\newline

Let $h=(F,\eta_0)\in\mathcal{W}_k$ be given. To construct the operator $R$, we consider a system of PDE on each chart $U_s$. As in the beginning of section 2, we let $\tilde{\varphi_s}:\pi^{-1}(U_s)\to U_s\times\mathbb{R}^N$ be the local trivialization of $E$ on the open set $U_s$, and let $\{\bold{e}_r^s\}_{r=1,..,N}$ be the canonical local frame of $\pi^{-1}(U_s)$ with respect to the trivialization. Then $F$ can be written as $F=\sum_{r=1}^NF^r_s\bold{e}_r^s$, similarly for $\eta$, where we abbreviate $F^r_s(x,t)=F^r_s\circ(\varphi_s^{-1}(x),t)$ for $x\in\varphi_s(U_s)\cong\mathbb{R}^n$.  On each chart $U_s$, we consider the following parabolic system: 

\begin{align}{\label{scalar parabolic system}}
    \begin{cases}
    \frac{\partial}{\partial t}\eta^{r}_s(x,t)-w^{kl}(x,t)\frac{\partial^2}{\partial x_k\partial x_l}\eta^{r}_s(x,t)= F^r_s(x,t) &\quad\text{on}\quad \mathbb{R}^n\times(0,T], \quad r=1,..,N\\
    \eta^{r}_s(x,0)= (\eta_0)^r_s(x) &\quad\text{on}\quad\mathbb{R}^n,  \quad r=1,..,N.
    \end{cases}
\end{align}

We see that the system is equivalent to $N$ uncoupled scalar equations, one for each $\eta^{r}_s$. We will prove the existences and the uniqueness for the uncoupled problems in this subsection. We begin with an auxiliary lemma for linear parabolic PDEs.

\bigskip
\begin{lemma}{\label{lemma second order u}}
Let  $\alpha, \gamma\in (0,1)$ be given such that $\alpha >\gamma$. Suppose that 
\begin{itemize}
    \item[(1)]$a_{ij}(x,t)\in C^{\gamma,\frac{\gamma}{2}}(\mathbb{R}^n\times [0,T])$ and satisfies the uniform parabolicity condition. i.e. there is $\lambda>0$ such that $\frac{1}{\lambda}\delta_{ij}<a_{ij}(x,t)<\lambda\delta_{ij}$;
    \item[(2)]$\|f\|_{\mathcal{C}^{0,\gamma}_{1-\frac{\alpha}{2}}(\mathbb{R}^n\times (0,T])}<\infty$ and $\|u_0\|_{\alpha;\mathbb{R}^n}<\infty$.
\end{itemize} Then the initial-value problem
\begin{align}{\label{model pde}}
    \begin{cases}
    \frac{\partial}{\partial t}u(x,t)-a_{kl}(x,t)\frac{\partial^2}{\partial x_k\partial x_l}u(x,t)=f(x,t) &\quad\text{on}\quad \mathbb{R}^n\times(0,T]\\
    u(x,0)= u_0(x) &\quad\text{on}\quad\mathbb{R}^n
    \end{cases}
\end{align}
has a unique solution $u$, where $u\in C^{\alpha,\frac{\alpha}{2}}(\mathbb{R}^n\times [0,T])$ and $Du\in \mathcal{C}^{1,\gamma}_{\frac{1}{2}-\frac{\alpha}{2}}(\mathbb{R}^n\times (0,T])$. Moreover, $u$ satisfies the estimate
\begin{align}
    \|u\|_{\alpha,\frac{\alpha}{2};\mathbb{R}^n\times [0,T]} + \|D_xu\|_{\mathcal{C}^{1,\gamma}_{\frac{1}{2}-\frac{\alpha}{2}}(\mathbb{R}^n\times (0,T])}
    &\leq K(\|f\|_{\mathcal{C}^{0,\gamma}_{1-\frac{\alpha}{2}}(\mathbb{R}^n\times (0,T])} +\|u_0\|_{\alpha;\mathbb{R}^n}).
\end{align}
Here $K$ is a constant depending only on $\lambda, \|a_{ij}\|_{\gamma,\frac{\gamma}{2}}$.

\end{lemma}
\bigskip

\begin{proof}
In the sequel of the proof, $K$ will denote a constant depending only on $\lambda, \|a_{ij}\|_{\gamma,\frac{\gamma}{2}}, \alpha, \gamma$ unless otherwise specified. 
\begin{itemize}
\item[\underline{Step1:}]

We use the single layer potential method to construct a unique solution $u$ to (\ref{model pde}). Given $(\xi,\tau)\in\mathbb{R}^n\times [0,T]$ ,let $\Gamma(x,t;\xi,\tau)$ to be the fundamental solution to 
\begin{align*}
\frac{\partial}{\partial t}u-a_{kl}(x,t)\frac{\partial^2}{\partial x_k\partial x_l}u=0	
\end{align*}
on $\mathbb{R}^n\times (\tau,T]$ such that $\Gamma(x,t;\xi,\tau)\to \delta(x-\xi)$ as $t\to\tau$ in the sense of distribution. We claim that the formula 
 \begin{equation}{\label{formula of u}}
    u(x,t)=- \int_0^t\int_{\mathbb{R}^n}\Gamma (x,t;\xi,\tau)f (\xi,\tau)d\xi d\tau + \int_{\mathbb{R}^n}\Gamma (x,t;\xi,0)u_0(\xi)d\xi 
\end{equation}
gives a solution to the system (\ref{model pde}), and uniqueness would follow from the maximum principle. It is well known from \cite{Fried} that the formula (\ref{formula of u}) gives a unique solution to (\ref{model pde}) provided that $f$ in a H\"older-continuous function on $\mathbb{R}^n\times [0,T]$. In our case, the assumption $\|f\|_{\mathcal{C}^{0,\gamma}_{1-\frac{\alpha}{2}}(\mathbb{R}^n\times(0,T])}<\infty$ implies that $f$ is H\"older-continuous on $\mathbb{R}^n\times [\sigma/2,\sigma]$ for any $\sigma\in (0,T]$, and that
	\[ \tau^{1-\frac{\alpha}{2}}|f|_{0;\mathbb{R}^n\times [\tau/2, \tau]} + \tau^{1-\frac{\alpha}{2}+\frac{\gamma}{2}}[f]_{\gamma,\frac{\gamma}{2}; \mathbb{R}^n\times [\tau/2,\tau]} \leq K\]
	 for each $\tau\in (0, T]$. Hence 
	 \begin{align*}
	 	\|f(\cdot, \tau)\|_{\gamma;\mathbb{R}^n} &= |f(\cdot, \tau)|_{0;\mathbb{R}^n} + [f(\cdot, \tau)]_{\gamma;\mathbb{R}^n}\\
	 	&\leq |f|_{0;\mathbb{R}^n\times [\tau/2, \tau]} + [f]_{\gamma,\frac{\gamma}{2}; \mathbb{R}^n\times [\tau/2,\tau]}\\
	 	&\leq K\tau^{\frac{\alpha}{2}-1} + K\tau^{\frac{\alpha}{2}-1-\frac{\gamma}{2}}\\
	 	&\leq K\tau^{\frac{\alpha}{2}-1-\frac{\gamma}{2}}
	 \end{align*}
	for each $\tau\in (0,T]$. In particular, the elliptic H\"older bound for $f(\cdot, \tau)$ is integrable over $\tau\in (0,t)$ for any $t\in (0,T]$. This implies that the proof of Theorem 9 in Chapter 1 of \cite{Fried} still works so that (\ref{formula of u}) satisfies the evolution equation in (\ref{model pde}) for every $t>0$. To show that (\ref{formula of u}) gives the correct initial condition, it suffices to show that the first integral in RHS of (\ref{formula of u}) tends to zero as $t\to 0$. 

Let us write (\ref{formula of u}) as $u:=u_1+u_2$, where $u_1$ stands for the first term in RHS of (\ref{formula of u}), and $u_2$ stands for the second term in RHS of (\ref{formula of u}). We recall the estimates for the fundamental solution 
\begin{align}{\label{fundamental solution estimate}}
    \Big|D_t^rD_x^s\Gamma(x,t;\xi,\tau)\Big|\leq K(t-\tau)^{-\frac{n+2r+s}{2}}\exp\left(-\frac{|x-\xi|^2}{K(t-\tau)}\right)\quad ,2r+s\leq 2
\end{align}
given in page 376 of \cite{Lad}. For any $(x,t)\in \mathbb{R}^n\times [0,T]$, we estimate
\begin{align}{\label{u_1 C^0 estimate}}
&|u_1(x,t)|\\
&\notag\leq \int_0^t\int_{\mathbb{R}^n}|\Gamma(x,t;\xi,\tau)f(\xi,\tau)|d\xi d\tau	 \\
&\notag\leq K\int_0^t\int_{\mathbb{R}^n}(t-\tau)^{-\frac{n}{2}}\exp\left(-\frac{|x-\xi|^2}{K(t-\tau)}\right)|f(\xi,\tau)|\ d\xi d\tau\\
&\notag \leq K\int_{0}^{t}\int_{\mathbb{R}^n}(t-\tau)^{-\frac{n}{2}}\exp\left(-\frac{|x-\xi|^2}{K(t-\tau)}\right)\tau^{-1+\frac{\alpha}{2}}\|f\|_{\mathcal{C}^{0,\gamma}_{1-\frac{\alpha}{2}}(M\times (0,T])}\ d\xi d\tau\\
&\notag \leq K\int_{0}^{t}\int_0^{\infty}\rho^{n-1}\exp(-\frac{1}{K}\rho^2)\tau^{-1+\frac{\alpha}{2}}\|f\|_{\mathcal{C}^{0,\gamma}_{1-\frac{\alpha}{2}}(M\times (0,T])}\ d\rho d\tau\\
&\notag  \leq Kt^{\frac{\alpha}{2}}\|f\|_{\mathcal{C}^{0,\gamma}_{1-\frac{\alpha}{2}}(M\times (0,T])}.
\end{align}
Since $u_2(x,t)$ converges to $u_0(x)$ as $t\to 0$, this shows that $u(x,t)$ converges to $u_0(x)$ as $t\to 0$, and that $u$ is continuous on $\mathbb{R}^n\times [0,T]$. For the integral $u_2$, it is easy to see that
\begin{align}
	|u_2(x,t)|\leq K\sup_{\mathbb{R}^n}|u_0|	
\end{align}
for any $(x,t)\in\mathbb{R}^n\times [0,T]$. Thus we have obtained a $C^0$ estimate for $u(x,t)$:
\begin{align}{\label{C^0 estimate}}
	\|u\|_{0;\mathbb{R}^n\times [0,T]} \leq K( T^{\frac{\alpha}{2}}\|f\|_{\mathcal{C}^{0,\gamma}_{1-\frac{\alpha}{2}}(M\times (0,T])}	+ \|u_0\|_{0;\mathbb{R}^n})
\end{align}

\bigskip
\item[\underline{Step2}:]

In this step, we are going to derive
\begin{align}{\label{C^a estimate}}
    \|u\|_{\alpha,\frac{\alpha}{2};\mathbb{R}^n\times [0,T]}
    &\leq K( \|f\|_{\mathcal{C}^{0,\gamma}_{1-\frac{\alpha}{2}}(M\times (0,T])} + \|u_0\|_{\alpha;\mathbb{R}^n}).
\end{align}
We first bound the H\"older semi-norm for $u_1$. That is, we seek the following inequalities for any $x,y\in\mathbb{R}^n$ and $s,t\in [0,T]$:
\begin{align}{\label{u_1 C^a estimate}}
    \begin{cases}
    &   |u_1(x,t)-u_1(y,t)|\leq K|x-y|^{\alpha}\|f\|_{\mathcal{C}^{0,\gamma}_{1-\frac{\alpha}{2}}(M\times (0,T])}\\
    &|u_1(x,t)-u_1(x,s)|\leq K|t-s|^{\frac{\alpha}{2}}\|f\|_{\mathcal{C}^{0,\gamma}_{1-\frac{\alpha}{2}}(M\times (0,T])}
    \end{cases}
\end{align}
To derive the first inequality in (\ref{u_1 C^a estimate}), we divide $\mathbb{R}^n$ into $A_1=\{\xi\in \mathbb{R}^n: |x-\xi|>2|x-y|\}$ and $A_2=\mathbb{R}^n-A_1$. Using the estimates for fundamental solutions (\ref{fundamental solution estimate}), we obtain
\begin{align}{\label{u_1 C^a estimate 1}}
    &|u_1(x,t)-u_1(y,t)| \\
    &\leq \int_0^t\int_{A_1}\sup_{z\in\overline{xy}}|D_x\Gamma(z,t;\xi,\tau)||x-y||f(\xi,\tau)|d\xi d\tau + \int_0^t\int_{A_2}(|\Gamma(x,t;\xi,\tau)|+|\Gamma(y,t; \xi,\tau)|)|f(\xi,\tau)|d\xi d\tau\notag \\
    &\leq K|x-y|\int_0^t\int_{\{|x-\xi|>2|x-y|\}}(t-\tau)^{-\frac{n+1}{2}}\sup_{z\in\overline{xy}}\exp\left(-\frac{|z-\xi|^2}{K(t-\tau)}\right)|f(\xi,\tau)|d\xi d\tau  \notag\\
    &\quad + K\int_0^t\int_{\{|x-\xi|<2|x-y|\}}(t-\tau)^{-\frac{n}{2}}\left(\exp\left(-\frac{|x-\xi|^2}{K(t-\tau)}\right)+\exp\left(-\frac{|y-\xi|^2}{K(t-\tau)}\right)\right)|f(\xi,\tau)|d\xi d\tau\notag. 
\end{align}
 Note that if $\xi\in A_1$ and $z$ is a point on the segment $\overline{xy}$, then $|z-\xi|>\frac{1}{2}|x-\xi|$. Thus
 $$\sup_{z\in\overline{xy}}\exp\left(-\frac{|z-\xi|^2}{K(t-\tau)}\right) \leq \exp\left(-\frac{|x-\xi|^2}{4K(t-\tau)}\right).$$
  We observe that for any $m\geq \alpha$, we have
 \begin{align}{\label{fundamental solution estimate 2}}
     (t-\tau)^{-\frac{m}{2}}\exp\left(-\frac{|x-\xi|^2}{K(t-\tau)}\right) &=  \frac{1}{|t-\tau|^{\frac{\alpha}{2}}|x-\xi|^{m-\alpha}}\frac{|x-\xi|^{m-\alpha}}{|t-\tau|^{\frac{m-\alpha}{2}}}\exp\left(-\frac{|x-\xi|^2}{K(t-\tau)}\right) \\
     &\leq \frac{K}{|t-\tau|^{\frac{\alpha}{2}}|x-\xi|^{m-\alpha}}.\notag
 \end{align}
 
Using the above estimate with $m=n+1$ for the first term and $m=n$ for the second term respectively in the last inequality in (\ref{u_1 C^a estimate 1}), we proceed as in (\ref{u_1 C^0 estimate}) to obtain that
\begin{align}
    &|u_1(x,t)-u_1(y,t)|\\
    &\leq K|x-y|\int_0^t\int_{\{|x-\xi|>2|x-y|\}}\frac{1}{|t-\tau|^{\frac{\alpha}{2}}|x-\xi|^{n+1-\alpha}}|f(\xi,\tau)|d\xi d\tau  \notag\\
    &\quad\notag + K\int_0^t\Big(\int_{\{|x-\xi|<2|x-y|\}}\frac{1}{|t-\tau|^{\frac{\alpha}{2}}|x-\xi|^{n-\alpha}}|f(\xi,\tau)|d\xi\\
    &\notag\quad\quad\quad\quad\quad\quad + \int_{\{|y-\xi|<3|x-y|\}}\frac{1}{|t-\tau|^{\frac{\alpha}{2}}|y-\xi|^{n-\alpha}}|f(\xi,\tau)|d\xi\Big) d\tau \notag \\
    &\notag\leq K|x-y|^{\alpha}\|f\|_{\mathcal{C}^{0,\gamma}_{1-\frac{\alpha}{2}}(M\times (0,T])},
\end{align}
where we have used the inequality
	\[\int_0^t\frac{\tau^{\frac{\alpha}{2}-1}}{(t-\tau)^{\frac{\alpha}{2}}}d\tau \leq K(\alpha)\]
	in deriving the last line. 
Next, to derive the second inequality in (\ref{u_1 C^a estimate}), we assume without loss of generality that $s<t$. and we divide $\mathbb{R}^n$ into $A_3=\{\xi\in \mathbb{R}^n: |x-\xi|>\sqrt{t-s}\}$ and $A_4=\mathbb{R}^n-A_3$. Then
\begin{align}{\label{u_1 C^a estimate 2}}
    &|u_1(x,t)-u_1(x,s)| \\
    &\leq \int_0^s\int_{A_3}\sup_{\bar{t}\in [s,t]}|D_t\Gamma(x,\bar{t};\xi,\tau)||t-s||f(\xi,\tau)|d\xi d\tau + \int_s^t\int_{A_3}|\Gamma(x,t;\xi,\tau)||f(\xi,\tau)|d\xi d\tau\notag \\
    &\quad +\int_0^t\int_{A_4}|\Gamma(x,t;\xi,\tau)||f(\xi,\tau)|d\xi d\tau +\int_0^s\int_{A_4}|\Gamma(x,s;\xi,\tau)||f(\xi,\tau)|d\xi d\tau \notag
\end{align}
For instance, using (\ref{fundamental solution estimate}) and (\ref{fundamental solution estimate 2}) with $m = n+1$, we can estimate the second term in the RHS of the above inequality
\begin{align*}
		&\int_s^t\int_{A_3}|\Gamma(x,t;\xi,\tau)||f(\xi,\tau)|d\xi d\tau \\
		& \leq K\int_s^t(t-\tau)^{\frac{1}{2}}\int_{A_3}(t-\tau)^{-\frac{n+1}{2}}\exp\left(-\frac{|x-\xi|^2}{K(t-\tau)}\right)|f(\xi, \tau)| d\xi d\tau\\
		&\leq  K\int_s^t(t-s)^{\frac{1}{2}}\int_{|x-\xi|>\sqrt{t-s}}\frac{1}{|t-\tau|^{\frac{\alpha}{2}}|x-\xi|^{n+1-\alpha}}|f(\xi, \tau)| d\xi d\tau\\
		&\leq K|t-s|^{\frac{\alpha}{2}}\|f\|_{\mathcal{C}^{0,\gamma}_{1-\frac{\alpha}{2}}(M\times (0,T])}\int_0^t\frac{\tau^{\frac{\alpha}{2}-1}}{|t-\tau|^{\frac{\alpha}{2}}} d\tau.
\end{align*}
We can estimate the other terms in the RHS of (\ref{u_1 C^a estimate 2}) similarly. 
Hence we obtain
\begin{align*}
	 &|u_1(x,t)-u_1(x,s)| \\
	 &\leq K|t-s|^{\frac{\alpha}{2}}\|f\|_{\mathcal{C}^{0,\gamma}_{1-\frac{\alpha}{2}}(M\times (0,T])}\left(\int_0^t\frac{\tau^{\frac{\alpha}{2}-1}}{|t-\tau|^{\frac{\alpha}{2}}} d\tau + \int_0^{s}\frac{\tau^{\frac{\alpha}{2}-1}}{|s-\tau|^{\frac{\alpha}{2}}} d\tau \right)\notag\\
    &\leq K|t-s|^{\frac{\alpha}{2}}\|f\|_{\mathcal{C}^{0,\gamma}_{1-\frac{\alpha}{2}}(M\times (0,T])}.
\end{align*}
From which (\ref{u_1 C^a estimate}) follows.

\bigskip

In the remainder of this step, we derive the following estimate for the H\"older semi-norm of $u_2$:
\begin{align*}
	[u_2]_{\alpha,\frac{\alpha}{2};\mathbb{R}^n\times [0,T]}\leq K\ [u_0]_{\alpha;\mathbb{R}^n}.
\end{align*}
That is, for any $x,y\in\mathbb{R}^n$ and $s,t\in [0,T]$, we claim that:
\begin{align}{\label{u_2 C^a estimate}}
    \begin{cases}
    &   |u_2(x,t)-u_2(y,t)|\leq K|x-y|^{\alpha}[u_0]_{\alpha;\mathbb{R}^n}\\
    &|u_2(x,t)-u_2(x,s)|\leq K|t-s|^{\frac{\alpha}{2}}[u_0]_{\alpha;\mathbb{R}^n}
    \end{cases}.
\end{align}
To derive the first inequality, we break $u_2(x,t)-u_2(y,t)$ into two integrals as follows:
\begin{align*}
 u_2(x,t) - u_2(y,t) &= \int_{\mathbb{R}^n}(\Gamma (x,t;\xi,0) - \Gamma(y,t;\xi,0)) (u_0(\xi) - u_0(x))d\xi \\
 &\quad  + u_0(x)\int_{\mathbb{R}^n}(\Gamma (x,t;\xi,0) - \Gamma(y,t;\xi,0)) d\xi	\\
 &:= J_1 + J_2.
\end{align*}
For the integral $J_1$, we use (\ref{fundamental solution estimate}) to estimate
\begin{align}{\label{J_1}}
	|J_1| &\leq \int_{A_1}\sup_{z\in\overline{xy}}|D_x\Gamma(z,t;\xi,0)||x-y||u_0(x)-u_0(\xi)|d\xi \\
	&\quad + \int_{A_2}(|\Gamma(x,t;\xi,0)|+|\Gamma(y,t;\xi,0)|)|u_0(x)-u_0(\xi)|d\xi \notag \\
    &\leq K|x-y|\int_{\{|x-\xi|>2|x-y|\}}(t-\tau)^{-\frac{n+1}{2}}\exp\left(-\frac{|x-\xi|^2}{4K(t-\tau)}\right)|x-\xi|^{\alpha}[u_0]_{\alpha;\Omega}\ d\xi   \notag\\
    &\quad + K\int_{\{|x-\xi|<2|x-y|\}}(t-\tau)^{-\frac{n}{2}}\left(\exp\left(-\frac{|x-\xi|^2}{K(t-\tau)}\right)+\exp\left(-\frac{|y-\xi|^2}{K(t-\tau)}\right)\right)|x-\xi|^{\alpha}[u_0]_{\alpha;\Omega}\ d\xi. \notag 
\end{align}
Similar to the derivation in (3.13) to (3.15), we apply (\ref{fundamental solution estimate 2}) with $\alpha = 0$ to above, and we subsequently obtain
	\[|J_1|\leq K|x-y|^{\alpha}[u_0]_{\alpha;\mathbb{R}^n}.\]
On the other hand, we see that the second integral $J_2$ vanishes since the fundamental solution satisfies $\int_{\Omega}\Gamma(x,t;\xi,0)d\xi = \int_{\Omega}\Gamma(y,t;\xi,0)d\xi = 1$.\\

 Next, we preform the same procedure on the term $u_2(x,t) - u_2(x,s)$. We break it into two integrals:
\begin{align*}
 u_2(x,t) - u_2(x,s) &= \int_{\mathbb{R}^n}(\Gamma (x,t;\xi,0) - \Gamma(x,s;\xi,0)) (u_0(\xi) - u_0(x))d\xi \\
 &\quad  + u_0(x)\int_{\mathbb{R}^n}(\Gamma (x,t;\xi,0) - \Gamma(x,s;\xi,0)) d\xi	\\
 &:= J_3 + J_4.
\end{align*}
Again, the integral $J_4$ vanishes. Using a similar argument as in (\ref{u_1 C^a estimate 2}) and (\ref{J_1}), we have
\begin{align}
	|J_3| &\leq \int_{A_3}\sup_{\bar{t}\in [s,t]}|D_t\Gamma(x,\bar{t};\xi,0)||t-s||u_0(x)-u_0(\xi)|d\xi \\
	&\quad + \int_{A_4}(|\Gamma(x,t;\xi,0)|+|\Gamma(x,s;\xi,0)|)|u_0(x)-u_0(\xi)|d\xi \notag \\
    &\leq K|t-s|^{\frac{\alpha}{2}}\ [u_0]_{\alpha;\Omega}. \notag 
\end{align}
From this, (\ref{u_2 C^a estimate}) follows.

Therefore we summarize that
\begin{align}{\label{u C^a estimate}}
    [u]_{\alpha,\frac{\alpha}{2};\mathbb{R}^n\times [0,T]}
    &\leq K( \|f\|_{\mathcal{C}^{0,\gamma}_{1-\frac{\alpha}{2}}(\mathbb{R}^n\times (0,T])}+ [u_0]_{\alpha;\mathbb{R}^n}).
\end{align}
From (\ref{C^0 estimate}) and (\ref{u C^a estimate}), the estimate (\ref{C^a estimate}) follows.

\bigskip
\item[\underline{Step3}:]
In this step we claim that the solution $u(x,t)$ satisfies the estimates
\begin{align}{\label{C^2 estimate}}
    \|D_xu\|_{\mathcal{C}^{1,\gamma}_{\frac{1}{2}-\frac{\alpha}{2}}(\mathbb{R}^n\times (0,T])}
    &\leq K( \|f\|_{\mathcal{C}^{0,\gamma}_{1-\frac{\alpha}{2}}(\mathbb{R}^n\times (0,T])}+ [u_0]_{\alpha;\mathbb{R}^n}).
\end{align}
Fix a point $z\in\mathbb{R}^n$ and $\sigma\in (0,T]$, we consider the parabolic cylinder
\begin{align}
	P_{\sigma}(z) := B_{\sqrt{\sigma}}(z)\times [0,\sigma].	
\end{align}
We define a function $v$ on the parabolic cylinder $P_1(0) = B_1(0)\times [0,1]$ by scaling:
    \begin{equation}
    v(y,s) := u(z + \sigma^{\frac{1}{2}}y, \sigma s)	.
    \end{equation}
  	Let $\chi(s)$ be a three times continuously differentiable cutoff function  on $[0,1]$ such that
    \begin{align}
    \chi(s)=
        \begin{cases}
        &1,\quad\text{if}\quad s\in[\frac{1}{2},1]\\
        &0, \quad\text{if}\quad s\in[0,\frac{1}{4}]
        \end{cases}
    \end{align}
    and
    \begin{align}
        |D_s^j\chi(s)|\leq C\quad, j=0,1,2.
    \end{align}
    We now define a function $\tilde{v}$ on $P_1(0)$ by 
    	$$\tilde{v}(y,s):=\chi(s)\left(v(y,s) - v(0,1)\right).$$ 
    Then $\tilde{v}$ satisfies
    \begin{align}
        \begin{cases}
        \frac{\partial\tilde{v}}{\partial s}-a_{ij}(z+\sigma^{\frac{1}{2}}y, \sigma s)\frac{\partial^2\tilde{v}}{\partial y_i\partial y_j}=\tilde{f}(y,s)\quad&\text{on}\quad B_1(0)\times(0,1]\\
        \tilde{v}(y,0)=0\quad&\text{on}\quad B_1(0),
        \end{cases}
    \end{align}
    where $\tilde{f}(y,s)=\chi'(s)(v(y,s) - v(0,1))+\sigma\chi(s)f(z+\sigma^{\frac{1}{2}}y, \sigma s)$.

     Note that $\tilde{f}\in C^{\gamma,\frac{\gamma}{2}}(B_1(0)\times[0,1])$ after change of variables since by the previous step we have $u\in C^{\alpha,\frac{\alpha}{2}}(B_{\sqrt{\sigma}}(z)\times [0,\sigma])$ and $\alpha\geq\gamma$. Hence by standard parabolic Schauder interior estimate $\tilde{v}\in C^{2+\gamma,\frac{2+\gamma}{2}}(B_{\frac{1}{2}}(0)\times[\frac{1}{2},1])$ and we have the following estimate: 
    \begin{align}
        |\tilde{v}|_{2+\gamma,\frac{2+\gamma}{2};B_{\frac{1}{2}}(0)\times [\frac{1}{2},1]}&\leq K(\|a_{ij}\|_{\gamma,\frac{\gamma}{2}})(|\chi'(v - v(0,1))|_{\gamma,\frac{\gamma}{2};B_1(0)\times [0,1]}+\sigma|\chi f|_{\gamma,\frac{\gamma}{2};B_1(0)\times [0,1]})\\
        &\notag\leq K(\|a_{ij}\|_{\gamma,\frac{\gamma}{2}}, C)(|v - v(0,1)|_{\gamma, \frac{\gamma}{2};B_1(0)\times [\frac{1}{4},1]}+\sigma|f|_{\gamma, \frac{\gamma}{2};B_1(0)\times [\frac{1}{4},1]} ).
    \end{align}
    Then it follows from the rescaling $u(x,t) = v(\sigma^{-\frac{1}{2}}(x-z),\sigma^{-1} t)$ that
    \begin{align}{\label{weight estimate 1}}
    &\sum_{i=1}^2\sigma^{\frac{i}{2}}|D^i_xu|_{0;B_{\sqrt{\sigma}/2}(z)\times [\frac{\sigma}{2},\sigma]}	+ \sum_{i=1}^2\sigma^{\frac{i}{2}+\frac{\gamma}{2}}[D^i_xu]_{\gamma,\frac{\gamma}{2};B_{\sqrt{\sigma}/2}(z)\times [\frac{\sigma}{2},\sigma]}\\
    &\notag\leq K(\|a_{ij}\|_{\gamma,\frac{\gamma}{2}}, C)\Big(|u - u(z,\sigma)|_{0;B_{\sqrt{\sigma}}(z)\times [\frac{\sigma}{4},\sigma]} + \sigma^{\frac{\gamma}{2}}[u - u(z,\sigma)]_{\gamma,\frac{\gamma}{2};B_{\sqrt{\sigma}}(z)\times [\frac{\sigma}{4},\sigma]}\\
    &\quad\quad\quad\quad\quad\quad\quad + \sigma|f|_{0;B_{\sqrt{\sigma}}(z)\times [\frac{\sigma}{4},\sigma]} + \sigma^{1+\frac{\gamma}{2}}[f]_{\gamma,\frac{\gamma}{2};B_{\sqrt{\sigma}}(z)\times [\frac{\sigma}{4},\sigma]}\Big)\notag\\
    &\notag\leq K(\|a_{ij}\|_{\gamma,\frac{\gamma}{2}}, C)\Big(|u - u(z,\sigma)|_{0;B_{\sqrt{\sigma}}(z)\times [\frac{\sigma}{4},\sigma]} + \sigma^{\frac{\gamma}{2}}[u - u(z,\sigma)]_{\gamma,\frac{\gamma}{2};B_{\sqrt{\sigma}}(z)\times [\frac{\sigma}{4},\sigma]}\\
    &\quad\quad\quad\quad\quad\quad\quad + \sigma^{\frac{\alpha}{2}}\|f\|_{\mathcal{C}^{0,\gamma}_{1-\frac{\alpha}{2}}(\mathbb{R}^n\times (0,T])}\Big).\notag   
    \end{align}
   Observe that $[f]_{\gamma,\frac{\gamma}{2};B_{\sqrt{\sigma}}(z)\times [\frac{\sigma}{4},\sigma]}$ can be controlled by norms of $f$ on $[\frac{\sigma}{4},\frac{\sigma}{2}]$ and $[\frac{\sigma}{2},\sigma]$. Indeed, for any $x\in B_{\sqrt{\sigma}}(z)$ and $\tau, t\in [\frac{\sigma}{4}, \sigma]$, we have
   \begin{align*}
   \frac{|f(x,t)-f(x,\tau)|}{|t-\tau|^{\frac{\gamma}{2}}}\leq  \frac{|f(x,t)-f(x,\frac{\sigma}{2})|}{|t-\tau|^{\frac{\gamma}{2}}}	+ \frac{|f(x,\tau)-f(x,\frac{\sigma}{2})|}{|t-\tau|^{\frac{\gamma}{2}}}.
   \end{align*}
	It suffices to consider the case where $\tau\leq\frac{\sigma}{2}\leq t$, then
   \begin{align*}
   \frac{|f(x,t)-f(x,\tau)|}{|t-\tau|^{\frac{\gamma}{2}}}&\leq  \frac{|f(x,t)-f(x,\frac{\sigma}{2})|}{|t-\frac{\sigma}{2}|^{\frac{\gamma}{2}}}	+ \frac{|f(x,\tau)-f(x,\frac{\sigma}{2})|}{|\tau-\frac{\sigma}{2}|^{\frac{\gamma}{2}}}\\
   &\leq [f]_{\gamma,\frac{\gamma}{2};B_{\sqrt{\sigma}}(z)\times [\frac{\sigma}{4},\frac{\sigma}{2}]} + [f]_{\gamma,\frac{\gamma}{2};B_{\sqrt{\sigma}}(z)\times [\frac{\sigma}{2},\sigma]}.
   \end{align*}
   This implies 
   \[ \sigma^{1+\frac{\gamma}{2}}[f]_{\gamma,\frac{\gamma}{2};B_{\sqrt{\sigma}}(z)\times [\frac{\sigma}{4},\sigma]} \leq \sigma^{\frac{\alpha}{2}}\|f\|_{\mathcal{C}^{0,\gamma}_{1-\frac{\alpha}{2}}(\mathbb{R}^n\times (0,T])}.\]
   Putting this inequality into (\ref{weight estimate 1}), we obtain
   
     \begin{align}{\label{weight estimate 2}}
    &\sum_{i=1}^2\sigma^{\frac{i}{2}}|D^i_xu|_{0;B_{\sqrt{\sigma}/2}(z)\times [\frac{\sigma}{2},\sigma]}	+ \sum_{i=1}^2\sigma^{\frac{i}{2}+\frac{\gamma}{2}}[D^i_xu]_{\gamma,\frac{\gamma}{2};B_{\sqrt{\sigma}/2}(z)\times [\frac{\sigma}{2},\sigma]}\\
     &\notag\leq K(\|a_{ij}\|_{\gamma,\frac{\gamma}{2}}, C)\Big(|u - u(z,\sigma)|_{0;B_{\sqrt{\sigma}}(z)\times [\frac{\sigma}{4},\sigma]} + \sigma^{\frac{\gamma}{2}}[u - u(z,\sigma)]_{\gamma,\frac{\gamma}{2};B_{\sqrt{\sigma}}(z)\times [\frac{\sigma}{4},\sigma]}\\
    &\quad\quad\quad\quad\quad\quad\quad + \sigma^{\frac{\alpha}{2}}\|f\|_{\mathcal{C}^{0,\gamma}_{1-\frac{\alpha}{2}}(\mathbb{R}^n\times (0,T])}\Big).\notag   
    \end{align}

    Now, note that the H\"older estimate for $[u]_{\alpha,\frac{\alpha}{2}}$ in (\ref{u C^a estimate}) implies that
    \begin{align*}
    |u - u(z,\sigma)|_{0;B_{\sqrt{\sigma}}(z)\times [\frac{\sigma}{4},\sigma]} &\leq K \sigma^{\frac{\alpha}{2}}\ [u]_{\alpha,\frac{\alpha}{2};B_{\sqrt{\sigma}}(z)\times [\frac{\sigma}{4},\sigma]}	\\
    &\leq K \sigma^{\frac{\alpha}{2}}( \|f\|_{\mathcal{C}^{0,\gamma}_{1-\frac{\alpha}{2}}(\mathbb{R}^n\times (0,T])}+ [u_0]_{\alpha;\mathbb{R}^n})
    \end{align*}
	and 
    \begin{align*}
    [u - u(z,\sigma)]_{\gamma,\frac{\gamma}{2};B_{\sqrt{\sigma}}(z)\times [\frac{\sigma}{4},\sigma]} &\leq K \sigma^{\frac{\alpha}{2}-\frac{\gamma}{2}}\ [u]_{\alpha,\frac{\alpha}{2};B_{\sqrt{\sigma}}(z)\times [\frac{\sigma}{4},\sigma]} \\
    &\leq K \sigma^{\frac{\alpha}{2}-\frac{\gamma}{2}} ( \|f\|_{\mathcal{C}^{0,\alpha}_{1-\frac{\alpha}{2}}(\mathbb{R}^n\times (0,T])}+ [u_0]_{\alpha;\mathbb{R}^n}).
    \end{align*}
Putting these facts into (\ref{weight estimate 2}), we obtain
 	\begin{align}
    &\sum_{i=1}^2\sigma^{\frac{i}{2}-\frac{\alpha}{2}}|D^i_xu|_{0;B_{\sqrt{\sigma}/2}(z)\times [\frac{\sigma}{2},\sigma]}	+ \sum_{i=1}^2\sigma^{\frac{i}{2}-\frac{\alpha}{2}+\frac{\gamma}{2}}[D^i_xu]_{\gamma,\frac{\gamma}{2};B_{\sqrt{\sigma}/2}(z)\times [\frac{\sigma}{2},\sigma]}\\
    &\notag\leq K(\|a_{ij}\|_{\gamma,\frac{\gamma}{2}}, C)\Big(\ \|f\|_{\mathcal{C}^{0,\gamma}_{1-\frac{\alpha}{2}}(M\times (0,T])} + \|u_0\|_{\alpha;\mathbb{R}^n} \Big).
    \end{align}
 	Since $z\in\mathbb{R}^n$ and $\sigma\in (0,T]$ are arbitrary, the desired estimate (\ref{C^2 estimate}) follows. Putting (\ref{C^a estimate}) and (\ref{C^2 estimate}) together, the lemma is thus proved.
	\end{itemize}
 	\end{proof}

\bigskip

Using a standard bootstrap argument, we can improve the regularity of $u$. We have the following auxiliary lemmas for higher order regularity:

\bigskip

\begin{lemma}{\label{lemma higher order u}}
Let $\alpha,\gamma\in (0,1)$ be given such that $\alpha>\gamma$. Suppose that 
\begin{itemize}
    \item[(1)]$a_{ij}(x,t)\in C^{\gamma,\frac{\gamma}{2}}(\mathbb{R}^n\times [0,T])$ and satisfies the uniform parabolicity condition. i.e. there is $\lambda>0$ such that $\frac{1}{\lambda}\delta_{ij}<a_{ij}(x,t)<\lambda\delta_{ij}$;
    \item[(2)] $\|D_xa_{ij}\|_{\mathcal{C}^{k-1,\gamma}_{\frac{1}{2}}(\mathbb{R}^n\times (0,T])}<\infty\ $ if $\ k\geq 1$;
    \item[(3)]$\|f\|_{\mathcal{C}^{k,\gamma}_{1-\frac{\alpha}{2}}(\mathbb{R}^n\times (0,T])}<\infty$ and $\|u_0\|_{\alpha;\mathbb{R}^n}<\infty$.
\end{itemize} Then the initial-value problem
\begin{align}
    \begin{cases}
    \frac{\partial}{\partial t}u(x,t)-a_{kl}(x,t)\frac{\partial^2}{\partial x_k\partial x_l}u(x,t)=f(x,t) &\quad\text{on}\quad \mathbb{R}^n\times(0,T]\\
    u(x,0)= u_0(x) &\quad\text{on}\quad\mathbb{R}^n
    \end{cases}
\end{align}
has a unique solution $u$, where $u\in  C^{\alpha,\frac{\alpha}{2}}(\mathbb{R}^n\times [0,T])$ and $Du \in\mathcal{C}^{k+1,\gamma}_{\frac{1}{2}-\frac{\alpha}{2}}(\mathbb{R}^n\times(0,T])$, such that
\begin{align}{\label{C^k estimate}}
    \|u\|_{\alpha,\frac{\alpha}{2};\mathbb{R}^n\times [0,T]} + \|D_xu\|_{\mathcal{C}^{k+1,\gamma}_{\frac{1}{2}-\frac{\alpha}{2}}(\mathbb{R}^n\times (0,T])}
    &\leq K(\|f\|_{\mathcal{C}^{k, \gamma}_{1-\frac{\alpha}{2}}(\mathbb{R}^n\times (0,T])} +\|u_0\|_{\alpha;\mathbb{R}^n}).
\end{align}
Here $K$ is a constant depending only on $k, \mathbb{R}^n, \|a_{ij}\|_{\alpha,\frac{\alpha}{2}}$ and $\|a_{ij}\|_{\mathcal{C}^{k,\alpha}_{0}}$.

\end{lemma}
\bigskip

\begin{proof}
We intend to prove the lemma by induction on $k$. The case $k=0$ follows from Lemma \ref{lemma second order u}. We suppose that the conditions (2) and (3) in the lemma holds with $k$ replaced by $k+1$. That is,
	\begin{itemize}
	\item $\|D_xa_{ij}\|_{\mathcal{C}^{k,\gamma}_{\frac{1}{2}}(\mathbb{R}^n\times (0,T])}<\infty\ $;
    \item $\|f\|_{\mathcal{C}^{k+1,\gamma}_{1-\frac{\alpha}{2}}(\mathbb{R}^n\times (0,T])}<\infty$ and $\|u_0\|_{\alpha;\mathbb{R}^n}<\infty$.	
	\end{itemize}
	Moreover, the induction hypothesis implies that the estimate
	\[  \|u\|_{\alpha,\frac{\alpha}{2};\mathbb{R}^n\times [0,T]} + \|D_xu\|_{\mathcal{C}^{k+1,\gamma}_{\frac{1}{2}-\frac{\alpha}{2}}(\mathbb{R}^n\times (0,T])} \leq K(\|f\|_{\mathcal{C}^{k, \gamma}_{1-\frac{\alpha}{2}}(\mathbb{R}^n\times (0,T])} +\|u_0\|_{\alpha;\mathbb{R}^n})\]
	  holds.  Let us fix $\sigma\in (0,T]$. Let $\chi(t)$ be a three times continuously differentiable cutoff function on $[0,\sigma]$ such that
    \begin{align}
    \chi(t)=
        \begin{cases}
        &1,\quad\text{if}\quad s\in[\frac{\sigma}{2},\sigma]\\
        &0, \quad\text{if}\quad s\in[0,\frac{\sigma}{4}]
        \end{cases}
    \end{align}
    and
    \begin{align}
        |D_t^h\chi(t)|\leq C\sigma^{-h}\quad, h=0,1,2.
    \end{align}
    It follows that the $\frac{\gamma}{2}$-H\"older norm for $\chi$ and $\chi'$ have estimates
    \[ [\chi]_{\frac{\gamma}{2}; [0,\sigma]}\leq C\sigma^{-\frac{\gamma}{2}}\quad\text{and}\quad [\chi']_{\frac{\gamma}{2}; [0,\sigma]}\leq C\sigma^{-1 -\frac{\gamma}{2}}. \]
 On $\mathbb{R}^n\times (0,\sigma]$, we define $\bar{u}(x,t):=\chi(t)D_x^{k+1}u(x,t)$, then $\bar{u}$ is the solution to the system 
\begin{align*}
    \begin{cases}
    \frac{\partial}{\partial t}\bar{u}(x,t)-a_{ij}\frac{\partial^2}{\partial x_i\partial x_j}\bar{u}(x,t)=\bar{f}(x,t) \quad &\text{on}\quad \mathbb{R}^n\times(0,\sigma]\\
    \bar{u}(x,0)=0 \quad &\text{on}\quad\mathbb{R}^n,
    \end{cases}
\end{align*}
where $\bar{f}$ is given by 
\begin{align}
    \bar{f}= \chi'D^{k+1}_xu + \chi D^{k+1}f + \chi\sum_{r=0}^k\binom{k+1}{r}D^{k+1-r}a_{ij}D^rD^2_{ij}u.
\end{align}
Note that by the induction hypothesis we have $\bar{f}\in C^{\gamma,\frac{\gamma}{2}}(\mathbb{R}^n\times [0,\sigma])$. Then we apply Theorem 5.1 in \cite{Lad} to obtain $D^{k+1}_xu\in C^{2+\gamma,\frac{2+\gamma}{2}}(\mathbb{R}^n\times [\frac{\sigma}{2},\sigma])$ and the estimate
\begin{align*}
	\sigma |D^{k+3}u|_{0;\mathbb{R}^n\times [\frac{\sigma}{2},\sigma]} + \sigma^{1+\frac{\gamma}{2}} [D^{k+3}u]_{\gamma,\frac{\gamma}{2};\mathbb{R}^n\times [\frac{\sigma}{2},\sigma]}	\leq K\Big(\ \sigma |\bar{f}|_{0;\mathbb{R}^n\times [0,\sigma]} + \sigma^{1+\frac{\gamma}{2}}[\bar{f}]_{\gamma,\frac{\gamma}{2};\mathbb{R}^n\times [0,\sigma]}\Big).
\end{align*}
Hence,
\begin{align}{\label{weight C^k estimate}}
	&\sigma^{\frac{1}{2}-\frac{\alpha}{2}+\frac{k+2}{2}} |D^{k+3}u|_{0;\mathbb{R}^n\times [\frac{\sigma}{2},\sigma]} + \sigma^{\frac{1}{2}-\frac{\alpha}{2}+ \frac{k+2}{2} +\frac{\gamma}{2} } [D^{k+3}u]_{\gamma,\frac{\gamma}{2};\mathbb{R}^n\times [\frac{\sigma}{2},\sigma]}\\
	&\notag\leq K\Big(\ \sigma^{\frac{1}{2}-\frac{\alpha}{2}+\frac{k+2}{2}} |\bar{f}|_{0;\mathbb{R}^n\times [\frac{\sigma}{4},\sigma]} + \sigma^{\frac{1}{2}-\frac{\alpha}{2}+\frac{k+2}{2} +\frac{\gamma}{2} }[\bar{f}]_{\gamma,\frac{\gamma}{2};\mathbb{R}^n\times [\frac{\sigma}{4},\sigma]}\Big).
\end{align}

We can check that (\ref{weight C^k estimate}) implies the estimate (\ref{C^k estimate}). For instance, we can check for the H\"older-semi norm term of $\bar{f}$ in the last line of (\ref{weight C^k estimate}). We have
\begin{align*}
	 \sigma^{\frac{1}{2}-\frac{\alpha}{2} +\frac{k+2}{2} +\frac{\gamma}{2}}[\chi'D^{k+1}_xu]_{\gamma,\frac{\gamma}{2};\mathbb{R}^n\times [\frac{\sigma}{4},\sigma]} &\leq K\Big(  \sigma^{\frac{1}{2}-\frac{\alpha}{2}+\frac{k}{2}}|D^{k+1}_xu|_{0;\mathbb{R}^n\times [\frac{\sigma}{4},\sigma]} +  \sigma^{\frac{1}{2}-\frac{\alpha}{2} +\frac{k}{2} +\frac{\gamma}{2}}[D^{k+1}_xu]_{\gamma,\frac{\gamma}{2};\mathbb{R}^n\times [\frac{\sigma}{4},\sigma]}\Big)\\
	 &\leq K(\|u\|_{\alpha,\frac{\alpha}{2};\mathbb{R}^n\times [0,T]} + \|D_xu\|_{\mathcal{C}^{k+1,\gamma}_{\frac{1}{2}-\frac{\alpha}{2}}(\mathbb{R}^n\times (0,T])}) \\
	&\leq K (\|f\|_{\mathcal{C}^{k-1,\gamma}_{1-\frac{\alpha}{2}}(\mathbb{R}^n\times (0,T])} + \|u_0\|_{\alpha;\mathbb{R}^n} )
\end{align*}
and

\begin{align*}
	\sigma^{\frac{1}{2}-\frac{\alpha}{2} +\frac{k+2}{2} +\frac{\gamma}{2}}[\chi D^{k+1}f]_{\gamma,\frac{\gamma}{2};\mathbb{R}^n\times [\frac{\sigma}{4},\sigma]} &\leq K\Big(\sigma^{\frac{1}{2}-\frac{\alpha}{2}+\frac{k+2}{2}} |D^{k+1}f|_{0;\mathbb{R}^n\times [\frac{\sigma}{4},\sigma]} + \sigma^{\frac{1}{2}-\frac{\alpha}{2} +\frac{k+2}{2} +\frac{\gamma}{2}}[ D^{k+1}f]_{\gamma,\frac{\gamma}{2};\mathbb{R}^n\times [\frac{\sigma}{4},\sigma]} \Big) \\
	&\leq K \|f\|_{\mathcal{C}^{k+1,\gamma}_{1-\frac{\alpha}{2}}(\mathbb{R}^n\times (0,T])}.
\end{align*}
Also, for each integer $r\in [0,k]$ we have

\begin{align*}
	&\sigma^{\frac{1}{2}-\frac{\alpha}{2} +\frac{k+2}{2} +\frac{\gamma}{2}}[\chi D^{k+1-r}a_{ij}D^ru_{ij}]_{\gamma,\frac{\gamma}{2};\mathbb{R}^n\times [\frac{\sigma}{4},\sigma]}	\\
	&\leq K\Big( \sigma^{\frac{1}{2}-\frac{\alpha}{2}+\frac{k+2}{2}}|D^{k+1-r}a_{ij}D^{r+2}u|_{0;\mathbb{R}^n\times [\frac{\sigma}{4},\sigma]} + \sigma^{\frac{1}{2}-\frac{\alpha}{2} +\frac{k+2}{2} +\frac{\gamma}{2}}|D^{k+1-r}a_{ij}|_{0;\mathbb{R}^n\times [\frac{\sigma}{4},\sigma]}[D^{r+2}u]_{\gamma,\frac{\gamma}{2};\mathbb{R}^n\times [\frac{\sigma}{4},\sigma]} \\
	&\quad\quad\quad + \sigma^{\frac{1}{2}-\frac{\alpha}{2} +\frac{k+2}{2} +\frac{\gamma}{2}}[D^{k+1-r}a_{ij}]_{\gamma,\frac{\gamma}{2};\mathbb{R}^n\times [\frac{\sigma}{4},\sigma]}|D^{r+2}u|_{0;\mathbb{R}^n\times [\frac{\sigma}{4},\sigma]}\Big)\\
	&\leq K\|D_xa_{ij}\|_{\mathcal{C}^{k,\gamma}_{\frac{1}{2}}(\mathbb{R}^n\times (0,T])}\ \|D_xu\|_{\mathcal{C}^{k+1,\gamma}_{\frac{1}{2}-\frac{\alpha}{2}}(\mathbb{R}^n\times (0,\sigma])}\\
	&\leq K\|D_xa_{ij}\|_{\mathcal{C}^{k,\gamma}_{\frac{1}{2}}(\mathbb{R}^n\times (0,T])}\ \big( \|f\|_{\mathcal{C}^{k, \gamma}_{1-\frac{\alpha}{2}}(\mathbb{R}^n\times (0,T])} +\|u_0\|_{\alpha;\mathbb{R}^n} \big).
\end{align*}
From these, we have
\begin{align}
	\sigma^{\frac{1}{2}-\frac{\alpha}{2} +\frac{k+2}{2} +\frac{\gamma}{2}}[\bar{f}]_{\gamma,\frac{\gamma}{2};\mathbb{R}^n\times [\frac{\sigma}{4},\sigma]} \leq 	K(a_{ij})\ \big( \|f\|_{\mathcal{C}^{k+1, \gamma}_{1-\frac{\alpha}{2}}(\mathbb{R}^n\times (0,T])} +\|u_0\|_{\alpha;\mathbb{R}^n} \big).
\end{align}
And similarly we get 
\begin{align}
	\sigma^{\frac{1}{2}-\frac{\alpha}{2} +\frac{k+2}{2}} |\bar{f}|_{0;\mathbb{R}^n\times [\frac{\sigma}{4},\sigma]} + \sigma^{\frac{1}{2}-\frac{\alpha}{2} +\frac{k+2}{2} +\frac{\gamma}{2}}[\bar{f}]_{\gamma,\frac{\gamma}{2};\mathbb{R}^n\times [\frac{\sigma}{4},\sigma]} \leq 	K(a_{ij})\ \big( \|f\|_{\mathcal{C}^{k+1, \gamma}_{1-\frac{\alpha}{2}}(\mathbb{R}^n\times (0,T])} +\|u_0\|_{\alpha;\mathbb{R}^n} \big).
\end{align}
Therefore we have completed the induction.
\end{proof}
\bigskip

\bigskip
\bigskip
\subsection{Completion of the proof}\

Now we are on the ground of constructing the operator $R:\mathcal{W}_k\to\mathcal{X}_{k+2}$. To begin with, let $\{\rho_s\}$ be the partition of unity subordinate to the charts $\{U_s\}$. On the basis of Lemma \ref{lemma second order u} and Lemma \ref{lemma higher order u}, there is a unique solution set $\{\eta^{r}_s\}$ to the system (\ref{scalar parabolic system}) such that for each $(r,s)$, 
	$$\eta^{r}_s\in C^{\alpha,\frac{\alpha}{2}}(\mathbb{R}^n\times [0,T]),\quad D_x\eta^{r}_s\in \mathcal{C}^{k+1,\gamma}_{\frac{1}{2}-\frac{\alpha}{2}}(\mathbb{R}^n\times (0,T])$$
 and 
\begin{align}{\label{component estimate}}
    \|\eta^{r}_s\|_{\alpha,\frac{\alpha}{2};\varphi_s(U_s)\times [0,T]} + \|D_x\eta^{r}_s\|_{\mathcal{C}^{k+1,\gamma}_{\frac{1}{2}-\frac{\alpha}{2}}(\varphi_s(U_s)\times (0,T])} 
    &\leq  K(||F||_{\mathcal{C}^{k,\gamma}_{1-\frac{\alpha}{2}}(M\times (0,T])} + \|\eta_0\|_{\alpha;M} ).
\end{align}
\newline
We then define $R_s:\mathcal{W}_k\to \Gamma(U_s\times[0,T], \pi^{-1}(U_s))$ by 
\begin{align}
    (R_sh)(p,t) = \rho_s(p)\sum_r \eta^{r}_s(\varphi_s(p),t)\bold{e}_r^s(p)
\end{align}
for any $(p,t)\in U_s\times[0,T]$.  Moreover, by setting $R_sh$ to be the zero section outside the support of $\rho_s$, we can extend $R_sh$ to a section of $ \Gamma(M\times[0,T], E)$. Now, we define a map $R:\mathcal{W}_k\to\Gamma(M\times [0,T], E)$ by 
\begin{align}
    Rh = \sum_s R_sh.
\end{align}
From the construction of $R$ it is clear that $Rh\in\mathcal{X}_{k+2}$ for any $h\in\mathcal{W}_k$. By (\ref{component estimate}), $Rh$ satisfies the following estimates: 
\begin{align}{\label{R estimate}}
    \|Rh\|_{\mathcal{X}_{k+2}(M\times [0,T])}
    &\leq  K \|h\|_{\mathcal{W}_k(M\times (0,T])}.
\end{align}
Here $K$ is a constant depending only on $k, M, A$. Recall that $A$ is the positive constant such that $\|w\|_{\gamma,\frac{\gamma}{2};M\times [0,T]} + \|\hat{\nabla}w\|_{\mathcal{C}^{k-1,\gamma}_{\frac{1}{2}}(M\times (0,T])}\leq A$ provided in the assumption of the theorem.

Next, we define $S:\mathcal{W}_k\to\mathcal{W}_k$ and $G:\mathcal{X}_{k+2}\to\mathcal{X}_{k+2}$ by
\begin{align}{\label{two operators}}
    \begin{cases}
    &Sh=HRh-h\\
    &G\eta=RH\eta-\eta
    \end{cases}.
\end{align}
\bigskip

\begin{lemma}{\label{good operators}}
The operators defined above in (\ref{two operators}) satisfies
\begin{align*}
    &||Sh||_{\mathcal{W}_k}\leq K T^{\frac{1}{2}} ||h||_{\mathcal{W}_k}\\
    \text{and}\quad &||G\eta||_{\mathcal{X}_{k+2}}\leq K T^{\frac{1}{2}} ||\eta||_{\mathcal{X}_{k+2}}.
\end{align*}
Here $K$ is a constant depending only on $k, M, \hat{g}, A$.
\end{lemma}
\bigskip
\begin{proof}
Firstly, note that on the intersection $U_{\mu}\cap U_{\nu}$, the transition map for the vector bundle $E$ is given by
	\[\tilde{\varphi}_{\mu \nu} := \tilde{\varphi}_{\mu}\circ(\tilde{\varphi}_{\nu})^{-1} : (U_{\mu}\cap U_{\nu})\times\mathbb{R}^N \to (U_{\mu}\cap U_{\nu})\times\mathbb{R}^N. \]
	We denote by $\Phi_{\mu \nu}: (U_{\mu}\cap U_{\nu})\to \textup{GL}_N\mathbb{R}$ the induced isomorphism which is given by
	\[ \tilde{\varphi}_{\mu \nu} (p, V) := (p, \Phi_{\mu \nu}(x)V),\quad\forall p\in U_{\mu}\cap U_{\nu},\ V\in\mathbb{R}^N. \]
	Then with respect to the canonical basis $\{\bold{e}_i^{\mu}\}$ and $\{\bold{e}_j^{\nu}\}$ of $\pi^{-1}(U_{\mu})$ and $\pi^{-1}(U_{\nu})$, the matrix components for the map $\Phi_{\mu \nu}$ are given by
	\[ \bold{e}_i^{\mu} = (\Phi_{\mu \nu})_i^j(\bold{e}_j^{\nu}).\]
	With this notation, we can write the components of $R_{\mu}h$ on  $U_{\mu}\cap U_{\nu}$ with respect to the basis  $\{\bold{e}_j^{\nu}\}$ as
	\[ (R_{\mu}h)^j_{\nu} = \sum_i\rho_{\mu}\eta_{\mu}^i(\Phi_{\mu \nu})_i^j.\]
	Thus we can write the component $(Rh)^i_{\mu}$ on $U_{\mu}$ as
	\begin{align}{\label{Rh component}}
		 (Rh)^i_{\mu} = \sum_{\nu}(R_{\nu}h)^i_{\mu} = \sum_{j,\nu}\rho_{\nu}(\Phi_{\nu \mu})_j^i\eta_{\nu}^j.
	\end{align}

Now, let us prove the first inequality in the lemma. For any $h = (F, \eta_0)\in\mathcal{W}_k$, we have $$Sh=(LRh-F, 0).$$
Let us write $v=Rh$. Then on each chart $U_{\mu}$, we have $v=\sum_iv^i_{\mu}\bold{e}_i^{\mu}$ satisfying
\[ v^i_{\mu} = \sum_{j,\nu}\rho_{\nu}(\Phi_{\nu \mu})_j^i\eta_{\nu}^j,\]
where $\{\eta_{\nu}^j\}$ denotes the unique solution of the system (\ref{scalar parabolic system}). By (\ref{Rh component}), we have
	\[\frac{\partial}{\partial t}v^i_{\mu} = \sum_{j,\nu}\rho_{\nu}(\Phi_{\nu \mu})_j^i(\frac{\partial}{\partial t}\eta_{\nu}^j),\]
	and
	\[ w^{kl}D^2_{kl}v^i_{\mu} = \sum_{j,\nu}\left\{\rho_{\nu}(\Phi_{\nu \mu})_j^i(w^{kl}D^2_{kl}\eta_{\nu}^j) + 2w^{kl}D_k(\rho_{\nu}(\Phi_{\nu \mu})_j^i)\cdot D_l(\eta_{\nu}^j) + w^{kl}D^2_{kl}(\rho_{\nu}(\Phi_{\nu \mu})_j^i)\cdot\eta_{\nu}^j \right\}.\]
This implies that $\{v^i_{\mu}\}$ satisfies the system
\begin{align*}
    \begin{cases}
    \frac{\partial}{\partial t}v^i_{\mu}(x,t)-w^{kl}D^2_{kl}v^i_{\mu}(x,t)= \tilde{F}^i_{\mu}(x,t) &\quad\text{on}\quad \varphi_{\mu}(U_{\mu})\times(0,T], \quad i=1,..,N\\
    v^i_{\mu}(x,0)=0 &\quad\text{on}\quad \varphi_{\mu}(U_{\mu}),\quad i=1,..,N
    \end{cases},
\end{align*}
where 
\begin{align}{\label{tilde F}}
	\tilde{F}^i_{\mu} = \sum_{j,\nu}\left\{\rho_{\nu}(\Phi_{\nu \mu})_j^iF_{\nu}^j + 2w^{kl}D_k(\rho_{\nu}(\Phi_{\nu \mu})_j^i)\cdot D_l(\eta_{\nu}^j) + w^{kl}D^2_{kl}(\rho_{\nu}(\Phi_{\nu \mu})_j^i)\cdot\eta_{\nu}^j\right\}.
\end{align}
Note that $v$ satisfies the estimate 
\[\|v\|_{\alpha,\frac{\alpha}{2};M\times [0,T]} + \|\hat{\nabla}v\|_{\mathcal{C}^{k+1,\gamma}_{\frac{1}{2}-\frac{\alpha}{2}}(M\times (0,T])}\leq K\|h\|_{\mathcal{W}_k(M\times (0,T])}\]
by (\ref{R estimate}). And by $F_{\mu}^i = \sum_{\nu}\rho_{\nu}F_{\mu}^i = \sum_{j,\nu}\rho_{\nu}(\Phi_{\nu \mu})_j^iF_{\nu}^j$, we have the estimate
\begin{align}{\label{F estimate}}
	&\|\tilde{F}^i_{\mu} - F^i_{\mu}\|_{\mathcal{C}^{k,\gamma}_{1-\frac{\alpha}{2}}(\mathbb{R}^N\times (0,T])}\\
	&\notag= \left\|\sum_{j,\nu}\left\{  2w^{kl}D_k(\rho_{\nu}(\Phi_{\nu \mu})_j^i)\cdot D_l(\eta_{\nu}^j) + w^{kl}D^2_{kl}(\rho_{\nu}(\Phi_{\nu \mu})_j^i)\cdot\eta_{\nu}^j\right\}\right\|_{\mathcal{C}^{k,\gamma}_{1-\frac{\alpha}{2}}(\mathbb{R}^N\times (0,T])}\\
	&\notag\leq K\sum_{j,\nu}\left\{\|D_x\eta_{\nu}^j\|_{\mathcal{C}^{k,\gamma}_{1-\frac{\alpha}{2}}(\mathbb{R}^N\times (0,T])} + \|\eta_{\nu}^j\|_{\mathcal{C}^{k,\gamma}_{1-\frac{\alpha}{2}}(\mathbb{R}^N\times (0,T])}\right\}\\
	&\notag\leq K(\|v\|_{\mathcal{C}^{k,\gamma}_{1-\frac{\alpha}{2}}(M\times (0,T])} + \|\hat{\nabla}v\|_{\mathcal{C}^{k,\gamma}_{1-\frac{\alpha}{2}}(M\times (0,T])})\\
	 &\leq KT^{\frac{1}{2}} (\|v\|_{\alpha,\frac{\alpha}{2};M\times [0,T]} + \|\hat{\nabla}v\|_{\mathcal{C}^{k+1,\gamma}_{\frac{1}{2}-\frac{\alpha}{2}}(M\times (0,T])}). \notag
\end{align}

  Hence, we have
\begin{align*}
    LRh-F &= \sum_r\left(\frac{\partial}{\partial t}v^i_{\mu}(x,t)-w^{kl}\hat{\nabla}_k\hat{\nabla}_lv^i_{\mu}(x,t)-F^{\mu}_i(x,t)\right)\bold{e}_{\mu}^i\\
    &= \sum_r\left((\partial\hat{\Gamma}+\hat{\Gamma}*_w\hat{\Gamma})*_wv+\hat{\Gamma}*_w\hat{\nabla}v + \tilde{F}_{\mu}^i -F_{\mu}^i \right)\bold{e}_{\mu}^i,\notag
\end{align*}
where $\hat{\Gamma}$ are the connection terms with respect to $\hat{g}$. Thus, using Lemma \ref{properties of weighted space} and (\ref{F estimate}) we obtain the estimate
\begin{align}
    \|Sh\|_{\mathcal{W}_k(M\times (0,T])} 
    &= ||LRh-F||_{\mathcal{C}^{k,\gamma}_{1-\frac{\alpha}{2}}(M\times (0,T])}\\
    &\notag\leq K \left(\|v\|_{\mathcal{C}^{k,\gamma}_{1-\frac{\alpha}{2}}(M\times (0,T])} + \|\hat{\nabla}v\|_{\mathcal{C}^{k,\gamma}_{1-\frac{\alpha}{2}}(M\times (0,T])} +  \sum_{i,\mu}\|\tilde{F}^i_{\mu} - F^i_{\mu}\|_{\mathcal{C}^{k,\gamma}_{1-\frac{\alpha}{2}}(\mathbb{R}^N\times (0,T])}\right) \\
    &\leq KT^{\frac{1}{2}} (\|v\|_{\alpha,\frac{\alpha}{2};M\times [0,T]} + \|\hat{\nabla}v\|_{\mathcal{C}^{k+1,\gamma}_{\frac{1}{2}-\frac{\alpha}{2}}(M\times (0,T])}) \notag\\
    &\leq KT^{\frac{1}{2}} \|h\|_{\mathcal{W}_k(M\times (0,T])},\notag 
\end{align}
where $K = K(k, M, \hat{g}, A)$. 

\bigskip

\bigskip

In the next step, we derive the second inequality in the lemma. For any $\eta\in\mathcal{X}_{k+2}(M\times [0,T];E)$, we denote by $(F,\eta_0):= (L\eta, \eta(\cdot,0)) = H\eta$ and $\zeta :=RH\eta$. Then on each chart $U_{\mu}$, we see that
     $ (F,\eta_0) = (\sum_{r=1}^NF_{\mu}^i\bold{e_i^{\mu}},\ \eta(\cdot,0))$ satisfies 
    \begin{align*}
    F_{\mu}^i(x,t) = \frac{\partial}{\partial t}\eta_{\mu}^i(x,t)-w^{kl}\hat{\nabla}_k\hat{\nabla}_l\eta_{\mu}^i(x,t)\quad\text{on}\quad\varphi_{\mu}(U_{\mu})\times(0,T],\quad r=1,..,N.
    \end{align*}
	Now, we denote by $\{\tilde{\eta}_{\nu}^j\}$ the unique solution of the system
	\begin{align*}
    \begin{cases}
    \frac{\partial}{\partial t}\tilde{\eta}_{\nu}^j(x,t)-w^{kl}(x,t)D^2_{kl}\tilde{\eta}_{\nu}^j(x,t)= F_{\nu}^j(x,t) &\quad\text{on}\quad \varphi_{\nu}(U_{\nu})\times(0,T], \quad r=1,..,N\\
   \tilde{\eta}_{\nu}^j(x,0)= (\eta_0)_{\nu}^j(x) &\quad\text{on}\quad\varphi_{\nu}(U_{\nu}),  \quad r=1,..,N.
    \end{cases}
	\end{align*}
	Then by (\ref{Rh component}),  $RH\eta = \zeta = \sum_{r=1}^N\zeta_{\mu}^i\bold{e_i^{\mu}}$ satisfies 
	\[\zeta^i_{\mu} = \sum_{j,\nu}\rho_{\nu}(\Phi_{\nu \mu})_j^i\tilde{\eta}_{\nu}^j,\]
	and in particular
	\[ \zeta^i_{\mu}(x,0) = \sum_{j,\nu}\rho_{\nu}(\Phi_{\nu \mu})_j^i(\eta_0)_{\nu}^j(x,0) =  (\eta_0)_{\mu}^i(x,0). \]
	Moreover, the components $\{\zeta^i_{\mu}\}$ satisfy the system
	\begin{align*}
    \begin{cases}
    \frac{\partial}{\partial t}\zeta^i_{\mu}(x,t)-w^{kl}D^2_{kl}\zeta^i_{\mu}(x,t)= \tilde{F}^i_{\mu}(x,t) &\quad\text{on}\quad \varphi_{\mu}(U_{\mu})\times(0,T], \quad i=1,..,N\\
    \zeta^i_{\mu}(x,0)=  (\eta_0)_{\mu}^i(x,0) &\quad\text{on}\quad \varphi_{\mu}(U_{\mu}),\quad i=1,..,N,
    \end{cases}
	\end{align*}
	where $\tilde{F}_{\mu}^i$ is again defined by (\ref{tilde F}) with $\eta$ replaced by $\tilde{\eta}$. Similar to (\ref{F estimate}), we have the estimate
	\begin{align}{\label{F estimate 2}}
	&\|\tilde{F}^i_{\mu} - F^i_{\mu}\|_{\mathcal{C}^{k,\gamma}_{1-\frac{\alpha}{2}}(\mathbb{R}^N\times (0,T])}\\
	&\notag= \left\|\sum_{j,\nu}\left\{  2w^{kl}D_k(\rho_{\nu}(\Phi_{\nu \mu})_j^i)\cdot D_l(\tilde{\eta}_{\nu}^j) + w^{kl}D^2_{kl}(\rho_{\nu}(\Phi_{\nu \mu})_j^i)\cdot\tilde{\eta}_{\nu}^j\right\}\right\|_{\mathcal{C}^{k,\gamma}_{1-\frac{\alpha}{2}}(\mathbb{R}^N\times (0,T])}\\
	&\notag\leq K\sum_{j,\nu}\left\{\|D_x\tilde{\eta}_{\nu}^j\|_{\mathcal{C}^{k,\gamma}_{1-\frac{\alpha}{2}}(\mathbb{R}^N\times (0,T])} + \|\tilde{\eta}_{\nu}^j\|_{\mathcal{C}^{k,\gamma}_{1-\frac{\alpha}{2}}(\mathbb{R}^N\times (0,T])}\right\}\\
	&\notag\leq KT^{\frac{1}{2}} \sum_{j,\nu}\left\{\|D_x\tilde{\eta}_{\nu}^j\|_{\mathcal{C}^{k+1,\gamma}_{\frac{1}{2}-\frac{\alpha}{2}}(\mathbb{R}^N\times (0,T])} + \|\tilde{\eta}_{\nu}^j\|_{\alpha,\frac{\alpha}{2};\mathbb{R}^N\times [0,T]}\right\}\\
	&\leq KT^{\frac{1}{2}}  (||F||_{\mathcal{C}^{k,\gamma}_{1-\frac{\alpha}{2}}(M\times (0,T])} + \|\eta_0\|_{\alpha;M} ) \notag\\
	&\leq  K T^{\frac{1}{2}} \|\eta\|_{\mathcal{X}_{k+2}(M\times [0,T])},\notag
	\end{align}
where the second last inequality follows from Lemma \ref{lemma higher order u}.\\
This implies that $G\eta = RH\eta-\eta = \sum_{r=1}^N(\zeta^i_{\mu}-\eta^i_{\mu})\bold{e_i^{\mu}}$ satisfies 
\begin{align*}
    \begin{cases}
    \frac{\partial}{\partial t}(\zeta^i_{\mu}-\eta^i_{\mu})-w^{kl}D^2_{kl}(\zeta^i_{\mu}-\eta^i_{\mu}) = (\partial\hat{\Gamma}+\hat{\Gamma}*_w\hat{\Gamma})*_w\eta+\hat{\Gamma}*_w\hat{\nabla}\eta + \tilde{F}^i_{\mu} - F^i_{\mu} &\quad\text{on}\quad\varphi_s(U_s)\times(0,T]\\
    (\zeta^i_{\mu}-\eta^i_{\mu})(x,0)=0  &\quad\text{on}\quad\varphi_s(U_s).
    \end{cases}
\end{align*}
Then (\ref{F estimate 2}) and Lemma \ref{lemma higher order u} imply that 
\begin{align}
    &||G\eta||_{\mathcal{X}_{k+2}(M\times [0,T])}\\
    &\notag\leq K \left( ||(\partial\hat{\Gamma}+\hat{\Gamma}*_w\hat{\Gamma})*_w\eta||_{\mathcal{C}^{k,\gamma}_{1-\frac{\alpha}{2}}(M\times (0,T])} + ||\hat{\Gamma}*_w\hat{\nabla}\eta||_{\mathcal{C}^{k,\gamma}_{1-\frac{\alpha}{2}}( M\times (0,T])} + \sum_{i,\mu}\|\tilde{F}^i_{\mu} - F^i_{\mu}\|_{\mathcal{C}^{k,\gamma}_{1-\frac{\alpha}{2}}(\mathbb{R}^N\times (0,T])}\right)\\
    &\leq KT^{\frac{1}{2}} \left(\|\eta\|_{\alpha,\frac{\alpha}{2};M\times [0,T]} + \|\hat{\nabla}\eta\|_{\mathcal{C}^{k+1,\gamma}_{\frac{1}{2}-\frac{\alpha}{2}}(M\times (0,T])} + \|\eta\|_{\mathcal{X}_{k+2}(M\times [0,T])}\right) \notag\\
    &\leq K T^{\frac{1}{2}} ||\eta||_{\mathcal{X}_{k+2}(M\times [0,T])}.  \notag
\end{align}
We thereby proved the Lemma. 
\end{proof}
\bigskip
To complete the proof of Theorem \ref{theorem linear equation}, we choose $T^*$ sufficiently small so that the constant given in Lemma \ref{good operators} satisfies $KT^* \leq\frac{1}{2}$. Then for each $T\in (0,T^*]$, Lemma \ref{good operators} implies 
$$||Sh||_{\mathcal{W}_k}\leq \frac{1}{2}||h||_{\mathcal{W}_k},\quad ||G\eta||_{\mathcal{X}_{k+2}}\leq \frac{1}{2}||\eta||_{\mathcal{X}_{k+2}}.$$
Then on the basis of contraction mapping principle and the Fredholm alternative the operators $Id_{\mathcal{W}}+S$ and $Id_{\mathcal{X}}+G$ have bounded inverses. From this and (\ref{two operators}) we  conclude that the operator $H$ has bounded inverse such that
\begin{align}
    R(Id_{\mathcal{W}}+S)^{-1} = (Id_{\mathcal{X}}+G)^{-1}R = H^{-1}.
\end{align}
Consequently we have
\begin{align}
   &\|\eta\|_{\alpha,\frac{\alpha}{2};M\times [0,T]} + \|\hat{\nabla}\eta\|_{\mathcal{C}^{k+1,\gamma}_{\frac{1}{2}-\frac{\alpha}{2}}(M\times (0,T])} \\
    &\notag = ||\eta||_{\mathcal{X}_{k+2}(M\times [0,T])}\\
    &\notag = ||H^{-1}h||_{\mathcal{X}_{k+2}(M\times [0,T])}\\
    &\notag \leq K||h||_{\mathcal{W}_k}\\
    &\notag \leq K\big(||F||_{\mathcal{C}^{k,\gamma}_{1-\frac{\alpha}{2}}(M\times (0,T])} + \|\eta_0\|_{\alpha; M} \big)
\end{align}
for each $T\in (0,T^*]$. Here $K$ is a constant depending only on $k, M, \hat{g}, A$. Having established the theorem on a small time interval $[0,\min\{T^*,I\}]$, we can prove the theorem on an arbitrary interval $[0,T]$ for each $T\leq I$ by standard parabolic theory. Since for $t>0$,  we have $\eta(\cdot,t)\in C^{k+2+\gamma}(M)$ and $F, w\in C^{k+\gamma,\frac{k+\gamma}{2}}(M\times [t,T])$. Therefore we have established Theorem \ref{theorem linear equation}. 
\bigskip

\newpage
\section{Ricci Flow with H\"older continuous initial metrics}

Let $M$ be a smooth, compact Riemannian manifold with boundary $\partial M$. Let $g_0$ be a smooth Riemannian metric on $M$. Recall that the goal of this paper to prove short time existence to Ricci flow on manifold with boundary in the following sense:
\begin{align}{\label{Ricci flow system}}
    \begin{cases}
    \frac{\partial}{\partial t}g(t)=-2Ric(g(t))\quad&\text{on}\quad M\times(0,T]\\
    A_{g(t)}=0\quad\quad\quad\quad\quad &\text{on}\quad \partial M\times (0,T],
    \end{cases}
\end{align}
where $A_{g(t)}$ is the second fundamental form of $\partial M$ in $(M,g(t))$. In \cite{Sh}, Shen proved short time existence to the above system provided that the initial metric $g_0$ is totally geodesic. We remark here that we do not impose any condition on the boundary second fundamental form $A_{g_0}$ for the initial metric $g_0$.

\bigskip

Equivalently, we can prove short time existence to the above boundary value problem via doubling the manifold and solving the corresponding Ricci flow on the doubled manifold but with rough initial data. Let $\tilde{M}$ be the double of $M$. More precisely, we define $\tilde{M} = {}^{M_1\cup M_2}/_{\sim}$, where $M_1$ and $M_2$ are identical copies of $M$ and $p_1\sim p_2$ if $p\in\partial M$. Fix a smooth background metric $\hat{g}$ on $M$ such that in a small collar neighborhood of $\partial M$ the metric $\hat{g}$ is isometric to a product $\partial M\times [0,\varepsilon)$. Note that $\hat{g}$ extends to a smooth metric on the doubled manifold $\tilde{M}$ via reflection about $\partial M$, which we would still denote it as $\hat{g}$. Next, we extend $g_0$ to a metric $\tilde{g}_0$ on the doubled manifold $\tilde{M}$ via reflection about $\partial M$. Then $\tilde{g}_0$ is a Lipschitz metric on $\tilde{M}$. In particular, we have $\tilde{g}_0\in C^{\alpha}(M; \textup{Sym}^2(T^*M))$ for all $\alpha\in(0,1)$. We consider the Ricci flow on $\tilde{M}$:
\begin{align}{\label{Ricci flow on double}}
    \frac{\partial}{\partial t}\tilde{g}(t)=-2Ric(\tilde{g}(t))\quad&\text{on}\quad \tilde{M}\times(0,T].
\end{align}
Note that solving (\ref{Ricci flow system}) is equivalent to solving (\ref{Ricci flow on double}) with initial metric $\tilde{g}_0$. Namely if $\tilde{g}(t)$ is a smooth solution to (\ref{Ricci flow on double}) on $\tilde{M}\times (0,I]$, then $\tilde{g}(t)$ preserves the $\mathbb{Z}_2$ symmetry of $\tilde{g}_0$ and therefore $\tilde{g}(t)|_{M_i}$ is a smooth solution to (\ref{Ricci flow system}) on $M_i\times (0,I]$ such that $(M_i, \tilde{g}(t)|_{M_i})$ has totally geodesic boundary. 

\bigskip
We use the DeTurck's trick to relate the system (\ref{Ricci flow on double}) to a modified system which is strictly parabolic. In the sequel of this section, we will work on the doubled manifold $(\tilde{M},\tilde{g}_0)$. For the sake of notation simplicity we will still denote the doubled Riemannian manifold $(\tilde{M}, \tilde{g}_0)$ as $(M, g_0)$ if no confusions would be made. $\hat{\nabla}$ will denote the covariant derivative with respect to the background metric $\hat{g}$. We consider the following Ricci-DeTurck system on the doubled manifold $(M,g_0)$:

\begin{align}{\label{Ricci DeT}}
    \begin{cases}
    \frac{\partial}{\partial t}g  = -2Ric(g)+\mathcal{L}_Wg\quad &\text{on}\quad M\times (0,T]\\
     g(x,0) = g_0 \quad &\text{on}\quad M.
    \end{cases}
\end{align}
Here the vector field $W$ is defined as  $W^k=g^{ij}((\Gamma_{g})_{ij}^k-(\hat{\Gamma}_{\hat{g}})_{ij}^k)$. Moreover, we have $-2Ric(g)+\mathcal{L}_Wg = g^{kl}\hat{\nabla}_k\hat{\nabla}_lg + \mathcal{Q}(g,\hat{\nabla}g) $, where $\mathcal{Q}(g,\hat{\nabla}g)$ is defined by
\begin{align*}
     \mathcal{Q}(g,\hat{\nabla}g)_{ij} =&-g^{kl}g_{ip}\hat{g}^{pq}\hat{R}_{jkql} - g^{kl}g_{jp}\hat{g}^{pq}\hat{R}_{ikql}\\
    &+\frac{1}{2}g^{kl}g^{pq}(\hat{\nabla}_ig_{pk}\hat{\nabla}_jg_{ql}+2\hat{\nabla}_kg_{ip}\hat{\nabla}_qg_{jl}-2\hat{\nabla}_kg_{ip}\hat{\nabla}_lg_{jq}\\
    &-4\hat{\nabla}_ig_{pk}\hat{\nabla}_lg_{jq}).
\end{align*}
 We remark that the initial metric $g_0$ is merely H\"older-continuous. We will use the Banach fixed point theorem to prove existence of a short time solution to (\ref{Ricci DeT}). \\
 
 To begin with, we define a suitable Banach space for the solutions. \\
  Let $\alpha,\gamma\in (0,1)$ be given such that $\alpha > \gamma$ and let $k\geq 1$, for $\eta\in\Gamma(M\times[0,T];E)$ we define a norm $\|\eta\|_{\mathcal{X}_{k,\gamma}^{(\alpha)}(M\times[0,T])}$ by
 \begin{align}
	&\|\eta\|_{\mathcal{X}_{k,\gamma}^{(\alpha)}(M\times[0,T])}:= ||\eta||_{\alpha,\frac{\alpha}{2};M\times [0,T]}  + ||\hat{\nabla}\eta||_{\mathcal{C}^{k-1,\gamma}_{\frac{1}{2}-\frac{\alpha}{2}}(M\times (0,T])}.
\end{align}
Moreover we define the Banach space 
	\begin{align}
		\mathcal{X}_{k,\gamma}^{(\alpha)}(M\times[0,T];E) := \{\eta: M\times [0,T]\to E|\ \|\eta\|_{\mathcal{X}_{k,\gamma}^{(\alpha)}(M\times[0,T])}<\infty.\}.
	\end{align}

\bigskip

\subsection{Formulation of the existence result to the Ricci-DeTurck Flow}\

\

Let $\alpha, \gamma\in (0,1)$ be given such that $\alpha>\gamma$ and $g_0\in C^{\alpha}(M)$. Denote by $E=\textup{Sym}^2(T^*M)$ the tensor bundle of symmetric (0,2)-tensors on the closed manifold $M$. The system (\ref{Ricci DeT}) is equivalent to 
\begin{align}{\label{Ricci DeT 2}}
    \begin{cases}
    \frac{\partial}{\partial t}g(x,t) - tr_{g}\hat{\nabla}^2g(x,t) =  \mathcal{Q}(g, \hat{\nabla}g)(x,t) \quad &\text{on}\quad M\times (0,T]\\
    g(x,0) = g_0 \quad\quad\quad\quad\quad\quad\quad\quad\quad\quad\quad\quad\quad &\text{on}\quad M.
    \end{cases}
\end{align}
The result in section 3 will help us to apply the Banach fixed point theorem.
\bigskip

 Let $w\in \mathcal{X}_{2,\gamma}^{(\alpha)}(M\times[0,T];E)$ be given such that $w(\cdot,t)$ is a family of Riemannian metrics on $M\times [0,T]$, we consider the following linear system:
\begin{align}{\label{linear Ricci DeT}}
    \begin{cases}
    \frac{\partial}{\partial t}\eta(x,t) - tr_{g_0}\hat{\nabla}^2\eta(x,t) = tr_w\hat{\nabla}^2w(x,t) - tr_{g_0}\hat{\nabla}^2w(x,t) +  \mathcal{Q}(w, \hat{\nabla}w)(x,t) \quad &\text{on}\quad M\times (0,T]\\
    \eta(x,0) = g_0(x)\quad\quad\quad\quad\quad\quad\quad\quad\quad\quad\quad\quad\quad &\text{on}\quad M
    \end{cases}
\end{align}

Note that if a solution $\eta$ to (\ref{linear Ricci DeT}) satisfies $\eta = w$, then $\eta$ solves the nonlinear system (\ref{Ricci DeT 2}).

\bigskip

\begin{proposition}{\label{prop linear Ricci DeT}}
Consider the linear system (\ref{linear Ricci DeT}). Suppose that
\begin{itemize}
	\item[(1)]  $w(x,0) = g_0(x)$;
	\item[(2)]  $w(\cdot,t)$ is a family of Riemannian metrics on $M\times [0,T]$ such that
		\[\Lambda g_0(x)\geq w(x,t)\geq\frac{1}{\Lambda}g_0(x)\]
		for any $(x,t)\in M\times[0,T]$. 
	\item[(3)]  $\|w\|_{\mathcal{X}_{2,\gamma}^{(\alpha)}(M\times[0,T])} \leq A$.
\end{itemize}
Then there is a unique solution $\eta\in\mathcal{X}_{2,\gamma}^{(\alpha)}(M\times [0,T])$ to the system (\ref{linear Ricci DeT}) and positive constants $K_1 = K_1(M,  \hat{g}, \|g_0\|_{\alpha}), K_2 = K_2(M, \Lambda, \hat{g}, \|g_0\|_{\alpha}, A)$ such that
    \begin{align*}
       \|\eta\|_{\mathcal{X}_{2,\gamma}^{(\alpha)}(M\times [0,T])} \leq K_1(K_2\  T^{\frac{\gamma}{2}} + \|g_0\|_{\alpha;M}).  
    \end{align*}
Moreover, there exists $T^* = T^*(M, \Lambda, \hat{g}, \|g_0\|_{\alpha}, A)$ such that $\eta(\cdot, t)$ is a family of Riemannian metrics on $M\times [0, \min\{T, T^*\}]$ satisfying
			\[\Lambda g_0(x)\geq \eta(x,t) \geq\frac{1}{\Lambda}g_0(x)\]
	for any $(x,t)\in M\times[0,T]$.
\end{proposition}

\begin{proof}
In the sequel, $K_1$ will denote a constant depending only on $M, \hat{g},  \|g_0\|_{\alpha;M}$ and  $K_2$ will denote a constant depending only on $M, \Lambda, \hat{g},  \|g_0\|_{\alpha;M}, A$. 
we first consider the term $$\mathcal{Q}=w^{-1}*w*\hat{R}+w^{-1}*w^{-1}*\hat{\nabla}w*\hat{\nabla}w.$$
By assumption (2) and the matrix identity $A^{-1}-B^{-1}=B^{-1}(B-A)A^{-1}$, we have
\[|w^{-1}|_{0;M\times[0,T]}\leq K_2(\Lambda, \hat{g}) \quad\text{and}\quad [w^{-1}]_{\alpha,\frac{\alpha}{2};M\times[0,T]}\leq K(|w^{-1}|_{0;M\times[0,T]}, [w]_{\alpha,\frac{\alpha}{2};M\times[0,T]}) \leq K_2(\Lambda, \hat{g}, A).\]
This implies
    \begin{align}{\label{Q estimate}}
        &||\mathcal{Q}(w,\hat{\nabla}w)||_{\mathcal{C}_{1-\frac{\alpha}{2}}^{0,\gamma}(M\times (0,T])}\\
        &\leq K_2\ \Big(\|w^{-1}*w\|_{\mathcal{C}_{1-\frac{\alpha}{2}}^{0,\gamma}(M\times (0,T])} + \|w^{-1}*w^{-1}*\hat{\nabla}w*\hat{\nabla}w\|_{\mathcal{C}_{1-\frac{\alpha}{2}}^{0,\gamma}(M\times (0,T])}\Big)\notag\\
        &\leq K_2\ \Big(T^{1-\frac{\alpha}{2}}\ \|w^{-1}\|_{\alpha,\frac{\alpha}{2};M\times[0,T]} \|w\|_{\alpha,\frac{\alpha}{2};M\times[0,T]} + T^{\frac{\alpha}{2}}\ \|w^{-1}\|^2_{\alpha,\frac{\alpha}{2};M\times [0,T]} \|\hat{\nabla}w\|_{\mathcal{C}_{\frac{1}{2}-\frac{\alpha}{2}}^{0,\gamma}(M\times (0,T])}^2\Big)\notag\\
        &\leq K_2\ T^{\frac{\alpha}{2}}.\notag
    \end{align}
Similarly, using the facts $\alpha > \gamma$ and  $|(w-g_0)(x,t)|\leq t^{\frac{\alpha}{2}}\ [w - g_0]_	{\alpha,\frac{\alpha}{2};M\times [0,T]}$, and the matrix identity $A^{-1}-B^{-1}=B^{-1}(B-A)A^{-1}$, we have
    \begin{align}{\label{Hessian estimate}}
        &\|tr_w\hat{\nabla}^2w - tr_{g_0}\hat{\nabla}^2w\|_{\mathcal{C}_{1-\frac{\alpha}{2}}^{0,\gamma}(M\times (0,T])}\\
        &\notag = \sup_{\sigma\in(0,T]}\sigma^{1-\frac{\alpha}{2}}|(w^{-1}-g_0^{-1})*\hat{\nabla}^2w|_{0; M\times [\frac{\sigma}{2},\sigma]} + \sup_{\sigma\in (0,T]}\sigma^{1-\frac{\alpha}{2}+\frac{\gamma}{2}}[(w^{-1}-g_0^{-1})*\hat{\nabla}^2w]_{\gamma,\frac{\gamma}{2}; M\times [\frac{\sigma}{2},\sigma]}\\
        &\notag \leq K_2\Big(\sup_{\sigma\in(0,T]}\sigma^{1-\frac{\alpha}{2}}|w-g_0|_{0; M\times [\frac{\sigma}{2},\sigma]}|\hat{\nabla}^2w|_{0; M\times [\frac{\sigma}{2},\sigma]}  + \sup_{\sigma\in(0,T]}\sigma^{1-\frac{\alpha}{2}+\frac{\gamma}{2}}|w-g_0|_{0; M\times [\frac{\sigma}{2},\sigma]}[\hat{\nabla}^2w]_{\gamma,\frac{\gamma}{2}; M\times [\frac{\sigma}{2},\sigma]}\\
        &\notag\quad\quad\quad + \sup_{\sigma\in(0,T]}\sigma^{1-\frac{\alpha}{2}+\frac{\gamma}{2}}[w-g_0]_{\gamma,\frac{\gamma}{2}; M\times [\frac{\sigma}{2},\sigma]}|\hat{\nabla}^2w|_{0; M\times [\frac{\sigma}{2},\sigma]}\Big) \\
        &\notag \leq K_2\Big(\sup_{\sigma\in(0,T]}\sigma|\hat{\nabla}^2w|_{0; M\times [\frac{\sigma}{2},\sigma]} 
       		+ \sup_{\sigma\in(0,T]}\sigma^{1+\frac{\gamma}{2}}[\hat{\nabla}^2w]_{\gamma,\frac{\gamma}{2}; M\times [\frac{\sigma}{2},\sigma]} + \sup_{\sigma\in(0,T]}\sigma^{1-\frac{\alpha}{2}+\frac{\gamma}{2}}|\hat{\nabla}^2w|_{0; M\times [\frac{\sigma}{2},\sigma]}\Big) \\
        &\notag \leq K_2  \Big( T^{\frac{\alpha}{2}}\sup_{\sigma\in(0,T]}\sigma^{1-\frac{\alpha}{2}}|\hat{\nabla}^2w|_{0; M\times [\frac{\sigma}{2},\sigma]} \notag\quad + T^{\frac{\alpha}{2}}\sup_{\sigma\in(0,T]}\sigma^{1-\frac{\alpha}{2}+\frac{\gamma}{2}}[\hat{\nabla}^2w]_{\gamma,\frac{\gamma}{2}; M\times [\frac{\sigma}{2},\sigma]} \\
        &\notag\quad\quad\quad+ T^{\frac{\gamma}{2}}\sup_{\sigma\in(0,T]}\sigma^{1-\frac{\alpha}{2}}|\hat{\nabla}^2w|_{0; M\times [\frac{\sigma}{2},\sigma]}\Big) \\
        &\notag\leq K_2\ T^{\frac{\gamma}{2}}.
    \end{align}

 From (\ref{Q estimate}), (\ref{Hessian estimate}) and Theorem \ref{theorem linear equation}, we conclude that there is a unique solution $\eta\in\mathcal{X}_{2,\gamma}^{(\alpha)}(M\times[0,T])$ to the linear system (\ref{linear Ricci DeT}) and a positive constant $K_1 = K_1(M,  \hat{g}, \|g_0\|_{\alpha})$ such that
\begin{align*}
		 &\|\eta\|_{\mathcal{X}_{2,\gamma}^{(\alpha)}(M\times[0,T])} \\
		 \notag &\leq K_1\left( ||tr_w\hat{\nabla}^2w - tr_{g_0}\hat{\nabla}^2w + \mathcal{Q}(w,\hat{\nabla}w)||_{\mathcal{C}_{1-\frac{\alpha}{2}}^{0,\gamma}(M\times (0,T])} +  \|g_0\|_{\alpha;M} \right) \\
		 & \leq K_1( K_2\ T^{\frac{\gamma}{2}} + \|g_0\|_{\alpha;M}). \notag
\end{align*}
Next, note that the bound for the semi-H\"older norm $[\eta]_{\alpha,\frac{\alpha}{2};M\times[0,T]}$ implies that
\begin{align*}
	\|\eta(x,t) - g_0(x)\|_{0;M\times [0,T]} \leq K(M, \Lambda, \hat{g}, \|g_0\|_{\alpha}, A)\ T^{\frac{\alpha}{2}}.
\end{align*}
Thus if $T^* = T^*(M, \Lambda, \hat{g}, \|g_0\|_{\alpha}, A)$ is sufficiently small, then the second conclusion of the proposition also holds. 
\end{proof}
\bigskip

\subsection{Short time existence and uniqueness to Ricci-DeTurck flow}\

\bigskip

We now prove the short time existence for the Ricci-DeTurck flow (\ref{Ricci DeT}) by employing the Banach fixed point theorem. Let $\Lambda > 2$ and $A> 10\Lambda K_1\|g_0\|_{\alpha;M}$ be large constants, where $K_1$ is the constant given in Proposition \ref{prop linear Ricci DeT}. We define a closed subset in $\mathcal{X}_{2,\gamma}^{(\alpha)}(M\times [0,T])$ by
	$$\mathcal{B} := \{w \in\mathcal{X}_{2,\gamma}^{(\alpha)} |\ w|_{t=0} = g_0,\ \Lambda g_0(\cdot)\geq w(\cdot,t)\geq\frac{1}{\Lambda}g_0(\cdot),\ \|w\|_{\mathcal{X}_{2,\gamma}^{(\alpha)}}\leq A\}.$$
The subset $\mathcal{B}$ is non-empty by our choice of $A$ provided that $T = T(M, \Lambda, \hat{g}, \|g_0\|_{\alpha}, A)$ is chosen sufficiently small. We next define an operator $\mathcal{R}:\mathcal{B}\to\mathcal{X}_{2,\gamma}^{(\alpha)}$ by
\begin{align}
	\eta := \mathcal{R}(w),	
\end{align}
where $\eta$ is the unique solution to the system (\ref{linear Ricci DeT}) in $\mathcal{X}_{2,\gamma}^{(\alpha)}$. The operator $\mathcal{R}$ is well defined by Proposition \ref{prop linear Ricci DeT}. Moreover, we can further set $T = T(M, \Lambda, \hat{g}, \|g_0\|_{\alpha}, A)$ sufficiently small such that $\|\eta\|_{\mathcal{X}_{2,\gamma}^{(\alpha)}}\leq A$ and $\Lambda g_0(\cdot)\geq \eta(\cdot,t)\geq\frac{1}{\Lambda}g_0(\cdot)$ by Proposition \ref{prop linear Ricci DeT}. Consequently we have $\mathcal{R}(\mathcal{B})\subset\mathcal{B}$ by our choice of $T$.
\bigskip

\begin{proposition}{\label{prop contraction}}
If $T = T(M, \Lambda, \hat{g}, \|g_0\|_{\alpha}, A)$ is chosen sufficiently small, the operator $\mathcal{R}$ is a contraction mapping.
\end{proposition}

\begin{proof}
In the sequel, $K$ will denote a constant depending only on $M,\Lambda, A, \hat{g},  \|g_0\|_{\alpha;M}$. Let $w_1, w_2\in\mathcal{B}$ and write $\eta_i :=\mathcal{R}(w_i)$ for $i=1,2$. Then $\eta := \eta_1 - \eta_2$ solves the following system:

\begin{align}
    \begin{cases}
    \frac{\partial}{\partial t}\eta - tr_{g_0}\hat{\nabla}^2\eta = \tilde{\mathcal{Q}}   &\text{on}\quad M\times (0,T]\\
    \eta(x,0) = 0 \quad &\text{on}\quad M,
    \end{cases}
\end{align}
where
$$\tilde{\mathcal{Q}} := tr_{w_1}\hat{\nabla}^2w_1 - tr_{w_2}\hat{\nabla}^2w_2 - tr_{g_0}\hat{\nabla}^2(w_1-w_2) +  \mathcal{Q}(w_1, \hat{\nabla}w_1)-\mathcal{Q}(w_2, \hat{\nabla}w_2).$$
Then Theorem \ref{theorem linear equation} with $\eta_0 = 0$ asserts that
\begin{align}{\label{contraction estimate}}
    \|\eta\|_{\mathcal{X}_{2,\gamma}^{(\alpha)}}\leq  K||\tilde{\mathcal{Q}}||_{\mathcal{C}_{1-\frac{\alpha}{2}}^{0,\gamma}(M\times (0,T])}.
\end{align}

We first derive the estimate for the term $\mathcal{Q}(w_1, \hat{\nabla}w_1)-\mathcal{Q}(w_2, \hat{\nabla}w_2)$. Recall that $$\mathcal{Q}(w,\hat{\nabla}w)=w^{-1}*w*\hat{R}+w^{-1}*w^{-1}*\hat{\nabla}w*\hat{\nabla}w,$$ thus we can write the difference as
\begin{align*}
	&\mathcal{Q}(w_1, \hat{\nabla}w_1)-\mathcal{Q}(w_2, \hat{\nabla}w_2)\\
	&= (w_1^{-1} - w_2^{-1})* w_1 * \hat{R}\ + w_2^{-1} * (w_1 - w_2) * \hat{R} \\
	&\quad\quad + (w_1^{-1} - w_2^{-1})*w_1^{-1}*\hat{\nabla}w_1*\hat{\nabla}w_1\ +\ w_2^{-1} * (w_1^{-1} - w_2^{-1})*\hat{\nabla}w_1*\hat{\nabla}w_1\\ 
	&\quad\quad + w_2^{-1}*w_2^{-1}* (\hat{\nabla}w_1 - \hat{\nabla}w_2) *\hat{\nabla}w_1\ +\ w_2^{-1}*w_2^{-1}*\hat{\nabla}w_2 * (\hat{\nabla}w_1 - \hat{\nabla}w_2).  
\end{align*}

Using the matrix identity $A^{-1}-B^{-1}=B^{-1}(B-A)A^{-1}$ and the fact that $\|w_i\|_{\mathcal{X}_{2,\gamma}^{(\alpha)}}\leq A$, we have for each $\sigma\in (0,T]$,
\begin{align*}
	&|\mathcal{Q}(w_1, \hat{\nabla}w_1) - \mathcal{Q}(w_2, \hat{\nabla}w_2)|_{0;M\times[\frac{\sigma}{2},\sigma]}	\\
	&\notag \leq K\Big(|w_1-w_2|_{0;M\times[\frac{\sigma}{2},\sigma]} + \sigma^{\alpha-1}|w_1^{-1}-w_2^{-1}|_{0;M\times[\frac{\sigma}{2},\sigma]} + \sigma^{\frac{\alpha}{2}-\frac{1}{2}}|\hat{\nabla}w_1-\hat{\nabla}w_2|_{0;M\times[\frac{\sigma}{2},\sigma]}\Big)\\
	&\notag \leq K\Big(\sigma^{\alpha-1}|w_1-w_2|_{0;M\times[\frac{\sigma}{2},\sigma]} + \sigma^{\frac{\alpha}{2}-\frac{1}{2}}|\hat{\nabla}w_1-\hat{\nabla}w_2|_{0;M\times[\frac{\sigma}{2},\sigma]}\Big).
\end{align*}
This implies
\begin{align}{\label{contraction Q}}
	&\sigma^{1-\frac{\alpha}{2}}|\mathcal{Q}(w_1, w_1^{-1}, \hat{\nabla}w_1) - \mathcal{Q}(w_2, w_2^{-1}, \hat{\nabla}w_2)|_{0;M\times[\frac{\sigma}{2},\sigma]}\\
	&\notag\leq K\sigma^{\frac{\alpha}{2}}\sup_{\sigma\in(0,T]}|w_1-w_2|_{0;M\times[\frac{\sigma}{2},\sigma]} + \sigma^{\frac{\alpha}{2}}\sup_{\sigma\in(0,T]}\sigma^{\frac{1}{2}-\frac{\alpha}{2}}|\hat{\nabla}w_1-\hat{\nabla}w_2|_{0;M\times[\frac{\sigma}{2},\sigma]}\\
	&\notag\leq  KT^{\frac{\alpha}{2}}\|w_1-w_2\|_{\mathcal{X}_{2,\gamma}^{(\alpha)}}.
\end{align}
On the other hand,
\begin{align*}
	&[\mathcal{Q}(w_1, \hat{\nabla}w_1) - \mathcal{Q}(w_2, \hat{\nabla}w_2)]_{\gamma,\frac{\gamma}{2};M\times[\frac{\sigma}{2},\sigma]}	\\
	&\notag \leq K \Big([w_1-w_2]_{\gamma,\frac{\gamma}{2};M\times[\frac{\sigma}{2},\sigma]} + |w_1-w_2|_{0;M\times[\frac{\sigma}{2},\sigma]} + \sigma^{\alpha-1}[w_1^{-1}-w_2^{-1}]_{\gamma,\frac{\gamma}{2};M\times[\frac{\sigma}{2},\sigma]}\\
	&\notag\quad+\sigma^{-\frac{\gamma}{2} + \alpha -1}|w_1^{-1}-w_2^{-1}|_{0;M\times[\frac{\sigma}{2},\sigma]} + \sigma^{\frac{\alpha}{2}-\frac{1}{2}}[\hat{\nabla}w_1-\hat{\nabla}w_2]_{\gamma,\frac{\gamma}{2};M\times[\frac{\sigma}{2},\sigma]} + \sigma^{-\frac{\gamma}{2}+\frac{\alpha}{2}-\frac{1}{2}}|\hat{\nabla}w_1-\hat{\nabla}w_2|_{0;M\times[\frac{\sigma}{2},\sigma]}\Big).
\end{align*}
Since $\alpha>\gamma$, the H\"older semi-norm $[w_1-w_2]_{\gamma,\frac{\gamma}{2};M\times[\frac{\sigma}{2},\sigma]}$ is controlled by $[w_1-w_2]_{\alpha,\frac{\alpha}{2};M\times[\frac{\sigma}{2},\sigma]}$, thus
\begin{align}{\label{contraction Q 2}}
		&\sigma^{1-\frac{\alpha}{2}+\frac{\gamma}{2}}[\mathcal{Q}(w_1, \hat{\nabla}w_1) - \mathcal{Q}(w_2, \hat{\nabla}w_2)]_{\gamma,\frac{\gamma}{2};M\times[\frac{\sigma}{2},\sigma]}\leq  KT^{\frac{\alpha}{2}}\|w_1-w_2\|_{\mathcal{X}_{2,\gamma}^{(\alpha)}}
\end{align}
for each $\sigma\in (0,T]$. Putting (\ref{contraction Q}) and (\ref{contraction Q 2}) together we obtain
\begin{align}{\label{contraction Q 3}}
	\|\mathcal{Q}(w_1,  \hat{\nabla}w_1) - \mathcal{Q}(w_2,  \hat{\nabla}w_2)\|_{\mathcal{C}_{1-\frac{\alpha}{2}}^{0,\gamma}(M\times (0,T])}\leq KT^{\frac{\alpha}{2}}\|w_1-w_2\|_{\mathcal{X}_{2,\gamma}^{(\alpha)}}.
\end{align}

\bigskip
Next, we derive the estimate for the term $\ tr_{w_1}\hat{\nabla}^2w_1 - tr_{w_2}\hat{\nabla}^2w_2- tr_{g_0}\hat{\nabla}^2(w_1-w_2)$. We write it as
\begin{align}{\label{contraction w}}
   \notag &tr_{w_1}\hat{\nabla}^2w_1 - tr_{w_2}\hat{\nabla}^2w_2- tr_{g_0}\hat{\nabla}^2(w_1-w_2) \\
   &= (w_1^{-1}-w_2^{-1})*\hat{\nabla}^2w_1 + (w_2^{-1}-g_0^{-1})*\hat{\nabla}^2(w_1-w_2).
\end{align}
For the first term in (\ref{contraction w}), we use the matrix identity $A^{-1}-B^{-1}=B^{-1}(B-A)A^{-1}$ to estimate
\begin{align*}
	|(w_1^{-1}-w_2^{-1})*\hat{\nabla}^2w_1|_	{0;M\times[\frac{\sigma}{2},\sigma]} &\leq K|w_1-w_2|_	{0;M\times[\frac{\sigma}{2},\sigma]}|\hat{\nabla}^2w_1|_	{0;M\times[\frac{\sigma}{2},\sigma]}\\
	&\leq K\sigma^{\frac{\alpha}{2}-1}|w_1-w_2|_	{0;M\times[\frac{\sigma}{2},\sigma]}.
\end{align*}
Since $w_1(\cdot, 0) = w_2(\cdot, 0)$, we have
	 \[|(w_1-w_2)(x,t)|\leq t^{\frac{\alpha}{2}}\ [w_1-w_2]_	{\alpha,\frac{\alpha}{2};M\times [0,\sigma]}\]
for any $(x,t)\in M\times [0,\sigma]$. Consequently,
\begin{align}{\label{contraction w 2}}
	\sigma^{1-\frac{\alpha}{2}}|(w_1^{-1}-w_2^{-1})*\hat{\nabla}^2w_1|_	{0;M\times[\frac{\sigma}{2},\sigma]} &\leq K\sigma^{\frac{\alpha}{2}}\sigma^{-\frac{\alpha}{2}}|w_1-w_2|_{0;M\times[\frac{\sigma}{2},\sigma]}\\
	&\notag\leq K\sigma^{\frac{\alpha}{2}}\sup_{\tau\in (0,T]}\ \tau^{-\frac{\alpha}{2}}|w_1-w_2|_{0;M\times[\frac{\tau}{2},\tau]}\\
	&\notag\leq K\sigma^{\frac{\alpha}{2}}\ \|w_1-w_2\|_{\mathcal{X}_{2,\gamma}^{(\alpha)}}
\end{align}
for any $\sigma\in (0,T]$. On the other hand,
\begin{align*}
	&[(w_1^{-1}-w_2^{-1})*\hat{\nabla}^2w_1]_	{\gamma,\frac{\gamma}{2};M\times[\frac{\sigma}{2},\sigma]} \\
	&\leq K\Big(|w_1-w_2|_{0;M\times[\frac{\sigma}{2},\sigma]}[\hat{\nabla}^2w_1]_{\gamma,\frac{\gamma}{2};M\times[\frac{\sigma}{2},\sigma]} + [w_1-w_2]_{\gamma,\frac{\gamma}{2};M\times[\frac{\sigma}{2},\sigma]}|\hat{\nabla}^2w_1|_{0;M\times[\frac{\sigma}{2},\sigma]}\Big)\\
	&\leq K\Big(\sigma^{-\frac{\gamma}{2}+\frac{\alpha}{2}-1}|w_1-w_2|_{0;M\times[\frac{\sigma}{2},\sigma]} + \sigma^{\frac{\alpha}{2}-1}[w_1-w_2]_{\gamma,\frac{\gamma}{2};M\times[\frac{\sigma}{2},\sigma]}\Big).
\end{align*}
for any $\sigma\in (0,T]$. By an argument similar to (\ref{contraction w 2}) and the fact $\alpha>\gamma$, we have
\begin{align}{\label{contraction w 3}}
	\sigma^{1-\frac{\alpha}{2}+\frac{\gamma}{2}}[(w_1^{-1}-w_2^{-1})*\hat{\nabla}^2w_1]_	{\gamma,\frac{\gamma}{2};M\times[\frac{\sigma}{2},\sigma]} &\leq K\Big(\sigma^{\frac{\gamma}{2}}\sigma^{-\frac{\gamma}{2}}|w_1-w_2|_{0;M\times[\frac{\sigma}{2},\sigma]} + \sigma^{\frac{\alpha}{2}}[w_1-w_2]_{\gamma,\frac{\gamma}{2};M\times[\frac{\sigma}{2},\sigma]}\Big)\\
	&\leq K\sigma^{\frac{\gamma}{2}}\ \|w_1-w_2\|_{\mathcal{X}_{2,\gamma}^{(\alpha)}}\notag
\end{align}
for any $\sigma\in (0,T]$.
For the second term in (\ref{contraction w}), we have
\begin{align}{\label{contraction w 4}}
	\sigma^{1-\frac{\alpha}{2}}|(w_2^{-1}-g_0^{-1})*\hat{\nabla}^2(w_1-w_2)|_	{0;M\times[\frac{\sigma}{2},\sigma]} &\leq K\sigma^{1-\frac{\alpha}{2}}|w_2-g_0|_{0;M\times[\frac{\sigma}{2},\sigma]}|\hat{\nabla}^2(w_1-w_2)|_{0;M\times[\frac{\sigma}{2},\sigma]}\\
	\notag &\leq K\sigma\ |\hat{\nabla}^2(w_1-w_2)|_{0;M\times[\frac{\sigma}{2},\sigma]}\\
	\notag &\leq K\sigma^{\frac{\alpha}{2}}\sup_{\tau\in (0,T]}\tau^{1-\frac{\alpha}{2}} |\hat{\nabla}^2(w_1-w_2)|_{0;M\times[\frac{\tau}{2},\tau]}\\
	\notag &\leq K\sigma^{\frac{\alpha}{2}} \|w_1-w_2\|_{\mathcal{X}_{2,\gamma}^{(\alpha)}},
\end{align}
for any $\sigma\in (0,T]$, since $|w_2-g_0|_{0;M\times[\frac{\sigma}{2},\sigma]}\leq \sigma^{\frac{\alpha}{2}} [w_2]_{\alpha,\frac{\alpha}{2};M\times[0,\sigma]} \leq A\sigma^{\frac{\alpha}{2}}$.  Moreover,
\begin{align}{\label{contraction w 5}}
	&\sigma^{1-\frac{\alpha}{2}+\frac{\gamma}{2}}[(w_2^{-1}-g_0^{-1})*\hat{\nabla}^2(w_1-w_2)]_	{\gamma,\frac{\gamma}{2};M\times[\frac{\sigma}{2},\sigma]} \\
	\notag &\leq K\Big(\sigma^{1-\frac{\alpha}{2}+\frac{\gamma}{2}}|w_1-g_0|_{0;M\times[\frac{\sigma}{2},\sigma]}[\hat{\nabla}^2(w_1-w_2)]_{\gamma,\frac{\gamma}{2};M\times[\frac{\sigma}{2},\sigma]}\\
	\notag &\quad\quad\quad + \sigma^{1-\frac{\alpha}{2}+\frac{\gamma}{2}}[w_1-g_0]_{\gamma,\frac{\gamma}{2};M\times[\frac{\sigma}{2},\sigma]}|\hat{\nabla}^2(w_1-w_2)|_{0;M\times[\frac{\sigma}{2},\sigma]}\Big)\\
	\notag &\leq K\Big(\sigma^{1-\frac{\alpha}{2}+\frac{\gamma}{2}}|\hat{\nabla}^2(w_1-w_2)|_{0;M\times[\frac{\sigma}{2},\sigma]} + \sigma^{1+\frac{\gamma}{2}}[\hat{\nabla}^2(w_1-w_2)]_{\gamma,\frac{\gamma}{2};M\times[\frac{\sigma}{2},\sigma]} \Big)\\
	\notag &\leq K \Big(\sigma^{\frac{\gamma}{2}} \sup_{\tau\in (0,T]}\tau^{1-\frac{\alpha}{2}}|\hat{\nabla}^2(w_1-w_2)|_{0;M\times[\frac{\tau}{2},\tau]} + \sigma^{\frac{\alpha}{2}} \sup_{\tau\in (0,T]}\tau^{1-\frac{\alpha}{2}+\frac{\gamma}{2}}\ [\hat{\nabla}^2(w_1-w_2)]_{\gamma,\frac{\gamma}{2};M\times[\frac{\tau}{2},\tau]} \Big)\\
	\notag &\leq K\sigma^{\frac{\gamma}{2}} \|w_1-w_2\|_{\mathcal{X}_{2,\gamma}^{(\alpha)}}.
\end{align}
Putting (\ref{contraction w}) to (\ref{contraction w 4}) together, we obtain that
\begin{align}{\label{contraction w 6}}
	\|tr_{w_1}\hat{\nabla}^2w_1 - tr_{w_2}\hat{\nabla}^2w_2- tr_{g_0}\hat{\nabla}^2(w_1-w_2)\|_{\mathcal{C}_{1-\frac{\alpha}{2}}^{0,\gamma}(M\times (0,T])}\leq KT^{\frac{\gamma}{2}}\|w_1-w_2\|_{\mathcal{X}_{2,\gamma}^{(\alpha)}}
\end{align}
Hence, putting (\ref{contraction estimate}), (\ref{contraction Q 3}) and (\ref{contraction w 5}) together, we obtain
\begin{align}
    ||\eta_1 - \eta_2||_{\mathcal{X}_{2,\gamma}^{(\alpha)}}\leq  KT^{\frac{\gamma}{2}}\|w_1-w_2\|_{\mathcal{X}_{2,\gamma}^{(\alpha)}}.
\end{align}
If $T = T(M, \Lambda, \hat{g}, \|g_0\|_{\alpha}, A)$ is chosen sufficiently small, we then have
\begin{align*}
	\|\mathcal{R}(w_1) - \mathcal{R}(w_2)\|_{\mathcal{X}_{2,\gamma}^{(\alpha)}} \leq \frac{1}{2}\|w_1-w_2\|_{\mathcal{X}_{2,\gamma}^{(\alpha)}}.		
\end{align*}
This proves the proposition.
\end{proof}

\bigskip
Therefore, by the Banach fixed point theorem we have proved

\begin{theorem}{\label{thm existence Ricci DeT}}
There exist $K = K(M, \hat{g},  \|g_0\|_{\alpha;M})$ and $T = T(M, \hat{g},  \|g_0\|_{\alpha;M})$ such that the following holds:\\

There is a unique solution $g\in\mathcal{X}_{2,\gamma}^{(\alpha)}(M\times[0,T])$ to the Ricci-DeTurck flow (\ref{Ricci DeT}) such that
\begin{itemize}
\item $g(\cdot,t)$ is a family of Riemannian metrics;
\item $\|g\|_{\mathcal{X}_{2,\gamma}^{(\alpha)}(M\times[0,T])}\leq K$.
\end{itemize}
\end{theorem}

\bigskip
Next, we can improve the regularity of $g$ by bootstrapping.

\begin{corollary}{\label{thm existence high Ricci DeT}}
Let $k\geq 2$ be given. There exist $K = K(M, k, \hat{g},  \|g_0\|_{\alpha;M})$ and $T = T(M, \hat{g},  \|g_0\|_{\alpha;M})$ such that the following holds:\\

There is a unique solution $g\in\mathcal{X}_{k,\gamma}^{(\alpha)}(M\times[0,T])$ to the Ricci-DeTurck flow (\ref{Ricci DeT}) such that
\begin{itemize}
\item $g(\cdot,t)$ is a family of Riemannian metrics;
\item $\|g\|_{\mathcal{X}_{k,\gamma}^{(\alpha)}(M\times[0,T])}\leq K$.
\end{itemize}
\end{corollary}
\bigskip

\begin{proof}
 We will prove the corollary by induction on $k$. It is clear from Theorem \ref{thm existence Ricci DeT} that the assertion holds when $k=2$. Let us assume that $g\in\mathcal{X}_{k,\gamma}^{(\alpha)}(M\times [0,T])$ for some $k\geq 2$. In the sequel, $K$ will denote a constant depending only on $M, k, \hat{g},  \|g_0\|_{\alpha;M}, \|g\|_{\mathcal{X}_{k,\gamma}^{(\alpha)}}$. Firstly, note that by the matrix identity $A^{-1}-B^{-1}=B^{-1}(B-A)A^{-1}$, the term $\|g^{-1}\|_{\mathcal{C}^{k-1,\gamma}_{0}(M\times (0,T])}$ is controlled by $\|g\|_{\mathcal{C}^{k-1,\gamma}_{0}(M\times (0,T])}$. Moreover, we observe that
	\[ \|g\|_{\mathcal{C}^{k-1,\gamma}_{0}(M\times (0,T])} = \|g\|_{\mathcal{C}^{0,\gamma}_{0}(M\times (0,T])} + \|\hat{\nabla}g\|_{\mathcal{C}^{k-2,\gamma}_{\frac{1}{2}}(M\times (0,T])}.\]
	Since $\alpha>\gamma$,
	\[ \|g\|_{\mathcal{C}^{0,\gamma}_{0}(M\times (0,T])} \leq K\|g\|_{\alpha,\frac{\alpha}{2};M\times [0,T]},\]
	and it follows from Lemma \ref{properties of weighted space} that
	\[ \|\hat{\nabla}g\|_{\mathcal{C}^{k-2,\gamma}_{\frac{1}{2}}(M\times (0,T])} \leq KT^{\frac{\alpha}{2}}\|\hat{\nabla}g\|_{\mathcal{C}^{k-1,\gamma}_{\frac{1}{2}-\frac{\alpha}{2}}(M\times (0,T])},\]
	we have
	\[ \|g\|_{\mathcal{C}^{k-1,\gamma}_{0}(M\times (0,T])} \leq K.\]
  Since $g$ solves the Ricci-DeTurck system (\ref{Ricci DeT}) on $M\times (0,T]$, by Lemma \ref{properties of weighted space},   the induction hypothesis, and the above estimates, we have
	\begin{align*}
	\|g^{-1}*g^{-1}*\hat{\nabla}g*\hat{\nabla}g\|_{\mathcal{C}^{k-1,\gamma}_{1-\frac{\alpha}{2}}(M\times (0,T])} &\leq K T^{\frac{\alpha}{2}}\ 	\|g^{-1}\|^2_{\mathcal{C}^{k-1,\gamma}_{0}(M\times (0,T])}\|\hat{\nabla}g\|^2_{\mathcal{C}^{k-1,\gamma}_{\frac{1}{2}-\frac{\alpha}{2}}(M\times (0,T])}\\
	&\leq K
	\end{align*}
and
\begin{align*}
	\|g^{-1}*g*\hat{R}\|_{\mathcal{C}^{k-1,\gamma}_{1-\frac{\alpha}{2}}(M\times (0,T])} &\leq K T^{1-\frac{\alpha}{2}}\ \|g^{-1}\|_{\mathcal{C}^{k-1,\gamma}_{0}(M\times (0,T])}\|g\|_{\mathcal{C}^{k-1,\gamma}_{0}(M\times (0,T])}	\\
	&\leq K.
\end{align*}
Theorem \ref{theorem linear equation} then implies that 
 \begin{align*}
 	\|g\|_{\mathcal{X}_{k+1,\gamma}^{(\alpha)}(M\times[0,T])}&\leq K\big(||\mathcal{Q}(g,\hat{\nabla}g)||_{\mathcal{C}^{k-1,\gamma}_{1-\frac{\alpha}{2}}(M\times (0,T])} + \|g_0\|_{\alpha} \big)\\
 	&\leq K.
 \end{align*}
From these, the assertion follows.

    \end{proof}

\bigskip

\bigskip

\subsection{Short time existence to the Ricci flow}\

\bigskip

\

We will prove the short time existence of solution to the Ricci flow with H\"older-continuous initial metrics in this subsection. More specifically, we show that any solution of the Ricci-DeTurck flow with H\"older-continuous initial data gives rise to a solution of the Ricci flow with H\"older-continuous initial data.\\

  For any $\alpha\in (0,1)$ and $k\geq 0$, we define by $C^{k, \alpha}(M;M)$ the space of $C^{k, \alpha}$ maps $f: M\to M$. Recall in section 2 that $\{U_{\mu}, \varphi_{\mu}\}_{{\mu}=1,..,m}$ is a fixed set of coordinate charts on $M$. On the coordinate chart $U_{\mu}$, a map $f:M\to M$ has components $(f_{\mu}^1,\dots, f_{\mu}^n)$. The H\"older norm for maps $f\in C^{k, \alpha}(M;M)$ is measured with respect to the fixed charts  $\{U_{\mu}, \varphi_{\mu}\}$. More precisely,  given  $f\in C^{k, \alpha}(M;M)$,  we define the associated $C^{k, \alpha}$ norm to be
  \[ \|f\|_{C^{k, \alpha}(M;M)} = \sum_{\mu}\sum_{r=1}^n |f_{\mu}^r|_{C^{k, \alpha}(\varphi_{\mu}(U_{\mu}))}.\]

In the remainder of this subsection, we will consider the ODE
\begin{align}{\label{DeT ODE}}
	\begin{cases}
    &\frac{\partial}{\partial t}\psi_t = -W(\psi_t,\ t)\\
    &\psi_T = \id
    \end{cases},
\end{align}
where  $W$ is the DeTurck vector field defined by $W^k=g^{ij}((\Gamma_{g})_{ij}^k-(\hat{\Gamma}_{\hat{g}})_{ij}^k)$, and $g\in\mathcal{X}_{k,\gamma}^{(\alpha)}$ is the solution to the Ricci-DeTurck flow on $M\times (0,T]$.
Note that the vector field $W$ is undefined at $t=0$. Nevertheless we can show that the one-parameter family of maps $\{\psi_t\}$ generated by (\ref{DeT ODE}) can be extended to a map $\psi_0$ at $t=0$.
In the sequel of the subsection, $K$ will denote a constant  depending only on $M, \alpha, \hat{g},  \|g_0\|_{\alpha;M}$, $\|g\|_{\mathcal{X}_{k,\gamma}^{(\alpha)}}$.

\begin{lemma}{\label{lem DeT}}
Let  $g\in\mathcal{X}_{k,\gamma}^{(\alpha)}$ be the solution to the Ricci-DeTurck flow on $M\times (0,T]$, and $W$ be the DeTurck vector field defined by $W^k=g^{ij}((\Gamma_{g})_{ij}^k-(\hat{\Gamma}_{\hat{g}})_{ij}^k)$. Then the one-parameter family of maps $\{\psi_t\}$ generated by (\ref{DeT ODE}) can be extended to a map $\psi_0$ at $t=0$ in $C^{1, \beta}(M;M)$ for any $\beta<\alpha$ and satisfies 
$$\|\psi_t\|_{C^{1, \beta}(M;M)} \leq K$$
for each $t\in [0,T]$. Moreover, $\psi_t$ remains a diffeomorphism for $t\in [0,T]$ if $T$ is sufficiently small
\end{lemma}

\begin{proof}
Since the $W=g^{-1}*g^{-1}*\hat{\nabla}g$, we have the estimate $\|W\|_{\mathcal{C}^{k-1,\gamma}_{\frac{1}{2}-\frac{\alpha}{2}}(M\times(0,T])}\leq K$. Note that the ODE (\ref{DeT ODE}) is uniquely solvable on $M\times (0,T]$, we would like to show that the solution $\psi_t$ can be extended to $M\times [0,T]$ in $C^{1, \beta}(M;M)$ for any $\beta<\alpha$.\\

We now work in local coordinates. Suppose that $\psi_t$ has components $\psi_t^k$ in a chart $U$,  we see that $\partial_i\psi_t^k$ satisfies the ODE
\begin{align}{\label{dpsi ODE}}
\frac{\partial}{\partial t}\partial_i\psi_t^k = -\partial_l W^k(\psi_t, t)\ \partial_i\psi_t^l,\quad i=1,..,n
\end{align}
for each $t\in (0,T)$.
Let $\beta\in (0,\alpha)$ be given and define $\beta' := \frac{\beta+\alpha}{2}$. We claim that $\|\psi_t\|_{C^{1, \beta'}(M;M)}$ has a uniform bound for every $t\in (0,T)$.  Define
	 \[f(t) := \|\psi_t\|_{C^{1,0}(M;M)},\] 
	then from (\ref{dpsi ODE}) we have
	\begin{align}{\label{dpsi ODE 1}}
		 \frac{\partial}{\partial t}\partial_i\psi_t^k  \leq K\ t^{\frac{\alpha}{2}-1} f(t).
	\end{align}
	To integrate the above inequality, it is possible that the integral curves pass into different charts. Fix $x\in M$ and $t\in (0,T]$. Suppose that $\psi_T(x)\in U$ and $\psi_t(x)\in V$. Since there are only finitely many coordinate charts, we may assume without loss of generality that $U\cap V\neq\emptyset$ and there is $s\in (t,T)$ such that $\psi_{\tau}(x)\in U$ for all $\tau\in [s,T]$ and $\psi_{\tau}(x)\in V$ for all $\tau\in [t,s]$.
	 By integrating the inequality (\ref{dpsi ODE 1}) from $s$ to $T$, and by noting that $\partial_i\psi_T^k = \delta_i^k$, we have
	 \begin{align}{\label{uniform estimate psi 1}}
	 	|\partial_i\psi_s^k(x)-\delta_i^k|\leq  \int_s^T K\ \tau^{\frac{\alpha}{2}-1} f(\tau)\ d\tau.
	 \end{align}
	 Moreover, on the intersection $U\cap V$ we may write $\partial_i\psi_{\tau}^k = \partial_i\tilde{\psi}_{\tau}^l\frac{\partial y_k}{\partial z_l}$, where $(y_k)$ and $(z_l)$ stand for the local coordinates on the charts $U$ and $V$ respectively, and $\tilde{\psi}_{\tau}^l$ stand for the components of $\psi_{\tau}$ on $V$. Since $\partial_i\tilde{\psi}_{\tau}^l$ satisfies (\ref{dpsi ODE 1}), it also satisfies a similar inequality as above by integrating (\ref{dpsi ODE 1}) from $t$ to $s$. This implies that
	 \begin{align}{\label{uniform estimate psi 2}}
	 	|\partial_i\tilde{\psi}_t^l(x) - \partial_i\tilde{\psi}_s^l(x)|\leq \int_t^s K\ \tau^{\frac{\alpha}{2}-1} f(\tau)\ d\tau.
	 \end{align}
	 From (\ref{uniform estimate psi 1}) and (\ref{uniform estimate psi 2}) we obtain
	 \begin{align}{\label{uniform estimate psi}}
	 	|\partial_i\tilde{\psi}_t^l(x)| &\leq K + \int_t^T K\ \tau^{\frac{\alpha}{2}-1} f(\tau)\ d\tau.
	 \end{align}
	 We can similarly obtain bounds for $|\tilde{\psi}_t^l(x)|$ by integrating the ODE (\ref{DeT ODE}). Since $x\in M$ is arbitrary, we obtain
\begin{align}{\label{integral C^1 bound psi}}
	f(t)\leq K + \int_t^T K\ \tau^{\frac{\alpha}{2}-1} f(\tau)\ d\tau.
\end{align}
Gr\"onwall's lemma then implies that
\begin{align}{\label{C^1 bound psi}}
	\|\psi_t\|_{C^{1,0}(M;M)} \leq K\exp(KT^{\frac{\alpha}{2}})	
\end{align}
for any $t\in (0,T]$.

Next, let $x,y\in M$ and fix $t\in (0,T)$. To bound the $C^{1,\beta'}$ norm for $\psi_t$, we may assume without loss of generality that $\psi_{\tau}(x)$ stays in the same chart with $\psi_{\tau}(y)$ for $\tau\in [t,T)$ using the uniform $C^1$ bound (\ref{C^1 bound psi}). Note that it is still possible that the integral curves $\psi_{\tau}(x)$ and $\psi_{\tau}(y)$ may pass into different charts. By (\ref{dpsi ODE}) we have
\begin{align}
	&\frac{\partial}{\partial \tau}\Big(\frac{\partial_i\psi_{\tau}^k(x) - \partial_i\psi_{\tau}^k(y)}{|x-y|^{\beta'}}\Big)\\
	 =\ \notag 	&- \frac{\partial_lW^k(\psi_{\tau}(x), \tau) - \partial_lW^k(\psi_{\tau}(y),\tau)}{|x-y|^{\beta'}}\partial_i\psi_{\tau}^l(x) - \partial_lW^k(\psi_{\tau}(y), \tau)\frac{\partial_i\psi_{\tau}^l(x) - \partial_i\psi_{\tau}^l(y)}{|x-y|^{\beta'}}
\end{align}
for each $\tau\in (t,T)$.
 By Lemma \ref{properties of weighted space}, we have $\|W\|_{\mathcal{C}^{k-1,\beta'}_{\frac{1}{2}-\frac{\alpha}{2}}(M\times(0,T])}\leq K$. Thus the bound for $\hat{\nabla}W$ and (\ref{C^1 bound psi}) imply
  \begin{align*}
  	|\partial_lW^k(\psi_{\tau}(x),\tau) - \partial_lW^k(\psi_{\tau}(y),\tau)| 
  	&\leq K\ \tau^{-1+\frac{\alpha-\beta'}{2}}\ |\psi_{\tau}(x) - \psi_{\tau}(y)|^{\beta'} \\
  	&\leq K\ \tau^{-1+\frac{\alpha-\beta'}{2}}\ |x-y|^{\beta'}
  \end{align*} 
  and
  \[ |\partial_lW^k(\psi_{\tau}(y), \tau)| \leq K\tau^{-1+\frac{\alpha}{2}}.\]
  Hence
  \begin{align*}
  \frac{\partial}{\partial \tau}\Big(\frac{\partial_i\psi_{\tau}^k(x) - \partial_i\psi_{\tau}^k(y)}{|x-y|^{\beta'}}\Big) &\leq K\tau^{-1+\frac{\alpha-\beta'}{2}} |\partial_i\psi_{\tau}^l(x)| + K\tau^{-1+\frac{\alpha}{2}} \frac{|\partial_i\psi_{\tau}^l(x) - \partial_i\psi_{\tau}^l(y)|}{|x-y|^{\beta'}}\\
  &\leq K\tau^{-1+\frac{\alpha-\beta'}{2}} + K\tau^{-1+\frac{\alpha}{2}}\|\psi_{\tau}\|_{C^{1,\beta'}(M;M)}
  \end{align*}
for each $\tau\in (t,T)$. Subsequently similar to the argument in deriving the $C^1$ bound (\ref{C^1 bound psi}), we can integrate the above inequality to obtain 
\begin{align*}
	 \frac{|\partial_i\psi_{t}^k(x) - \partial_i\psi_{t}^k(y)|}{|x-y|^{\beta'}} &\leq \int_t^T( K\tau^{-1+\frac{\alpha-\beta'}{2}} + K\tau^{-1+\frac{\alpha}{2}}\|\psi_{\tau}\|_{C^{1,\beta'}(M;M)}) d\tau\\
	 &\leq KT^{\frac{\alpha-\beta'}{2}} + \int_t^TK\tau^{-1+\frac{\alpha}{2}}\|\psi_{\tau}\|_{C^{1,\beta'}(M;M)} d\tau.
\end{align*}
Since $x,y\in M$ are arbitrary, we can combine the above inequality and the inequality (\ref{integral C^1 bound psi}) to obtain
\begin{align*}
	\|\psi_{t}\|_{C^{1,\beta'}(M;M)} \leq K + KT^{\frac{\alpha-\beta'}{2}} + \int_t^TK\tau^{-1+\frac{\alpha}{2}}\|\psi_{\tau}\|_{C^{1,\beta'}(M;M)} d\tau.
\end{align*}
Then Gr\"onwall's lemma implies
\begin{align}
		\|\psi_{t}\|_{C^{1,\beta'}(M;M)} \leq K'(T)
\end{align}
for any $t\in (0,T]$, which is uniform for each $t\in (0,T]$. Since bounded subsets in $C^{1,\beta'}(M;M)$ are precompact in $C^{1,\beta}(M;M)$ as $\beta' >\beta$, upon passing to a subsequence we have $\psi_t\to\psi_0$ in $C^{1,\beta}(M;M)$ as $t\to 0$. 

Lastly,  $\psi_t$ remains a diffeomorphism on $(0,T]$ provided $T$ is sufficiently small by the inverse function theorem. It remains to show that $\psi_0$ is also a diffeomorphism. We need a uniform lower bound for the differential $d\psi_t$. Fix $x, y\in M$ and $t\in (0,T]$. Let $c:[0,L]\to M$ be a curve such that $c(0) = x$ and $c(L) = y$. Define
	\[ h_c(\tau) = \int_0^L(\psi_{\tau}^*\hat{g})(c'(s),c'(s)) ds. \]
	Then $\|\psi_{\tau}(x) - \psi_{\tau}(y)\|_{\hat{g}} = \inf_{c}h_c(\tau)$. Moreover, from the fact that $\sigma^{1-\frac{\alpha}{2}}\|W\|_{0; M\times [\frac{\sigma}{2}, \sigma]}\leq K$, we have 
	\begin{align*}
	h_c'(\tau) &= \int_0^L(\psi_{\tau}^*(\mathcal{L}_{-W}\hat{g}))(c'(s),c'(s)) ds\\
	 &\leq K\|\hat{\nabla}W(\tau)\|_{0;M} h_c(\tau)\\
	 &\leq K\tau^{\frac{\alpha}{2}-1}h_c(\tau)
	\end{align*}
	for each $\tau\in (t,T]$. This implies
	\begin{align*}
	h_c(T) \leq h_c(t)\exp\left(\int_t^TK\tau^{\frac{\alpha}{2}-1} d\tau\right) \leq Kh_c(t).
	\end{align*}
	Similarly,
	\[h_c(T) \geq \frac{1}{K}h_c(t).\]
In particular, this implies the uniform estimate
\begin{align}{\label{Lip estimate}}
	\frac{1}{K}\|x-y\|_{\hat{g}}\leq \|\psi_t(x) - \psi_t(y)\|_{\hat{g}} \leq K\|x-y\|_{\hat{g}}
\end{align}
for any $t\in (0,T]$ and any $x,y\in M$. Consequently both $\psi_0$ and $D\psi_0$ are injective. Thus $\psi_0$ is a diffeomorphism to its image. Now we show that $\psi_0$ is surjective. Take any $y\in M$ and choose a sequence $s_j\to 0$. Define a sequence of points by $x_j:= \psi_{s_j}^{-1}(y)$. Since $M$ is compact, after passing to a subsequence we have $x_j\to x_{\infty}\in M$. The uniform estimate (\ref{Lip estimate}) then implies that $|\psi_{s_j}(x_j) - \psi_{s_j}(x_{\infty})| \to 0$. On the other hand, we have $|\psi_{s_j}(x_{\infty}) - \psi_{0}(x_{\infty})| \to 0$. Hence $|\psi_{s_j}(x_j) - \psi_{0}(x_{\infty})| \to 0$ and we get $\psi_0(x_{\infty}) = y$.

\end{proof}

\bigskip

For the family of diffeomorphisms $\{\psi_t\}_{t\in [0,T]}$, it is possible that the curve $t\mapsto \psi_t(x)$ does not stay in the same chart for all $t\in [0,T]$. Hence if $\tilde{\psi}: M\times [0,T]\to M$ is a map defined by $\tilde{\psi}(x,t) = \psi_t(x)$, we cannot measure its parabolic H\"older norm the way we did for the elliptic H\"older norm without any modification. However, our final goal is to show that the pullback metric $\psi_t^*g_t$ satisfies the Ricci flow equation and lies in the weighted space $\mathcal{X}_{*,\gamma}^{(*)}$ provided that $g\in\mathcal{X}_{k,\gamma}^{(\alpha)}$ is a solution to the Ricci-DeTurck flow.

Now, consider the section $d\psi_t: M\to T^*M\otimes (\psi_t)^*(TM)$. We introduce the multi-index notations
 \[ \mu = (\mu_1, \mu_2).\]
 If $(x_{i}^{\mu_1})_{i=1,..,n}$ and $(y_{j}^{\mu_2})_{j=1,..,n}$ are local coordinates on the charts $U_{\mu_1}$ and $U_{\mu_2}$ respectively, then
	\[ \bold{e}_{ij}^{\mu}(t) := dx_{i}^{\mu_1}\otimes\frac{\partial}{\partial (y_{j}^{\mu_2}\circ\psi_t)}, \quad i_1, i_2 = 1,..,n\]
	form a local frame on $T^*M\otimes (\psi_t)^*(TM)|_{U_{\mu_1}\cap \psi_t^{-1}(U_{\mu_2})}$. 
 Moreover, if $x\in U_{\mu}^t:= U_{\mu_1}\cap \psi_t^{-1}(U_{\mu_2})$, then locally around $x$ the section $d\psi_t$ can be expressed as  
	\[d\psi_t = (d\psi_t)_{\mu}^{ij}\bold{e}_{ij}^{\mu}(t),\quad\text{where}\quad  (d\psi_t)_{\mu}^{ij} = \partial_{i}(\psi_t)_{\mu}^{j}\]
	
\begin{definition}{\label{map Psi}}
	With the above notations, let $\{\rho_{\mu_2}\}$ be the partition of unity subordinate to the chart $U_{\mu_2}$, we define functions $\Psi_{\mu}^{ij}: M\times [0,T]\to \mathbb{R}$ by
	\begin{align*}
	\Psi_{\mu}^{ij}(x, t) = 
	\begin{cases}
		\rho_{\mu_2}(\psi_t(x))(d\psi_t)_{\mu}^{ij}(x)\quad &\text{if}\quad \psi_t(x)\in\text{supp}(\rho_{\mu_2})\\
		0\quad &\text{if}\quad \psi_t(x)\notin\text{supp}(\rho_{\mu_2})
	\end{cases}.
	\end{align*}
\end{definition}
Note that in this way, although the component maps $t\mapsto (d\psi_t)_{\mu}^{ij}(x)$ may not be defined for all $t\in [0,T]$, the maps $\Psi_{\mu}^{ij}(x, \cdot)$ are well defined for all $t\in [0,T]$.

\bigskip

\begin{lemma}{\label{lem Psi estimate}}
Let  $g\in\mathcal{X}_{k+2,\gamma}^{(\alpha)}$ be the solution to the Ricci-DeTurck flow on $M\times (0,T]$, and $\{\psi_t\}_{t\in [0,T]}$ be the one-parameter family of diffeomorphisms generated by (\ref{DeT ODE}).  Then given any $\beta<\alpha$, the functions $\Psi_{\mu}^{ij}$ which are defined by Definition \ref{map Psi} satisfy
\begin{itemize}
\item[(1)] $\Psi_{\mu}^{ij}\in C^{\beta,\frac{\beta}{2}}(M\times [0,T];\mathbb{R})\ $ and\quad  $\hat{\nabla}\Psi_{\mu}^{ij} \in \mathcal{C}^{k-2,\gamma}_{\frac{1}{2}-\frac{\beta}{2}}(M\times (0,T];\mathbb{R})$;	
\item[(2)] 
	$\|\Psi_{\mu}^{ij}\|_{\beta,\frac{\beta}{2}; M\times [0,T]} + \|\hat{\nabla}\Psi_{\mu}^{ij}\|_{\mathcal{C}^{k-2,\gamma}_{\frac{1}{2}-\frac{\beta}{2}}( M\times (0,T])} \leq K (M, k, \hat{g},  \|g_0\|_{\alpha}, \|g\|_{\mathcal{X}_{k+2,\gamma}^{(\alpha)}}).$
\end{itemize}
\end{lemma}

\begin{proof}
Firstly, Lemma \ref{lem DeT} has already shown that $\Psi_{\mu}^{ij}(\cdot, t)\in C^{0,\beta}(M;\mathbb{R})$ for each $t\in [0,T]$ with an uniform upper bound for the elliptic H\"older norm. Thus to show that $\Psi_{\mu}^{ij}\in C^{\beta,\frac{\beta}{2}}(M\times [0,T];\mathbb{R})$, it remains to show that $\Psi_{\mu}^{ij}(x, \cdot)$ is also H\"older-continuous for $t\in [0,T]$ for each $x\in M$. We rewrite the equation (\ref{dpsi ODE}) in the new notations as
\begin{align*}
\frac{\partial}{\partial t}(d\psi_t)_{\mu}^{ij} = -\partial_l W^j(\psi_t, t)\ (d\psi_t)_{\mu}^{il}.
\end{align*}
This gives
\[ \frac{\partial}{\partial t}\Psi_{\mu}^{ij} = -\langle D\rho_{\mu_2}(\psi_t), W\rangle (d\psi_t)_{\mu}^{ij} - \rho_{\mu_2}(\psi_t)\partial_l W^j(\psi_t, t)\ (d\psi_t)_{\mu}^{il}.\]
Using the $C^1$ bound (\ref{C^1 bound psi}) and the fact that $\|W\|_{\mathcal{C}^{k+1,\gamma}_{\frac{1}{2}-\frac{\alpha}{2}}(M\times(0,T])}\leq K$, we have
\[ \frac{\partial}{\partial t}\Psi_{\mu}^{ij} \leq Kt^{-1+\frac{\alpha}{2}}\]
for each $t\in (0,T)$. Hence
\begin{align*}
	|\Psi_{\mu}^{ij}(x, t) - \Psi_{\mu}^{ij}(x, s)| &\leq \int_s^t K\tau^{-1+\frac{\alpha}{2}}\ d\tau\\
	&\leq K( t^{\frac{\alpha}{2}} - s^{\frac{\alpha}{2}})\\
	&\leq K\ |t-s|^{\frac{\alpha}{2}}
\end{align*}
for $t,s\in [0,T]$. This gives $\sup_{t\neq s}\frac{|\Psi_{\mu}^{ij}(x, t) - \Psi_{\mu}^{ij}(x, s)|}{|t-s|^{\frac{\alpha}{2}}} \leq K$ for any $x\in M$. Hence, we conclude that 
	\[\|\Psi_{\mu}^{ij}\|_{\beta,\frac{\beta}{2}; M\times [0,T]} \leq K.\]

\bigskip

In the next step, we prove the remaining assertions. Heuristically as $g\in\mathcal{X}_{k+2,\gamma}^{(\alpha)}$, we have $W\in\mathcal{C}^{k+1,\gamma}_{\frac{1}{2}-\frac{\alpha}{2}}$, and so the highest regularity we can achieve for $\hat{\nabla}\Psi_{\mu}^{ij}$ would be $\mathcal{C}^{k-1,\gamma}_{\frac{1}{2}-\frac{\beta}{2}}$. On the other hand, we can easily control the H\"older semi-norm $[\hat{\nabla}^{m-1}\Psi]_{\gamma,\frac{\gamma}{2};M\times[\frac{\sigma}{2},\sigma]}$ by the higher order $C^0$ norm $|\hat{\nabla}^{m}\Psi|_{0;M\times[\frac{\sigma}{2},\sigma]}$. And since we have the freedom of choosing $k$ to be any large integer, it doesn't hurt if we omit controlling the highest order H\"older semi-norm $[\hat{\nabla}^{k}\Psi]_{\gamma,\frac{\gamma}{2};M\times[\frac{\sigma}{2},\sigma]}$. Nevertheless, we will first control the $C^0$ norm of $|\hat{\nabla}^{m}\Psi|_{0;M\times[\frac{\sigma}{2},\sigma]}$ up to the highest order. More precisely, we first prove the estimate
	\begin{align}{\label{high order Psi estimate}}
		\sigma^{\frac{1}{2}-\frac{\beta}{2}+\frac{m-2}{2}}|\hat{\nabla}^{m-1}\Psi_{\mu}^{ij}|_{0;M\times [\frac{\sigma}{2}, T]} \leq K
	\end{align}
	 by induction for $2\leq m\leq k+1$. For notation simplicity, we abbreviate the components $(d\psi_t)_{\mu}^{ij}$ of the differential $d\psi_t$ simply by $\partial\psi_t$, the partial derivatives $\partial_x^{m-1}(d\psi_t)_{\mu}^{ij}$ and $\partial_x^m W^i$ by $\partial^m\psi_t$ and $\partial^mW$ respectively, and the functions $\Psi_{\mu}^{ij}$ by $\Psi$. Moreover, with this abbreviation, for each $t\in (0,T]$ the $C^0$ norm for $\partial^m\psi_t$ is given by
	 \[\|\partial^m\psi_t\|_{0;M} = \sum_{\mu_2}\sum_{i,j}|\partial_x^{m-1}(d\psi_t)_{\mu}^{ij}|_{C^0(\psi_t^{-1}(U_{\mu_2});\mathbb{R})}.\]
	    Now suppose that  the estimate (\ref{high order Psi estimate}) holds for all $2\leq j\leq m$. Thus the induction hypothesis implies that
	    \[ \|\partial^{j}\psi_t\|_{0;M} \leq Kt^{\frac{\beta}{2}-\frac{1}{2}-\frac{j-2}{2}}\]
	    for all $2\leq j\leq m$ and $t\in [\frac{\sigma}{2}, T]$.
	     Using the Francesco Fa\`a di Bruno formula for higher-order chain rule, the following ODE is satisfied by $\partial^{m+1}\psi_t$:
		\begin{align}{\label{high order psi equation}}
			&\frac{\partial}{\partial t}(\partial^{m+1}\psi_t)\\
			&\notag = -\ \sum_{j_1 + 2j_2 +\cdots + (m+1)j_{m+1} = m+1} (\partial^{j_1+\cdots+j_{m+1}}W)	(\psi_t, t) * (\partial\psi_t)^{*j_1} *\cdots * (\partial^{m+1}\psi_t)^{*j_{m+1}}\\
			&\notag = -\ \sum_{j_1 + 2j_2 +\cdots + mj_{m} = m+1}(\partial^{j_1+\cdots+j_{m}}W) * (\partial\psi_t)^{*j_1}  *\cdots * (\partial^{m}\psi_t)^{*j_{m}}\ -\ \partial W * \partial^{m+1}\psi_t .
		\end{align}
 	Here $(\partial^i\psi)^{*j}$ stands for $\underbrace{\partial^i\psi *\cdots*\partial^i\psi}_{j\rm\ times}$ and the summation sums over all non-negative integer solutions of $j_1 + 2j_2 +\cdots + mj_{m} = m+1$. \\
 	Now observe that for any $t\in [\frac{\sigma}{2}, T]$, the $C^0$ bound for $\partial^jW$, $\partial\psi_t$, and the induction hypothesis imply
 	\begin{align}{\label{high order W estimate}}
 		&\left|(\partial^{j_1+\cdots+j_{m}}W)(\psi_t,t) * (\partial\psi_t)^{*j_1}  *\cdots * (\partial^{m}\psi_t)^{*j_{m}}\right|\\
 		&\notag \leq K\ t^{\frac{\alpha}{2}-\frac{1}{2}-\frac{j_1+\cdots+j_m}{2}}\ \prod_{r=2}^m \left\| \partial^r\psi_t \right\|^{j_r}_{0;M}\\
 		&\notag  \leq K\ t^{\frac{\alpha}{2}-\frac{1}{2}-\frac{j_1+\cdots+j_m}{2}}\ \prod_{r=2}^m\ t^{(\frac{\beta}{2}+\frac{1}{2}-\frac{r}{2})j_r}\\
 		&\notag = K\ t^{\frac{\alpha}{2}-\frac{m+2}{2}+\frac{\beta}{2}\sum_{r=2}^m j_r} \\
 		&\notag \leq K\ t^{\frac{\alpha}{2}-\frac{m+2}{2}}
 	\end{align}
	for each $\{j_1,..,j_m\}$ satisfying $j_1 + 2j_2 +\cdots + mj_m = m+1$.\\
	Let us define $f(t):=\|\partial^{m+1}\psi_t\|_{0;M}$. Since $(d\psi_T)_{\mu}^{ij}(x) = \delta^i_j$ for all $x\in\psi_t^{-1}(U_{\mu_2})$, we have $f(T) = 0$. By (\ref{high order Psi estimate}), (\ref{high order psi equation}), (\ref{high order W estimate}) and the fundamental theorem of calculus, we have
	\begin{align*}
		f(t) &\leq K\int_t^T\tau^{\frac{\alpha}{2}-\frac{m+2}{2}} d\tau \ + K\int_t^T\tau^{\frac{\alpha}{2}-1} f(\tau) d\tau\\
		&\leq K t^{\frac{\alpha}{2}-\frac{m}{2}} + K\int_t^T\tau^{\frac{\alpha}{2}-1} f(\tau) d\tau
	\end{align*}
	for each $t\in [\frac{\sigma}{2}, T]$. Hence Gr\"onwall's lemma implies
	\begin{align}{\label{partial psi estimate}}
		\|\partial^{m+1}\psi_t\|_{0;M} \leq K t^{\frac{\alpha}{2}-\frac{m}{2}} \exp (KT^{\frac{\alpha}{2}})
	\end{align}
	for each $t\in [\frac{\sigma}{2}, T]$. Now, since $\Psi = (\rho\circ\psi_t)*(\partial\psi_t)$, where we have abbreviated the function $\rho_{\mu_2}$ as $\rho$. Using the Francesco Fa\`a di Bruno formula again, we obtain
	\begin{align*}
	\partial^l(\rho\circ\psi_t) = \sum_{i_1 + 2i_2 +\cdots+li_l = l}\rho^{(i_1+\cdots+i_l)}(\psi_t)* (\partial\psi_t)^{*i_1} *\cdots * (\partial^{l}\psi_t)^{*i_{l}}
	\end{align*}
	for $1\leq l\leq m+1$. By the induction hypothesis and (\ref{partial psi estimate}), we have the estimate
	\begin{align}{\label{rho estimate}}
	|\partial^l(\rho\circ\psi_t)| &\leq K\sum_{i_1 + 2i_2 +\cdots+li_l = l}\left(\prod_{r=2}^l\left\| \partial^r\psi_t \right\|^{i_r}_{0;M}\right)\\
	&\notag\leq K\sum_{i_1 + 2i_2 +\cdots+li_l = l}\left(\prod_{r=2}^l t^{(\frac{\beta}{2}+\frac{1}{2}-\frac{r}{2})i_r}\right)\\
	&\notag\leq K\sum_{i_1 + 2i_2 +\cdots+li_l = l}t^{-\frac{1}{2}\sum_{r=2}^lri_r}\\
	&\notag\leq K\sum_{i_1 + 2i_2 +\cdots+li_l = l}t^{-\frac{l-i_1}{2}}\\
	&\leq Kt^{-\frac{l}{2}}\notag.
	\end{align}
	for any $1\leq l\leq m+1$ and $t\in [\frac{\sigma}{2}, T]$. Then we observe that $\partial^{m+1}\Psi$ satisfies
	\begin{align*}
	\partial^{m}\Psi = \sum_{j_1+j_2 = m} \partial^{j_1}(\rho\circ\psi_t)*\partial^{j_2+1}\psi_t.
	\end{align*}
	Using the estimate (\ref{rho estimate}), the induction hypothesis and (\ref{partial psi estimate}), we subsequently obtain
	\begin{align*}
	|\partial^{m}\Psi(x,t) | &\leq K \sum_{j_1+j_2 = m}t^{-\frac{j_1}{2}}t^{\frac{\beta}{2}-\frac{j_2}{2}} \leq Kt^{\frac{\beta}{2}-\frac{m}{2}}.
	\end{align*}
	 for each $t\in [\frac{\sigma}{2}, T]$. Therefore we conclude that
	\begin{align*}
		\sigma^{\frac{1}{2}-\frac{\beta}{2}+\frac{m-1}{2}}|\hat{\nabla}^{m}\Psi_{\mu}^{ij}|_{0;M\times [\frac{\sigma}{2}, T]} \leq K.
	\end{align*}
	Note that in the RHS of (\ref{partial psi estimate}) the exponent of $\sigma$ is negative for any $m\geq 1$. This is why we cannot hope for higher powers of $\hat{\nabla}^{m}\Psi_{\mu}^{ij}$ to converge for any $m\geq 1$ as $t\to 0$. Lastly, 
	to verify that (\ref{high order Psi estimate}) holds for $m=2$, it suffices to consider the following ODE satisfied by $\partial^2\psi_t$:
	\begin{align*}
		\frac{\partial}{\partial t}(\partial^2\psi_t) = -\ \partial^2W(\psi_t,t)*\partial\psi_t	 * \partial\psi_t -\ \partial W(\psi_t,t)*\partial^2\psi_t	.
	\end{align*}
	For each $t\in [\frac{\sigma}{2},T]$, the $C^0$ bound for $\partial^2W$ and $\partial\psi_t$ imply
	\begin{align*}
		 |\partial^2W(\psi_t,t)*\partial\psi_t * \partial\psi_t| \leq K\ t^{\frac{\alpha}{2}-\frac{3}{2}}.
	\end{align*}
	Now we can proceed as before to obtain (\ref{high order Psi estimate}) for the case $m=2$. This proves (\ref{high order Psi estimate}) by induction. In particular, this implies that
	\begin{align}{\label{high order Psi estimate 1}}
		\sigma^{\frac{1}{2}-\frac{\beta}{2}+\frac{m-2}{2}}|	\hat{\nabla}^{m-1}\Psi_{\mu}^{ij}|_{0;M\times [\frac{\sigma}{2}, \sigma]} \leq K
	\end{align}
	for $2\leq m\leq k+1$.
	
	Next, we show that
	\begin{align}{\label{C^a estimate Psi}}
		\sigma^{\frac{1}{2}-\frac{\beta}{2}+\frac{m-2}{2}+\frac{\gamma}{2}} [	\hat{\nabla}^{m-1}\Psi_{\mu}^{ij}]_{\gamma,\frac{\gamma}{2};M\times [\frac{\sigma}{2},\sigma]} \leq K
	\end{align}
	for any $2\leq m\leq k$. We first observe that the following equation is satisfied by $\partial^{m-1}\Psi$:
	\begin{align*}
	&\frac{\partial}{\partial t}	(\partial^{m-1}\Psi)\\
	&= \sum_{j_1 + j_2 = m-1}\left\{\partial^{j_1+1}\psi_t * \sum_{l_1 + l_2 = j_2}(\partial^{l_1}(\rho\circ\psi_t)*\frac{\partial}{\partial t}(\partial^{l_2}\psi_t)) +  \frac{\partial}{\partial t}(\partial^{j_1+1}\psi_t) * \partial^{j_2}(\rho\circ\psi_t) \right\}.
	\end{align*}
Using (\ref{high order psi equation}) to (\ref{partial psi estimate}) and the $C^0$ norm for $\hat{\nabla}W$, we obtain
\begin{align*}
	\left|\frac{\partial}{\partial t}(\partial^{l}\psi_t)\right| &\leq Kt^{\frac{\alpha}{2}-\frac{l+1}{2}} + Kt^{\frac{\alpha}{2}-1}\cdot t^{\frac{\alpha}{2}-\frac{l-1}{2}}\\
	&\leq Kt^{\frac{\alpha}{2}-\frac{l+1}{2}} + K t^{\alpha-\frac{l+1}{2}}\\
	&\leq  Kt^{\frac{\alpha}{2}-\frac{l+1}{2}}
\end{align*}
for all $0\leq l\leq k+1$.
Combining the above estimate, (\ref{partial psi estimate}) and (\ref{rho estimate}), we obtain
\begin{align*}
& \left|\frac{\partial}{\partial t}	(\partial^{m-1}\Psi)(x,t)\right|\\
&\leq K\sum_{j_1 + j_2 = m-1}\left\{t^{\frac{\alpha}{2}-\frac{j_1}{2}}  \sum_{l_1 + l_2 = j_2}t^{-\frac{l_1}{2}} \cdot t^{\frac{\alpha}{2}-\frac{l_2}{2}} +  t^{\frac{\alpha}{2}-\frac{j_1 + 2}{2}}\cdot t^{-\frac{j_2}{2}} \right\}\\
&\leq K\sum_{j_1 + j_2 = m-1}\left\{t^{\alpha-\frac{j_1+j_2}{2}} +  t^{\frac{\alpha}{2}-\frac{j_1 + j_2 + 2}{2}} \right\}\\
&\leq K(t^{\alpha-\frac{m-1}{2}} +  t^{\frac{\alpha}{2}-\frac{m+1}{2}})\\
&\leq Kt^{\frac{\alpha}{2}-\frac{m+1}{2}}.
\end{align*}

Now, we fix a point $z\in M$, and let $(x,t), (y,s)\in B_{\sqrt{\sigma}}(z)\times [\frac{\sigma}{2},\sigma]$. We may assume without loss of generality that $\psi_t(x), \psi_s(y)\in\text{supp}(\rho_{\mu_2})$, and $t>s$ are close enough such that $\psi_{\tau}(y)\in\text{supp}(\rho_{\mu_2})$ for $\tau\in [s, t]$. Using (\ref{high order psi equation}), (\ref{high order W estimate}) and (\ref{high order Psi estimate 1}) we have
	\begin{align*}
		&\frac{|\partial_x^{m-1}\Psi_{\mu}^{ij}(x,t) - \partial_x^{m-1}\Psi_{\mu}^{ij}(y,s)|}{|x-y|^{\gamma} + |t-s|^{\frac{\gamma}{2}}}\\
		&\leq \frac{|\partial_x^{m-1}\Psi_{\mu}^{ij}(x,t) - \partial_x^{m-1}\Psi_{\mu}^{ij}(y,t)|}{|x-y|^{\gamma} } + \frac{|\partial_x^{m-1}\Psi_{\mu}^{ij}(y,t) - \partial_x^{m-1}\Psi_{\mu}^{ij}(y,s)|}{|t-s|^{\frac{\gamma}{2}}}  \\
		&\leq K\ |	\partial_x^{m}\Psi_{\mu}^{ij}|_{0;M\times [\frac{\sigma}{2},\sigma]}\ |x-y|^{1-\gamma} + K|t-s|^{-\frac{\gamma}{2}}\int_s^t\left| \frac{\partial}{\partial \tau}	(\partial_x^{m-1}\Psi_{\mu}^{ij})(y,\tau)\right| d\tau\\	
		&\leq K\ \sigma^{\frac{\beta}{2}-\frac{m}{2}}\sigma^{\frac{1}{2}-\frac{\gamma}{2}}\ + K|t-s|^{-\frac{\gamma}{2}}\int_{s}^{t} \tau^{\frac{\alpha}{2}-\frac{m+1}{2}}   d\tau \\
		&\leq K\ \sigma^{\frac{\beta}{2}-\frac{\gamma}{2}-\frac{m-1}{2}}  +  K|t-s|^{-\frac{\gamma}{2}}\int_{s}^{t} \tau^{\frac{\alpha}{2}-\frac{m+1}{2}} d\tau\\
		&\leq K\ \sigma^{\frac{\beta}{2}-\frac{\gamma}{2}-\frac{m-1}{2}}  +  K|t-s|^{-\frac{\gamma}{2}}(t^{\frac{\alpha}{2}-\frac{m-1}{2}} - s^{\frac{\alpha}{2}-\frac{m-1}{2}}).
	\end{align*}
	Now, we can write $\frac{\alpha}{2}-\frac{m-1}{2} = (\frac{\beta}{2}-\frac{\gamma}{2} - \frac{m-1}{2}) + (\frac{\alpha-\beta}{2} + \frac{\gamma}{2})$. Using the fact that $\alpha > \beta$ and $t, s\in [\frac{\sigma}{2}, \sigma]$, we obtain
	\begin{align*}
		&|t-s|^{-\frac{\gamma}{2}}(t^{\frac{\alpha}{2}-\frac{m-1}{2}} - s^{\frac{\alpha}{2}-\frac{m-1}{2}})\\
		&\leq K\sigma^{\frac{\beta}{2}-\frac{\gamma}{2}-\frac{m-1}{2}}|t-s|^{-\frac{\gamma}{2}}(t^{\frac{\gamma}{2}} - s^{\frac{\gamma}{2}})\\
		&\leq K'\sigma^{-\frac{1}{2} + \frac{\beta}{2}-\frac{\gamma}{2}-\frac{m-2}{2}}.
	\end{align*}
	Hence
	\[ \frac{|\partial_x^{m-1}\Psi_{\mu}^{ij}(x,t) - \partial_x^{m-1}\Psi_{\mu}^{ij}(y,s)|}{|x-y|^{\gamma} + |t-s|^{\frac{\gamma}{2}}} \leq K\sigma^{-\frac{1}{2} + \frac{\beta}{2}-\frac{\gamma}{2}-\frac{m-2}{2}}\]
	for any $(x,t), (y,s)\in B_{\sqrt{\sigma}}(z)\times [\frac{\sigma}{2},\sigma]$.
	Since $z\in M$ is arbitrary, we have proved (\ref{C^a estimate Psi}).

\end{proof}

\begin{proposition}{\label{prop pullback g}}
Let $k\geq 2$ and $\beta, \gamma\in (0,\alpha)$ be given. There exist a $C^{1,\beta}$ diffeomorphism $\psi:M\to M$ and $K = K(M, k, \hat{g},  \|g_0\|_{\alpha;M})$, $\ T = T(M, \hat{g},  \|g_0\|_{\alpha;M})$ such that the following holds: 
There is a solution $\tilde{g}\in\mathcal{X}_{k,\gamma}^{(\beta)}(M\times[0,T])$ to the Ricci flow such that
$$\tilde{g}(\cdot,0) = \psi^*g_0 \quad\text{and}\quad \|\tilde{g}\|_{\mathcal{X}_{k,\gamma}^{(\beta)}(M\times[0,T])}\leq K. $$
\end{proposition}

\bigskip
\begin{proof}\

By Corollary \ref{thm existence high Ricci DeT}, we can find $T = T(M, \hat{g},  \|g_0\|_{\alpha})>0$ sufficiently small such that $g(t)\in\mathcal{X}_{k+3,\gamma}^{(\alpha)}(M\times[0,T])$ is the unique solution to the Ricci-DeTurck system such that
	$$ \|g\|_{\mathcal{X}_{k+3,\gamma}^{(\alpha)}}\leq K.$$
By Lemma \ref{lem DeT} and Lemma \ref{lem Psi estimate}, we can find a one-parameter family of diffeomorphisms $\{\psi_t\}_{t\in [0,T]}$ which is generated by 

\begin{align*}
    \begin{cases}
    &\frac{\partial}{\partial t}\psi_t = -W(\psi_t,t)\\
    &\psi_T = \id
    \end{cases},
\end{align*}
such that if the functions $\Psi_{\mu}^{ij}$ are defined by Definition \ref{map Psi}, then they satisfy
\begin{itemize}
\item[(1)] $\Psi_{\mu}^{ij}\in C^{\beta,\frac{\beta}{2}}(M\times [0,T];\mathbb{R})\ $ and\quad  $\hat{\nabla}\Psi_{\mu}^{ij} \in \mathcal{C}^{k-1,\gamma}_{\frac{1}{2}-\frac{\beta}{2}}(M\times (0,T];\mathbb{R})$;	
\item[(2)] 
	$\|\Psi_{\mu}^{ij}\|_{\beta,\frac{\beta}{2}; M\times [0,T]} + \|\hat{\nabla}\Psi_{\mu}^{ij}\|_{\mathcal{C}^{k-1,\gamma}_{\frac{1}{2}-\frac{\beta}{2}}( M\times (0,T])} \leq K (M, k, \hat{g},  \|g_0\|_{\alpha}, \|g\|_{\mathcal{X}_{k+3,\gamma}^{(\alpha)}})$.
\end{itemize}
Equivalently, we have $\|\hat{\nabla}\Psi_{\mu}^{ij}\|_{\mathcal{C}^{k-1,\gamma}_{\frac{1}{2}-\frac{\beta}{2}}( M\times (0,T])} \leq K$.

\bigskip

 We then define $\tilde{g}(t):=(\psi_t)^*g(t)$. Thus $\tilde{g}(0) = \psi_0^*g_0$ where $\psi_0$ is a $C^{1,\beta}$ diffeomorphism. We have
\begin{align*}
    \frac{\partial}{\partial t}\tilde{g} = (\psi_t)^*(-2Ric(g(t))+L_Wg(t))-L_W\tilde{g}(t)=-2Ric(\tilde{g}(t)).
\end{align*}
Hence $\tilde{g}$ is a solution to the Ricci flow on a $M\times (0,T]$. In the sequel, let us denote $g(t)$ by $g_t$. Next, we fix a chart $U$. Let $x\in U$ and $\psi_t(x)\in V$ for some chart $V$. Let $\{x_i\}$ and $\{\tilde{x}_i\}$ be the coordinates on the charts $U$ and $V$ respectively. Moreover, we denote by $h_{\mu}: U_{\mu}\cap V \to \textup{GL}_n(\mathbb{R})$ the induced transition maps between trivializations of the tangent bundle $TM$; and by $H_s: U_{s}\cap V \to \textup{GL}_N(\mathbb{R})$ the induced transition maps between trivializations of the bundle $\text{Sym}^2(T^*M)$. So on the intersection $U_{\mu}\cap V$, we have
	\[ (d\psi_t)^{ij} = (h_{\mu})^j_k(d\psi_t)_{\mu}^{ik}\quad\text{and}\quad g_{kl} = (H_{\mu})_{kl}^{ab}\cdot g_{ab}^{\mu}\]
	where $(d\psi_t)^{ij} = \partial_i\psi_t^j$ are the components of $d\psi_t$ on $U\cap\psi_t^{-1}(V)$, and $ g_{ab}^{\mu}$ are the components of $g$ on $U_{\mu}$. 
 Then
\begin{align}{\label{pullback metric relation}}
\tilde{g}(x,t)_{ij} &=
 \psi_t^*g_t(x)_{ij}\\
 &\notag= g_t(\psi_t(x))( (d\psi_t)(\frac{\partial}{\partial x_i}),\ (d\psi_t)(\frac{\partial}{\partial x_j})) \\
 &\notag= g_t(\psi_t(x))( (d\psi_t)^{ik}\frac{\partial}{\partial \tilde{x}_k}, (d\psi_t)^{jl}\frac{\partial}{\partial \tilde{x}_l})\\
 &\notag= (g_t\circ\psi_t)_{kl}(d\psi_t)^{ik}(d\psi_t)^{jl}\\
 &\notag= \sum_s\rho_s(\psi_t)(g_t\circ\psi_t)_{kl}\cdot\sum_{\mu}\rho_{\mu}(\psi_t)(d\psi_t)^{ik}\cdot\sum_{\nu}\rho_{\nu}(\psi_t)(d\psi_t)^{jl}\\
 &\notag= \sum_s\rho_s(\psi_t)(g_t\circ\psi_t)_{ab}^s (H_s)_{kl}^{ab}\cdot\sum_{\mu}\rho_{\mu}(\psi_t)(h_{\mu})_p^k(d\psi_t)_{\mu}^{ip}\cdot\sum_{\nu}\rho_{\nu}(\psi_t)(h_{\nu})_q^l(d\psi_t)_{\nu}^{jq}\\
 &\notag= \sum_{\mu,\nu, s}\rho_s(\psi_t)(g_t\circ\psi_t)_{ab}^s\cdot\Psi^{ip}_{\mu}\Psi^{jq}_{\nu}\cdot (h_{\mu})_p^k(h_{\nu})_q^l(H_s)_{kl}^{ab}.
\end{align}
Now, since $g\in C^{\alpha,\frac{\alpha}{2}}(M\times [0,T])$, $\Psi \in C^{\beta,\frac{\beta}{2}}(M\times [0,T])$ and $\alpha > \beta$, we have $\tilde{g}\in C^{\beta,\frac{\beta}{2}}(M\times [0,T])$ from (\ref{pullback metric relation}). Next, for $1\leq m\leq k$, we see from (\ref{pullback metric relation}) that
\begin{align*}
	\hat{\nabla}\tilde{g}	= \sum_{j_1+j_2+j_3+j_4 = 1}\hat{\nabla}^{j_1}(\rho\circ\psi_t) * \hat{\nabla}^{j_2}(g_t\circ\psi_t) * \hat{\nabla}^{j_3}\Psi * \hat{\nabla}^{j_4}\Psi.
\end{align*}
Note that we have $ \|\hat{\nabla}\Psi\|_{\mathcal{C}^{k-1,\gamma}_{\frac{1}{2}-\frac{\beta}{2}}( M\times (0,T])}\leq K$. Since $\psi_t$ is a $C^{1,\beta}$ diffeomorphism, $g_t\circ\psi_t$ has the same regularity as $g_t$, and so $ \|\hat{\nabla}(g_t\circ\psi_t)\|_{\mathcal{C}^{k-1,\gamma}_{\frac{1}{2}-\frac{\beta}{2}}( M\times (0,T])}\leq \|\hat{\nabla}(g_t\circ\psi_t)\|_{\mathcal{C}^{k-1,\gamma}_{\frac{1}{2}-\frac{\alpha}{2}}( M\times (0,T])}\leq  K$. By the definition of $\Psi$, $\rho\circ\psi_t$ has at least the regularity of $\Psi$, thus $ \|\hat{\nabla}(\rho\circ\psi_t)\|_{\mathcal{C}^{k-1,\gamma}_{\frac{1}{2}-\frac{\beta}{2}}( M\times (0,T])}\leq K$.
 Moreover, by Lemma \ref{properties of weighted space}, there is
\begin{align*}
	 \|\Psi\|_{\mathcal{C}^{k-1,\gamma}_{0}( M\times (0,T])} &\leq K( \|\Psi\|_{\beta,\frac{\beta}{2}; M\times (0,T])} + \|\hat{\nabla}\Psi\|_{\mathcal{C}^{k-2,\gamma}_{\frac{1}{2}}( M\times (0,T])})\\
	  &\leq K( \|\Psi\|_{\beta,\frac{\beta}{2}; M\times (0,T])} + \|\hat{\nabla}\Psi\|_{\mathcal{C}^{k-1,\gamma}_{\frac{1}{2}-\frac{\beta}{2}}( M\times (0,T])})\\
	  &\leq K.
\end{align*}
Similarly,
\[ \|g_t\circ\psi_t\|_{\mathcal{C}^{k-1,\gamma}_{0}( M\times (0,T])}\leq K\quad\text{and}\quad  \|\rho\circ\psi_t\|_{\mathcal{C}^{k-1,\gamma}_{0}( M\times (0,T])}\leq K.\]
Putting everything together, we obtain
\begin{align*}
	&\|(\rho\circ\psi_t) * (g_t\circ\psi_t) *\Psi * \hat{\nabla}\Psi\|_{\mathcal{C}^{k-1,\gamma}_{\frac{1}{2}-\frac{\beta}{2}}( M\times (0,T])}\\
	&\leq 	K\|(\rho\circ\psi_t) \|_{\mathcal{C}^{k-1,\gamma}_{0}( M\times (0,T])} \|(g_t\circ\psi_t) \|_{\mathcal{C}^{k-1,\gamma}_{0}( M\times (0,T])} \|\Psi \|_{\mathcal{C}^{k-1,\gamma}_{0}( M\times (0,T])} \| \hat{\nabla}\Psi\|_{\mathcal{C}^{k-1,\gamma}_{\frac{1}{2}-\frac{\beta}{2}}( M\times (0,T])}\\
	&\leq K.
\end{align*}
We can similarly derive the bounds for the other terms in the summation of (\ref{pullback metric relation}). Therefore we conclude that
\[ \|\hat{\nabla}\tilde{g}\|_{\mathcal{C}^{k-1,\gamma}_{\frac{1}{2}-\frac{\beta}{2}}( M\times (0,T])} \leq K.\]
This proves that $\tilde{g}\in\mathcal{X}_{k,\gamma}^{(\beta)}(M\times[0,T])$ and the assertion follows.
\end{proof}

\bigskip
\bigskip
\newpage

\section{Short Time Existence and Uniqueness to the Harmonic map heat flow}\

In this subsection, we assume that a solution to the Ricci flow is given. To construct a solution to the Ricci-Deturck flow, we prove the short time existence and the uniqueness to the associated harmonic map heat flow. 

Throughout this subsection, let $\alpha\in (0,1)                                                                                                                                   $ be given such that $g_0\in C^{\alpha}(M)$. Let $\gamma\in (0,\alpha)$ be given and we fix $\beta\in (\gamma, \alpha)$. Moreover, let $g(t)\in\mathcal{X}_{k,\gamma}^{(\beta)}(M\times[0,T])$ be a solution to the Ricci flow on $M\times (0,T]$ and $\psi$ be a $C^{1,\beta}$ diffeomorphism such that
\begin{itemize}
\item $g(0) = \psi^*g_0$;
\item $\|\psi\|_{C^{1,\beta}(M;M)}\leq C$;
\item $	\|g\|_{\mathcal{X}_{k,\gamma}^{(\beta)}(M\times[0,T])} = \|g\|_{\beta,\frac{\beta}{2};M\times [0,T]} + \|\hat{\nabla}g\|_{\mathcal{C}_{\frac{1}{2}-\frac{\beta}{2}}^{k-1,\gamma}} \leq C.$	
\end{itemize}
for some constant $C>0$.

\bigskip 
Associated with $g(t)$, we consider the harmonic map heat flow
 \begin{align}{\label{harmonic map equation}}
    \begin{cases}
    \frac{\partial\phi_t}{\partial t} = \Delta_{g(t),\hat{g}}\phi_t, \quad&\text{on}\quad M\times (0,T]\\
     \varphi_0 = \psi,\quad&\text{on}\quad M.
    \end{cases}
\end{align}
We seek short time existence and uniqueness to the initial value problem (\ref{harmonic map equation}). To do that, we reformulate the problem into an equivalent equation on $TM$ via the exponential map. Since $\psi\in C^{1,\beta}(M)$, we can find a $C^{\infty}$ map $\hat{\psi}:(M,g)\to (M,\hat{g})$ such that $d_{\hat{g}}(\psi(x),\hat{\psi}(x)) < \frac{1}{2}{\rm inj}(M,\hat{g})$. Here $d_{\hat{g}}$ is the Riemannian distance with respect to the metric $\hat{g}$. Thus we can write
\begin{align}
	\psi(x) = \exp_{\hat{\psi}(x)}(U(x))	
\end{align}
	for some $C^{1+\beta}$ vector field $U(x)$. The exponential map is taken with respect to the metric $\hat{g}$. If we assume that $T$ is sufficiently small so that $d_{\hat{g}}(\phi_t(x),\hat{\psi}(x)) < \frac{3}{4}{\rm inj}(M,\hat{g})$, then we can write the harmonic map heat flow $\phi_t(x)$ in the form
	\begin{align}
	\phi_t(x) = 	\exp_{\hat{\psi}(x)}(V(x,t))
	\end{align}
	for some vector field $V(x,t)$. Note that the assumptions on the injectivity radius ensures that $V(x,t)$ is well-defined. Now the idea is to transform the initial value problem (\ref{harmonic map equation}) into an equivalent PDE for the vector field $V(x,t)$ with initial condition $U(x)$. The following lemma gives such an equivalence. 
	
	\bigskip
	
	\begin{lemma}{\label{lem pullback harmonic}}
	If $\phi_t(x)\in\Gamma(M\times [0,T];M)$ is a solution to the harmonic map heat flow (\ref{harmonic map equation}), then the vector field $V(x,t)$ defined by $\phi_t(x) = 	\exp_{\hat{\psi}(x)}(V(x,t))$ is a solution in	$\Gamma(M\times [0,T];\hat{\psi}^*(TM))$ to the initial value problem
	\begin{align}{\label{pullback harmonic map}}
	    \begin{cases}
    \left(\frac{\partial}{\partial t} - tr_g\hat{\nabla}^2\right) V = \mathcal{P}(V, \hat{\nabla}V) \quad&\text{on}\quad M\times (0,T]\\
     V(x,0) = U(x)\quad&\text{on}\quad M,
    \end{cases}	
	\end{align}
	provided that $T$ is sufficiently small, where
	\begin{align}
		\mathcal{P}(V, \hat{\nabla}V)^a := g^{ij}(\Gamma_{\hat{g}} - \Gamma_g)_{ij}^k\ (\hat{\nabla}_kV^a + Z_k^a( V))+ g^{ij}S_{ij}^a( V, \hat{\nabla}V),
	\end{align}
	$Z(V)$ and $S(V, \hat{\nabla}V)$ are sections of $T^*M\otimes\hat{\psi}^*TM$ and $T^*M\otimes T^*M\otimes\hat{\psi}^*TM$ respectively.
Here $Z_k^a(V)$ is a smooth function in $V$, whereas $S_{ij}^a( V, \hat{\nabla}V)$ is a smooth function in $V$ and a polynomial of degree 2 in $\hat{\nabla}V$.
	Moreover, the converse is also true.
	\end{lemma}

\bigskip

\begin{proof}
		Let $\{x^i\}$, $\{y^a\}$ and $\{z^c\}$ be local coordinates around $x$, $\psi(x)$ and $\phi_t(x)$ respectively. We first show that
		\begin{align}{\label{dphi expression}}
		d\phi_t(x)(\frac{\partial}{\partial x^i}) = (d\exp_{\hat{\psi}(x)})_{V(x,t)}	\left(\hat{\nabla}_iV(x,t) + Z_i(V(x,t)) \right) ,
		\end{align}
		where $Z_i = Z_i(x, V(x,t))\in T_{\hat{\psi}(x)}M$ is a vector field depending smoothly on $x$ and $V$.
		
Fix $x$ and $t$, let $\lambda(\tau):= \exp_{\hat{\psi}(x)}(\tau V(x,t)) $ be a geodesic. Let $\gamma(s)$ be a curve such that $\gamma(0) = x$ and $\gamma'(0) = \frac{\partial}{\partial x^i}$. Let $F(s,\tau):= \exp_{\hat{\psi}(\gamma(s))}(\tau V(\gamma(s),t))$ be a variation of $\lambda$ through geodesics. Then 
	\begin{align*}
		J(\tau) := \frac{\partial F}{\partial s}(s,\tau)\Big|_{s=0}
	\end{align*}
	is a Jacobi field with initial conditions $J(0) = d\hat{\psi}(x)(\frac{\partial}{\partial x^i}),\ \hat{\nabla}_{\tau}J(0) = \hat{\nabla}_XV(x,t)$ such that $J(1) = d\phi_t(x)(\frac{\partial}{\partial x^i})$. Let us decompose the Jacobi field $J := J_1 + J_2$ in a way that $J_1(0) = 0,\ \hat{\nabla}_{\tau}J_1(0) =  \hat{\nabla}_iV(x,t)$ and $J_2(0) = d\hat{\psi}(x)(\frac{\partial}{\partial x^i}),\ \hat{\nabla}_{\tau}J_2(0) = 0$. Then we have
	\begin{align*}
		J_1(1)\ = (d\exp_{\hat{\psi}(x)})_{V(x,t)}	\left(\hat{\nabla}_iV(x,t) \right).
	\end{align*}
	
Moreover, the vector $J_2(1)$ depends smoothly on $x$ and $V(x,t)$. To see that, let\\ $\sigma(u) := \exp_{\hat{\psi}(x)}(u\ J_2(0))$ be a geodesic, and let $H(u,\tau):= \exp_{\sigma(u)}\big(\tau\ P_{\sigma(u)}(V(x,t))\big)$ be another geodesic variation of $\lambda(\tau)$ where $P_{\sigma(u)}(V(x,t))$ is the parallel transport of $V(x,t)$ through $\sigma(u)$. Note that the Jacobi field $J_2(\tau)$ arises from the variation $H$ since
	$$ \frac{\partial H}{\partial u}(u,\tau)\Big|_{u=0, \tau=0} = J_2(0)\quad \text{and}\quad  D_{\tau}\frac{\partial H}{\partial u}(u,\tau)\Big|_{u=0, \tau=0} = 0.$$
Then
	$$ \tilde{Z}_i(x, V(x,t)) := J_2(1) = \frac{\partial H}{\partial u}(u, 1)\Big|_{u=0} \in T_{\exp_{\hat{\psi}(x)}(V)}M $$
is a vector field depending smoothly on $x$ and $V(x,t)$. We then define the vector field\\ $Z_i(x, V(x,t))\in T_{\hat{\psi}(x)}M$ by
	\[(d\exp_{\hat{\psi}(x)})_{V(x,t)}(Z_i(x, V(x,t))) = \tilde{Z}_i(x, V(x,t)). \]
	Combining the above results, we obtain the identity (\ref{dphi expression}).

\bigskip
Next, let $\bar{\nabla}$ be the connection on $T^*M\otimes\phi_t^*(TM)$ induced by $\nabla_{g(t)}$ and $\phi_t^*\hat{\nabla}$. It is worth noting that although we are using the connection $\bar{\nabla} = \nabla_{g(t)}\otimes 1 + 1\otimes \phi_t^*\hat{\nabla}$ which applies to sections of the form $dx^i\otimes\phi_t^*\partial_{z^c}$, we will be using the basis $\{(d\exp_{\hat{\psi(x)}})_{V(x,t)}(\partial_{y^a})\}$ on the fibre $T_{\phi_t(x)}M$. Moreover, we denote by $\omega_a^c$ the connection 1-forms for the connection $\phi_t^*\hat{\nabla}$ with respect to the basis $\{(d\exp_{\hat{\psi(x)}})_{V(x,t)}(\partial_{y^a})\}$, thus   
	\[\hat{\nabla}_{(d\exp_{\hat{\psi}})_{V}(\partial_{y^a})}(d\exp_{\hat{\psi}})_{V}(\partial_{y^b}) = \omega_{ab}^c\cdot (d\exp_{\hat{\psi}})_{V}(\partial_{y^c}). \]
	Note that $\omega_a^c$ depend smoothly on $V(x,t)$. 
	 Then (\ref{dphi expression}) implies 
\begin{align*}
	&\bar{\nabla}d\phi_t \left(	\frac{\partial}{\partial x^i}, \frac{\partial}{\partial x^j}\right)\\
	&= (\phi_t^*\hat{\nabla})_{\frac{\partial}{\partial x^i}}\left(d\phi_t \left(	\frac{\partial}{\partial x^j}\right)\right) - d\phi_t\left((\nabla_{g(t)})_{\frac{\partial}{\partial x^i}}\frac{\partial}{\partial x^j}\right) \\
	&= (\phi_t^*\hat{\nabla})_{\frac{\partial}{\partial x^i}}\left((\hat{\nabla}_jV^b + Z_j^b)(d\exp_{\hat{\psi}})_{V}	\left(\frac{\partial}{\partial y^b}\right)\right) - d\phi_t\left((\Gamma_g)_{ij}^k\frac{\partial}{\partial x^k}\right)\\
	&= \left(\frac{\partial}{\partial x^i}\hat{\nabla}_jV^b + \frac{\partial}{\partial x^i}Z_j^b\right)(d\exp_{\hat{\psi}})_{V}	\left(\frac{\partial}{\partial y^b}\right)\\
	&\quad\quad + (\hat{\nabla}_jV^b + Z_j^b)\cdot \hat{\nabla}_{(\hat{\nabla}_iV^a + Z_i^a)(d\exp_{\hat{\psi}})_{V}	\left(\frac{\partial}{\partial y^a}\right)}\left((d\exp_{\hat{\psi}})_{V}	\left(\frac{\partial}{\partial y^b}\right)\right)\\
	&\quad\quad - (\Gamma_g)_{ij}^k (\hat{\nabla}_kV^c + Z_k^c)(d\exp_{\hat{\psi}})_{V}	\left(\frac{\partial}{\partial y^c}\right)\\
	&= \left(\hat{\nabla}_i\hat{\nabla}_jV^c + \hat{\nabla}_iZ_j^c\right)(d\exp_{\hat{\psi}})_{V}	\left(\frac{\partial}{\partial y^c}\right) + (\hat{\nabla}_jV^b + Z_j^b)(\hat{\nabla}_iV^a + Z_i^a)\cdot 	\omega_{ab}^c \left(\frac{\partial}{\partial y^c}\right)\\
	&\quad\quad + (\Gamma_{\hat{g}} - \Gamma_g)_{ij}^k (\hat{\nabla}_kV^c + Z_k^c)(d\exp_{\hat{\psi}})_{V}	\left(\frac{\partial}{\partial y^c}\right)\\
	&= \left( \hat{\nabla}_i\hat{\nabla}_jV^c +(\Gamma_{\hat{g}} - \Gamma_g)_{ij}^k\ (\hat{\nabla}_kV^c + Z_k^c) + \hat{\nabla}_iV^a\ \hat{\nabla}_jV^b \cdot\omega_{ab}^c \right)\ (d\exp_{\hat{\psi}})_{V}	\left(\frac{\partial}{\partial y^c} \right)\\
	&\quad\quad + \left( \hat{\nabla}_iZ_j^c  + (Z_i^a \hat{\nabla}_jV^b + Z_j^b\hat{\nabla}_iV^a + Z_i^aZ_j^b)\cdot\omega_{ab}^c\right)\ (d\exp_{\hat{\psi}})_{V}	\left(\frac{\partial}{\partial y^c} \right).
\end{align*}
This means that
\begin{align}{\label{Hessian phi}}
	&\bar{\nabla}d\phi_t \left(	\frac{\partial}{\partial x^i}, \frac{\partial}{\partial x^j}\right)\\
	& = \left( \hat{\nabla}_j\hat{\nabla}_iV^a +(\Gamma_{\hat{g}} - \Gamma_g)_{ij}^k\ (\hat{\nabla}_kV^a + Z_k^a( V)) + S_{ij}^a(V, \hat{\nabla}V) \right)(d\exp_{\hat{\psi}})_{V}	\left(\frac{\partial}{\partial y^a} \right)\notag.
\end{align}
Here $S_{ij}^a(V, \hat{\nabla}V) :=  \hat{\nabla}_iZ_j^a + (Z_i^b\hat{\nabla}_jV^c + Z_j^c\hat{\nabla}_iV^b + Z_i^bZ_j^c)\cdot\omega_{bc}^a + \hat{\nabla}_iV^b \hat{\nabla}_jV^c\cdot \omega_{bc}^a$ is a smooth function in $V(x,t)$ and a polynomial of degree 2 in $\hat{\nabla}V(x,t)$.
The Laplacian $\Delta_{g(t),\hat{g}}\phi_t$ is thus given by
\begin{align*}
	 &\Delta_{g(t),\hat{g}}\phi_t\\
	=&\ \text{tr}_g\nabla d\phi_t\\
	=& \left( \text{tr}_g\hat{\nabla}^2 V^a\ + g^{ij}(\Gamma_{\hat{g}} - \Gamma_g)_{ij}^k\ (\hat{\nabla}_kV^a + Z_k^a) + g^{ij}S_{ij}^a \right)(d\exp_{\hat{\psi}})_{V}	\left(\frac{\partial}{\partial y^a} \right).
\end{align*}
On the other hand, it is easy to see that
\begin{align*}
	\frac{\partial}{\partial t}\phi_t = 	(d\exp_{\hat{\psi}})_{V} \left(\frac{\partial}{\partial t}V\right) = \left(\frac{\partial}{\partial t}V^a\right)\ (d\exp_{\hat{\psi}})_{V}\left(\frac{\partial}{\partial y^a} \right).
\end{align*}
Therefore, we conclude that
\begin{align}
	&\frac{\partial\phi_t}{\partial t} - \Delta_{g(t),\hat{g}}\phi_t	\\
	\notag = & \left(\frac{\partial}{\partial t}V^a - \text{tr}_g\hat{\nabla}^2 V^a\ - g^{ij}(\Gamma_{\hat{g}} - \Gamma_g)_{ij}^k\ (\hat{\nabla}_kV^a + Z_k^a) - g^{ij}S_{ij}^a \right)(d\exp_{\hat{\psi}})_{V}	\left(\frac{\partial}{\partial y^a} \right).
\end{align}
The assertion then follows since the exponential map $\exp$ is a smooth diffeomorphism with respect to the metric $\hat{g}$.
\end{proof}

\bigskip

Similar to the proof of the short time existence and uniqueness to the Ricci De-Turck flow, the short time existence and uniqueness to (\ref{pullback harmonic map}) can be obtained by applying the Banach fixed point theorem to the following linear system:
	\begin{align}{\label{linear harmonic map}}
	\begin{cases}
    \left(\frac{\partial}{\partial t} - tr_g\hat{\nabla}^2\right) V = \mathcal{P}(W, \hat{\nabla}W) \quad&\text{on}\quad M\times (0,T]\\
     V(x,0) = U(x)\quad&\text{on}\quad M,
    \end{cases}	
	\end{align}
	where
	\begin{align}
		\mathcal{P}(W, \hat{\nabla}W)^a := g^{ij}(\Gamma_{\hat{g}} - \Gamma_g)_{ij}^k\ (\hat{\nabla}_kW^a + Z_k^a( W)) + g^{ij}S_{ij}^a( W, \hat{\nabla}W).
	\end{align}
Moreover, we note that the fact $U\in C^{1,\beta}(M)$ implies that $U\in C^{\delta}(M)$ for any exponent  $\delta\in (0,1)$ which can be arbitrarily close to $1$.

\begin{proposition}{\label{prop linear estimate harmonic}}
	Let $k\geq 2$ and $\delta\in (0,1)$ be given such that $\delta\geq\beta$. Suppose that
\begin{itemize}
\item $W(x,0) = U(x)$;
\item $\|W\|_{\mathcal{X}_{k,\gamma}^{(\delta)}(M\times [0,T])}\leq B$.
\end{itemize}
Then there exists a unique solution $V\in \mathcal{X}_{k+1,\gamma}^{(\delta)}(M\times [0,T];\hat{\psi}^*(TM))$ to the linear system (\ref{linear harmonic map}) and there are positive constants $ K_1 = K_1(\hat{g}, M,  \|g\|_{\mathcal{X}_{k,\gamma}^{(\beta)}})$, $\ K_2 = K_2(\hat{g}, M, \|g\|_{\mathcal{X}_{k,\gamma}^{(\beta)}}, B)$ such that
\begin{align*}
	\|V\|_{\mathcal{X}_{k+1,\gamma}^{(\delta)}(M\times [0,T])} \leq K_1(K_2\ T^{\frac{1}{2}} + \|U\|_{C^{1,\beta}(M)}).
\end{align*}
\end{proposition}

\begin{proof}\

We first claim that the assumption $\|W\|_{\mathcal{X}_{k,\gamma}^{(\delta)}(M\times [0,T])}\leq B$ implies
\begin{align}{\label{P estimate}}
	\|\mathcal{P}(W, \hat{\nabla}W)\|_{\mathcal{C}^{k-1,\gamma}_{1-\frac{\beta}{2}-\frac{\delta}{2}}(M\times (0,T])} \leq K(\hat{g}, \|g\|_{\mathcal{X}_{k,\gamma}^{(\beta)}}, B).
\end{align}

To see that, we estimate
	\begin{align*}
	&\left\|g^{ij}(\Gamma_{\hat{g}} - \Gamma_g)_{ij}^k\ (\hat{\nabla}_kW + Z_k)\right\|_{\mathcal{C}^{k-1,\gamma}_{1-\frac{\beta}{2}-\frac{\delta}{2}}(M\times (0,T])} \\
	&\leq K(\hat{g})\ \|g^{-1}\|_{\mathcal{C}^{k-1,\gamma}_{0}(M\times (0,T])}\|\Gamma_{\hat{g}} - \Gamma_g\|_{\mathcal{C}^{k-1,\gamma}_{\frac{1}{2}-\frac{\beta}{2}}(M\times (0,T])}\|\hat{\nabla}W +Z\|_{\mathcal{C}^{k-1,\beta}_{\frac{1}{2}-\frac{\delta}{2}}(M\times (0,T])}\\
	&\leq K(\hat{g}, \|g\|_{\mathcal{X}_{k,\gamma}^{(\beta)}}, B)\ 
	\end{align*}
and
	\begin{align*}
	&\left\| g^{ij}S_{ij}^a(W, \hat{\nabla}W) \right\|_{\mathcal{C}^{k-1,\gamma}_{1-\frac{\beta}{2} - \frac{\delta}{2}}(M\times (0,T])} \\
	&\leq K(\hat{g})\  T^{\frac{\delta}{2}-\frac{\beta}{2}}\ \|g^{-1}\|_{\mathcal{C}^{k-1,\gamma}_{0}(M\times (0,T])}\|\hat{\nabla}W\|^2_{\mathcal{C}^{k-1,\gamma}_{\frac{1}{2}-\frac{\delta}{2}}(M\times (0,T])}\\
	&\leq K(\hat{g}, \|g\|_{\mathcal{X}_{k,\gamma}^{(\beta)}}, B) T^{\frac{\delta}{2}-\frac{\beta}{2}}.
	\end{align*}
Hence the estimate (\ref{P estimate}) is established. By Lemma \ref{properties of weighted space}, this in particular implies that
\begin{align*} 
	\|\mathcal{P}(W, \hat{\nabla}W)\|_{\mathcal{C}^{k-1,\gamma}_{1-\frac{\delta}{2}}(M\times (0,T])} \leq K(\hat{g}, \|g\|_{\mathcal{X}_{k,\gamma}^{(\beta)}}, B)\ T^{\frac{\beta}{2}}.
\end{align*}

 Since the initial condition in (\ref{linear harmonic map}) satisfies $U\in C^{\delta}(M;\hat{\psi}^*(TM))$, then by Theorem \ref{theorem linear equation} there exists a unique solution $V$ to the linear system (\ref{linear harmonic map}) such that $V\in C^{\delta,\frac{\delta}{2}}(M\times [0,T])$ and $\hat{\nabla}V\in \mathcal{C}^{k,\gamma}_{\frac{1}{2}-\frac{\delta}{2}}(M\times (0,T]).$ Moreover, $V$ satisfies the estimate
\begin{align}
	&\|V\|_{\delta,\frac{\delta}{2};M\times [0,T]}\ + \|\hat{\nabla}V\|	_{\mathcal{C}^{k,\gamma}_{\frac{1}{2}-\frac{\delta}{2}}(M\times (0,T])}\\
	&\notag\leq K_1 \left(\|\mathcal{P}(W, \hat{\nabla}W)\|_{\mathcal{C}^{k-1,\gamma}_{1-\frac{\delta}{2}}(M\times (0,T])} + \|U\|_{\delta;M}  \right)\\
	&\leq K_1(K_2\ T^{\frac{\beta}{2}} +\|U\|_{C^{1,\beta}(M)})\notag,
\end{align}
where $K_1 = K_1(\hat{g}, M,  \|g\|_{\mathcal{X}_{k,\gamma}^{(\beta)}})$ and $K_2 = K_2(\hat{g}, M, \|g\|_{\mathcal{X}_{k,\gamma}^{(\beta)}}, B)$. 
\end{proof}

\bigskip

Now, we choose $B > 2K_1\|U\|_{C^{1,\beta}(M)}$ to be a large positive constant. For $k\geq 2$, we define a closed subset $\mathcal{W}$ in $\mathcal{X}_{k,\gamma}^{(\delta)}(M\times [0,T])$ by
\begin{align*}
	\mathcal{W}:= \{ W\in\mathcal{X}_{k,\gamma}^{(\delta)}(M\times [0,T]) |\ \|W\|_{\mathcal{X}_{k,\gamma}^{(\delta)}(M\times [0,T])}\leq B\}.	
\end{align*}
Next, we define an operator $\mathcal{S}:\mathcal{W}\to\mathcal{X}_{k,\gamma}^{(\delta)}(M\times [0,T])$ by
\begin{align*}
	V := \mathcal{S}(W),	
\end{align*}
where $V$ is the unique solution to the system (\ref{linear harmonic map}) in $\mathcal{X}_{k,\gamma}^{(\delta)}(M\times [0,T])$. By Proposition \ref{prop linear estimate harmonic} and our choice of $B$, we can make $\|V\|_{\mathcal{X}_{k,\gamma}^{(\delta)}(M\times [0,T])}\leq B$ provided that $T$ is sufficiently small. Consequently $\mathcal{S}(\mathcal{W})\subset\mathcal{W}$.

\bigskip

\begin{proposition}{\label{prop contraction harmonic}}
	If $T = T(\hat{g}, M,  \|g\|_{\mathcal{X}_{k,\gamma}^{(\beta)}}, B)$ is chosen sufficiently small, the operator $\mathcal{S}$ is a contraction mapping.
\end{proposition}

\begin{proof}
In the sequel, $K$ will denote a constant depending only on $\hat{g}, M, \|g\|_{\mathcal{X}_{k,\gamma}^{(\beta)}}, B$. \\

Let $W_1, W_2\in\mathcal{W}$ and write $V_i:=\mathcal{S}(W_i)$ for $i=1,2$. Then $V = V_1 - V_2$ solves the system
\begin{align}
    \begin{cases}
    \left(\frac{\partial}{\partial t} - tr_g\hat{\nabla}^2\right)V =  \mathcal{P}(W_1, \hat{\nabla}W_1) - \mathcal{P}(W_2, \hat{\nabla}W_2) &\quad\text{on}\quad M\times (0,T]\\
    V|_{t=0} = 0 &\quad\text{on}\quad M.
    \end{cases}
\end{align}

We recall that
\begin{align*}
	\mathcal{P}(W, \hat{\nabla}W)^a = g^{ij}(\Gamma_{\hat{g}} - \Gamma_g)_{ij}^k\ (\hat{\nabla}_kW^a + Z_k^a( W))+ g^{ij}S_{ij}^a( W, \hat{\nabla}W).	
\end{align*}
Now, we define the tensors $Z(s)\in\Gamma(M; T^*M\otimes\hat{\psi}^*TM)$ and $S(s)\in\Gamma(M; T^*M\otimes T^*M\otimes\hat{\psi}^*TM)$ by
	\[Z(s) := Z(sW_1 + (1-s)W_2))\]
	and
	\[S(s) := S(sW_1 + (1-s)W_2,\ s\hat{\nabla}W_1 + (1-s)\hat{\nabla}W_2).\]
	Note that $(1-s)W_1 + sW_2\in\mathcal{W}$ for all $s\in [0,1]$.
From Lemma \ref{lem pullback harmonic}, we observe that
\begin{itemize}
\item $\nabla_{\bold{q}}Z_k^a(x,\bold{q})$ is a smooth function in $x$ and $\bold{q}$;
\item $\nabla_{\bold{q}}S_{ij}^a(x, \bold{q}, \bold{A})$ is a smooth function in $x$ and $\bold{q}$, and a polynomial of degree two in $\bold{A}$;
\item $\nabla_{\bold{A}}S_{ij}^a(x, \bold{q}, \bold{A})$ is a smooth function in $x$ and $\bold{q}$, and a polynomial of degree one in $\bold{A}$.
\end{itemize}
Let $\{dx^i\otimes\frac{\partial}{\partial y^a}\}$ be the local frame of $T^*M\otimes\hat{\psi}^*TM$ on the chart $U_{ia}$, then from the above observations and Lemma \ref{properties of weighted space}, we have
\begin{align*}
	\|\frac{\partial}{\partial s}Z(s)\|_{\mathcal{C}^{k-2,\gamma}_{0}(M\times (0,T])}
	&= \sum_{i,a}\|\frac{\partial}{\partial s}Z(s)_i^a\|_{\mathcal{C}^{k-2,\gamma}_{0}(U_{ia}\times (0,T])} \\
	&= \sum_{i,a}\|\langle D_{\bold{q}}Z(s)_i^a,\ W_1 - W_2\rangle\|_{\mathcal{C}^{k-2,\gamma}_{0}(U_{ia}\times (0,T])}\\
	&\leq \sum_{i,a}K\sup_{s\in [0,1]}\| D_{\bold{q}}Z(s)_k^a\|_{\mathcal{C}^{k-2,\gamma}_{0}(U_{ia}\times (0,T])} \|W_1-W_2\|_{\mathcal{C}^{k-2,\gamma}_{0}(U_{ia}\times (0,T])}\\
	&\leq K \|W_1-W_2\|_{\mathcal{X}_{k,\gamma}^{(\delta)}(M\times [0,T])}.
\end{align*}
Hence
\begin{align}{\label{Z_1-Z_2}}
	\|Z(W_1) - Z(W_2)\|_{\mathcal{C}^{k-2,\gamma}_{0}(M\times (0,T])} \leq \|\frac{\partial}{\partial s}Z(s)\|_{\mathcal{C}^{k-2,\gamma}_{0}(M\times (0,T])} \leq K  \|W_1-W_2\|_{\mathcal{X}_{k,\gamma}^{(\delta)}(M\times [0,T])}.
\end{align}
Let $\{dx^i\otimes dx^j\otimes\frac{\partial}{\partial y^a}\}$ be the local frame of $T^*M\otimes T^*M\otimes\hat{\psi}^*TM$ on the chart $U_{i j a}$, then we similarly have
\begin{align*}
	&\|\frac{\partial}{\partial s}S(s)\|_{\mathcal{C}^{k-2,\gamma}_{1-\frac{\delta}{2}}(M\times (0,T])}\\
	&= \sum_{i,j,a}\|\frac{\partial}{\partial s}S(s)_{ij}^a\|_{\mathcal{C}^{k-2,\gamma}_{1-\frac{\delta}{2}}(U_{ija}\times (0,T])}\\
	&\leq \sum_{i,j,a}\|\langle D_{\bold{q}}S(s)_{ij}^a,\ W_1 - W_2\rangle\|_{\mathcal{C}^{k-2,\gamma}_{1-\frac{\delta}{2}}(U_{ija}\times (0,T])} + \|\langle D_{\bold{A}}S(s)_{ij}^a,\ \hat{\nabla}W_1 - \hat{\nabla}W_2\rangle\|_{\mathcal{C}^{k-2,\gamma}_{1-\frac{\delta}{2}}(U_{ija}\times (0,T])}\\
	&\leq \sum_{i,j,a} KT^{\frac{\delta}{2}}\sup_{s\in[0,1]}\| D_{\bold{q}}S(s)_{ij}^a\|_{\mathcal{C}^{k-2,\gamma}_{1-\delta}(U_{ija}\times (0,T])}\|W_1 - W_2\|_{\mathcal{C}^{k-2,\gamma}_{0}(U_{ija}\times (0,T])}\\
	&\quad\quad + \sum_{i,j,a}  KT^{\frac{\delta}{2}}\sup_{s\in[0,1]}\| D_{\bold{A}}S(s)_{ij}^a\|_{\mathcal{C}^{k-2,\gamma}_{\frac{1}{2}-\frac{\delta}{2}}(U_{ija}\times (0,T])}\|\hat{\nabla}W_1 - \hat{\nabla}W_2\|_{\mathcal{C}^{k-2,\gamma}_{\frac{1}{2}-\frac{\delta}{2}}(U_{ija}\times (0,T])}\\
	&\leq  KT^{\frac{\delta}{2}} \|W_1 - W_2\|_{\mathcal{X}_{k,\gamma}^{(\delta)}(M\times [0,T])}.
\end{align*}
This gives
\begin{align}{\label{S_1-S_2}}
	\|S(W_1, \hat{\nabla}W_1) - S(W_2, \hat{\nabla}W_2)\|_{\mathcal{C}^{k-2,\gamma}_{1-\frac{\delta}{2}}(M\times (0,T])} &\leq \|\frac{\partial}{\partial s}S(s)\|_{\mathcal{C}^{k-2,\gamma}_{1-\frac{\delta}{2}}(M\times (0,T])}\\
	&\notag\leq K  \|W_1-W_2\|_{\mathcal{X}_{k,\gamma}^{(\delta)}(M\times [0,T])}.
\end{align}
From (\ref{Z_1-Z_2}), (\ref{S_1-S_2}) and Lemma \ref{properties of weighted space} we obtain
\begin{align*}
	&\|\mathcal{P}(W_1, \hat{\nabla}W_1) - \mathcal{P}(W_2, \hat{\nabla}W_2)\|_{\mathcal{C}^{k-2,\gamma}_{1-\frac{\delta}{2}}(M\times (0,T])}\\
	&\leq \|g^{-1}*(\Gamma_{\hat{g}}-\Gamma_g)* \hat{\nabla} (W_1-W_2)\|_{\mathcal{C}^{k-2,\gamma}_{1-\frac{\delta}{2}}(M\times (0,T])} \\
	&\quad + \|g^{-1}*(\Gamma_{\hat{g}}-\Gamma_g)* (Z(W_1) - Z(W_2))\|_{\mathcal{C}^{k-2,\gamma}_{1-\frac{\delta}{2}}(M\times (0,T])}\\
	&\quad  +  \|g^{-1}  * (S( W_1, \hat{\nabla}W_1) - S(W_2, \hat{\nabla}W_2))\|_{\mathcal{C}^{k-2,\gamma}_{1-\frac{\delta}{2}}(M\times (0,T])}\\
	&\leq K\ \Big(T^{\frac{\beta}{2}}\|\Gamma_{\hat{g}}-\Gamma_g\|_{\mathcal{C}^{k-2,\gamma}_{\frac{1}{2}-\frac{\beta}{2}}(M\times (0,T])} \|\hat{\nabla} (W_1-W_2)\|_{\mathcal{C}^{k-2,\gamma}_{\frac{1}{2}-\frac{\delta}{2}}(M\times (0,T])}\\
	&\quad\quad\quad +T^{\frac{1}{2}+\frac{\beta}{2}-\frac{\delta}{2}}\ \|\Gamma_{\hat{g}}-\Gamma_g\|_{\mathcal{C}^{k-2,\gamma}_{\frac{1}{2}-\frac{\beta}{2}}(M\times (0,T])}\|Z( W_1) - Z( W_2)\|_{\mathcal{C}^{k-2,\gamma}_{0}(M\times (0,T])}\\
	&\quad\quad\quad + \|S(W_1, \hat{\nabla}W_1) - S( W_2, \hat{\nabla}W_2) \|_{\mathcal{C}^{k-2,\gamma}_{1-\frac{\delta}{2}}(M\times (0,T])}\Big)\\
	&\leq K\ T^{\frac{\beta}{2}}\ \|W_1 - W_2\|_{\mathcal{X}_{k,\gamma}^{(\delta)}(M\times [0,T])}.
	\end{align*}
	Theorem \ref{theorem linear equation} then implies 
	\begin{align*}
		\|V\|_{\delta,\frac{\delta}{2};M\times [0,T]}\ + \|\hat{\nabla}V\|	_{\mathcal{C}^{k-1,\gamma}_{\frac{1}{2}-\frac{\delta}{2}}(M\times (0,T])} \leq K\ T^{\frac{\beta}{2}}\ \|W_1 - W_2\|_{\mathcal{X}_{k,\gamma}^{(\delta)}(M\times [0,T])}.
	\end{align*}
	Therefore,
	\begin{align}
		\|V_1 - V_2\|_{\mathcal{X}_{k,\gamma}^{(\delta)}(M\times [0,T])} \leq K\ T^{\frac{\beta}{2}}\ \|W_1 - W_2\|_{\mathcal{X}_{k,\gamma}^{(\delta)}(M\times [0,T])}.
	\end{align}
This proves the proposition.
\end{proof}

\bigskip

From the previous proposition and Lemma \ref{lem pullback harmonic}, the harmonic map heat flow (\ref{harmonic map equation}) has a unique solution $\phi_t$ such that $\phi_t\in C^{\delta,\frac{\delta}{2}}(M\times [0,T])$ and $d\phi_t\in \mathcal{C}^{k-1,\gamma}_{\frac{1}{2}-\frac{\delta}{2}}(M\times (0,T])$ provided that $g\in\mathcal{X}_{k,\gamma}^{(\beta)}(M\times[0,T])$ and that $T = T(\hat{g}, M,  \|g\|_{\mathcal{X}_{k,\gamma}^{(\beta)}}, \|\psi\|_{C^{1,\beta}(M)})$ is chosen sufficiently small. Furthermore, as the initial condition satisfies $\psi\in C^{1,\beta}(M)$, we can choose $\delta$ to be arbitrarily close to $1$. It turns out that the regularity of $\phi_t$ can be improved.\\

\begin{theorem}{\label{thm harmonic map}}
Let $\lambda\in (\gamma, \beta)$ be given. If $T = T(\hat{g}, M,  \|g\|_{\mathcal{X}_{k,\gamma}^{(\beta)}}, \|\psi\|_{C^{1,\beta}(M)})$ is chosen sufficiently small, then there exists a unique solution $V(x,t)$ to the initial value problem (\ref{pullback harmonic map}) such that
\begin{itemize}
\item[(1)] $V \in C^{1+\lambda,\frac{1+\lambda}{2}}(M\times [0,T])\ $ and\quad  $\hat{\nabla}^2V \in \mathcal{C}^{k-1,\gamma}_{\frac{1}{2}-\frac{\lambda}{2}}(M\times (0,T])$;	
\item[(2)] 
	$\|V\|_{1+\lambda,\frac{1+\lambda}{2}; M\times [0,T]} + \|\hat{\nabla}^2V\|_{\mathcal{C}^{k-1,\gamma}_{\frac{1}{2}-\frac{\lambda}{2}}( M\times (0,T])} \leq K (\hat{g},  M,  \|g\|_{\mathcal{X}_{k,\gamma}^{(\beta)}}, \|U\|_{C^{1,\beta}(M)}).$
\end{itemize}
\end{theorem}

\begin{proof}
In the sequel of the proof, $K$ will always denote a constant depending only on $\hat{g}, M,  \|g\|_{\mathcal{X}_{k,\gamma}^{(\beta)}}, \|U\|_{C^{1,\beta}(M)}$.  

Let us take $\delta = 1+\lambda - \beta$ in the definition of $\mathcal{X}_{k,\gamma}^{(\delta)}(M\times [0,T];\hat{\psi}^*(TM))$. Note that $\lambda\in (\gamma,\beta)$ implies $\delta<1$. By Proposition \ref{prop contraction harmonic}, we can find a unique solution $V\in  \mathcal{X}_{k,\gamma}^{(\delta)}(M\times [0,T])$ to the initial value problem (\ref{pullback harmonic map}) if $T$ is chosen sufficiently small. Since $\|\hat{\nabla}U\|_{\lambda;M}\leq K$, we claim that
	\begin{align}{\label{V = U in C^1}}
		\lim_{t\to 0^+}\hat{\nabla}_iV(x,t) = \hat{\nabla}_iU(x).
	\end{align}
For instance, let us consider an arbitrary chart $\varphi: U\to\mathbb{R}^n$, so that on this chart $V = V^a\frac{\partial}{\partial y^a}$ and we abbreviate $V^a(x,t) = V^a(\varphi^{-1}(x), t)$ for $x\in\varphi(U)$. Then $V^a$ solves the equation
	\begin{align*}
	\begin{cases}
		\frac{\partial}{\partial t}V^a - g^{kl}D^2_{kl}V^a =  \tilde{\mathcal{P}}^a(V, \hat{\nabla}V) \quad&\text{on}\quad \varphi(U)\times (0,T]\\
		V^a(x,0) = U^a(x) \quad&\text{on}\quad \varphi(U)
	\end{cases},
	\end{align*}
where $\tilde{\mathcal{P}}^a(V, \hat{\nabla}V) = \mathcal{P}^a(V, \hat{\nabla}V) + (\partial\hat{\Gamma} + \hat{\Gamma}*_g\hat{\Gamma})*_gV + \hat{\Gamma}*_g\hat{\nabla}V$. Since $\|V\|_{\mathcal{X}_{k,\gamma}^{(\delta)}(M\times [0,T])}\leq K$, the estimate (\ref{P estimate}) implies that
\begin{align}{\label{tilde P estimate}}
	\|\tilde{\mathcal{P}}^a(V, \hat{\nabla}V)\|_{\mathcal{C}^{k-1,\gamma}_{\frac{1}{2}-\frac{\lambda}{2}}(U\times (0,T])} \leq K.
\end{align}
Moreover, the uniqueness part of Lemma \ref{lemma second order u} and (\ref{formula of u}) imply that $V^a$ has the form
\begin{align}{\label{DV integral}}
	D_iV^a(x,t)=- \int_0^t\int_{\mathbb{R}^n}D_i\Gamma (x,t;\xi,\tau)\tilde{\mathcal{P}}^a(V, \hat{\nabla}V) (\xi,\tau)d\xi d\tau + \int_{\mathbb{R}^n}D_i\Gamma (x,t;\xi,0)U^a(\xi)d\xi.
\end{align}
Using the estimates of fundamental solution (\ref{fundamental solution estimate}) and (\ref{tilde P estimate}), we have
\begin{align}{\label{tilde P integral estimate}}
&\int_0^t\int_{\mathbb{R}^n}|D_i\Gamma (x,t;\xi,\tau)\tilde{\mathcal{P}}^a (\xi,\tau) | d\xi d\tau\\
&\notag\leq K\int_0^t\int_{\mathbb{R}^n}(t-\tau)^{-\frac{n+1}{2}}\exp\left(-\frac{|x-\xi|^2}{K(t-\tau)}\right)|\tilde{\mathcal{P}}^a (\xi,\tau)|\ d\xi d\tau\\
&\notag \leq K\int_{0}^{t}\int_{\mathbb{R}^n}(t-\tau)^{-\frac{n+1}{2}}\exp\left(-\frac{|x-\xi|^2}{K(t-\tau)}\right)\tau^{-\frac{1}{2}+\frac{\lambda}{2}}\|\tilde{\mathcal{P}}^a\|_{\mathcal{C}^{k-1,\gamma}_{\frac{1}{2}-\frac{\lambda}{2}}(U\times (0,T])}\ d\xi d\tau\\
&\notag \leq K\int_{0}^{t}\int_0^{\infty}(t-\tau)^{-\frac{1}{2}}\rho^{n-1}\exp(-\frac{1}{K}\rho^2)\tau^{-\frac{1}{2}+\frac{\lambda}{2}}\ d\rho d\tau\\
&\notag  \leq Kt^{\frac{\lambda}{2}}.
\end{align}
For the second integral in (\ref{DV integral}), we note that by \cite[(11.13)]{Lad} the fundamental solution $\Gamma(x,t;\xi,0)$ can be written in the form
\[\Gamma(x,t;\xi,0) = \Gamma_0(x-\xi, t; \xi, 0) + \Gamma_1(x,t;\xi,0),\]
where the function $\Gamma_0(x-\xi,t; \xi,0)$ is defined in \cite[(11.2)]{Lad} which is the fundamental solution obtained by freezing the operator $\frac{\partial}{\partial t}-g^{kl}D^2_{kl}$ at the point $(\xi, 0)$. Moreover, $\Gamma_0$ also satisfies the estimates (\ref{fundamental solution estimate}) by \cite[(11.3)]{Lad}. On the other hand, since $g^{kl}\in C^{\beta,\frac{\beta}{2}}(\mathbb{R}^n\times[0,T])$, the minor term $\Gamma_1(x,t;\xi,0)$ satisfies the estimate 
\begin{align}{\label{minor fundamental solution estimate}}
    \Big|D_x\Gamma_1(x,t;\xi, 0)\Big|\leq K t^{-\frac{n+1-\beta}{2}}\exp\left(-\frac{|x-\xi|^2}{Kt}\right)\quad 
\end{align}
by \cite[~P.377]{Lad}. Since $\int_{\mathbb{R}^n}D_z\Gamma_0(z,t;\xi,0)dz = 0$ for any fixed $\xi$ by \cite[(11.5)]{Lad}, we also have $\int_{\mathbb{R}^n}D_x\Gamma_1(x,t;\xi,0)d\xi = 0$. Hence (\ref{minor fundamental solution estimate}) implies
\begin{align}{\label{Gamma_1 estimate}}
 &\left|\int_{\mathbb{R}^n}D_i\Gamma_1 (x,t;\xi,0)U^a(\xi)d\xi\right|	\\
 &\notag= \left|\int_{\mathbb{R}^n}D_i\Gamma_1 (x,t;\xi,0)(U^a(\xi) - U^a(x))d\xi\right|	\\
 &\notag\leq K\int_{\mathbb{R}^n} t^{-\frac{n+1-\beta}{2}}\exp\left(-\frac{|x-\xi|^2}{Kt}\right)\|\hat{\nabla}U\|_{0;M}|x-\xi|\ d\xi\\
 &\notag\leq  K\int_0^{\infty}t^{\frac{\beta}{2}} \rho^{n}\exp(-\frac{1}{K}\rho^2)\ d\rho\\
 &\notag\leq K t^{\frac{\beta}{2}}.
\end{align}
We next write
\begin{align}{\label{Gamma_0 integral}}
	&\int_{\mathbb{R}^n}	D_i\Gamma_0(x-\xi, t; \xi, 0)U^a(\xi) d\xi \\
	&\notag= \int_{\mathbb{R}^n}	D_i\Gamma_0(x-\xi, t; x, 0) U^a(\xi) d\xi\\
	&\notag\quad\quad + \int_{\mathbb{R}^n}	(D_i\Gamma_0(x-\xi, t; \xi, 0) - D_i\Gamma_0(x-\xi, t; x, 0))(U^a(\xi)- U^a(x)) d\xi\\
	&\notag:= I_1 + I_2.
\end{align}
For the second term $I_2$, we apply the estimate \cite[(11.4)]{Lad} to obtain
\begin{align}{\label{second Gamma_0 integral}}
	|I_2|
	& \leq K\int_{\mathbb{R}^n} |x-\xi|^{\beta}t^{-\frac{n+1}{2}}\exp\left(-\frac{|x-\xi|^2}{Kt}\right)|U^a(x) - U^a(\xi)|\ d\xi	\\
	&\notag\leq K\int_{\mathbb{R}^n} |x-\xi|^{1+\beta}t^{-\frac{n+1}{2}}\exp\left(-\frac{|x-\xi|^2}{Kt}\right)\|\hat{\nabla}U\|_{0;M}\ d\xi	\\
	 &\notag\leq  K\int_0^{\infty}t^{\frac{\beta}{2}} \rho^{n}\exp(-\frac{1}{K}\rho^2)\ d\rho\\
	 &\notag\leq K t^{\frac{\beta}{2}}.
\end{align}{\label{Gamma_0 integral 1}}
For the first term $I_1$, by noting that $D_i\Gamma_0(x-\xi, t; x, 0) = -D_{\xi^i}\Gamma_0(x-\xi, t; x, 0)$, we have
\begin{align}
	I_1 = 	\int_{\mathbb{R}^n}	\Gamma_0(x-\xi, t; x, 0)D_iU^a(\xi) d\xi.
\end{align}
Now, by putting (\ref{tilde P integral estimate}) to (\ref{second Gamma_0 integral}) into (\ref{DV integral}), we obtain
\begin{align*}
	\left| D_iV^a(x,t) - \int_{\mathbb{R}^n}\Gamma_0 (x-\xi,t;x,0)D_{i}U^a(\xi)d\xi \right|	\leq Kt^{\frac{\lambda}{2}}.
\end{align*}
Note that $\Gamma_0 (x-\xi,t;x,0)\to\delta(x-\xi)$ as $t\to 0^+$ in the sense of distribution. Hence by taking $t\to 0^+$, we obtain (\ref{V = U in C^1}).

\bigskip 

 Next, we observe that $\hat{\nabla}_iV(x,t)$ solves the system
	\begin{align}{\label{grad V equation}}
	    \begin{cases}
    \left(\frac{\partial}{\partial t} - tr_g\hat{\nabla}^2\right) (\hat{\nabla}_iV) = \tilde{\mathcal{P}}(V, \hat{\nabla}V, \hat{\nabla}^2V) \quad&\text{on}\quad M\times (0,T]\\
     \hat{\nabla}_iV(x,0) = \hat{\nabla}_iU(x)\quad&\text{on}\quad M,
    \end{cases}	
	\end{align}
where 
	$$\tilde{\mathcal{P}}(V, \hat{\nabla}V, \hat{\nabla}^2V) = \hat{\nabla}g^{-1}*\hat{\nabla}^2V + g^{-1}*\hat{R}*\hat{\nabla}V + g^{-1}*\hat{\nabla}\hat{R}*V + \hat{\nabla}_i(\mathcal{P}(V, \hat{\nabla}V))$$
and the initial condition makes sense in view of (\ref{V = U in C^1}). 

\bigskip

Since $\|V\|_{\mathcal{X}_{k,\gamma}^{(\delta)}(M\times [0,T])}\leq K$, the estimate (\ref{P estimate}) implies that
\begin{align}{\label{tilde P estimate 1}}
	\|\hat{\nabla}_i(\mathcal{P}(V, \hat{\nabla}V))\|_{\mathcal{C}^{k-2,\gamma}_{\frac{3}{2}-\frac{\beta}{2}-\frac{\delta}{2}}(M\times (0,T])} \leq K.
\end{align}
It is also easy to see that
\begin{align}{\label{tilde P estimate 2}} 
	\|g^{-1}*\hat{R}*\hat{\nabla}V + g^{-1}*\hat{\nabla}\hat{R}*V\|_{\mathcal{C}^{k-2,\gamma}_{\frac{3}{2}-\frac{\beta}{2}-\frac{\delta}{2}}(M\times (0,T])} \leq K.
\end{align}
Moreover, $\ \|V\|_{\mathcal{X}_{k,\gamma}^{(\delta)}(M\times [0,T])}\leq K$ implies that $\|\hat{\nabla}^2V\|_{\mathcal{C}^{k-2,\gamma}_{1-\frac{\delta}{2}}(M\times (0,T])} \leq K$. This gives   
\begin{align}{\label{tilde P estimate 3}}
	\|	\hat{\nabla}g^{-1}*\hat{\nabla}^2V \|_{\mathcal{C}^{k-2,\gamma}_{\frac{3}{2}-\frac{\beta}{2}-\frac{\delta}{2}}(M\times (0,T])}
	&\leq K\ \|\hat{\nabla} g^{-1}\|_{\mathcal{C}^{k-2,\gamma}_{\frac{1}{2}-\frac{\beta}{2}}(M\times (0,T])}\|\hat{\nabla}^2V\|_{\mathcal{C}^{k-2,\gamma}_{1-\frac{\delta}{2}}(M\times (0,T])} 	\\
	&\notag\leq K'.
\end{align}
As $1-\frac{\lambda}{2} = \frac{3}{2}-\frac{\beta}{2}-\frac{\delta}{2}$ and $\lambda <\beta$, we obtain by putting (\ref{tilde P estimate 1}) to (\ref{tilde P estimate 3}) that
\begin{align*}
	\left\| \tilde{\mathcal{P}}(V, \hat{\nabla}V, \hat{\nabla}^2V)\right\|_{\mathcal{C}^{k-2,\gamma}_{1-\frac{\lambda}{2}}(M\times (0,T])} \leq K.
\end{align*}
Since $\|\hat{\nabla}_iU\|_{\lambda;M}\leq K$, we apply Theorem \ref{theorem linear equation} with $\alpha = \lambda$ to the system (\ref{grad V equation}) to obtain  
\begin{align*}
	\|\hat{\nabla}V\|_{\lambda,\frac{\lambda}{2};M\times [0,T]} + \|\hat{\nabla}^2V\|_{\mathcal{C}^{k-1,\gamma}_{\frac{1}{2}-\frac{\lambda}{2}}(M\times (0,T])} \leq K.	
\end{align*}
From this, the assertion follows.

\end{proof}

\bigskip

Now, associated with the differential $d\phi_t$ of the harmonic map flow $\phi_t$, we define the functions 
	\[\Phi_{\mu}^{ij}: M\times [0,T] \to \mathbb{R}\]
	as in Definition \ref{map Psi}. Then Theorem \ref{thm harmonic map} and Lemma \ref{lem pullback harmonic} imply 
	\begin{align}{\label{Phi estimate}}
	\|\Phi_{\mu}^{ij}\|_{\lambda,\frac{\lambda}{2}; M\times [0,T]} + \|\hat{\nabla}\Phi_{\mu}^{ij}\|_{\mathcal{C}^{k-1,\gamma}_{\frac{1}{2}-\frac{\lambda}{2}}( M\times (0,T])} \leq K (\hat{g}, M,  \|g\|_{\mathcal{X}_{k,\gamma}^{(\beta)}}, \|\psi\|_{C^{1,\beta}(M)})
	\end{align}
for $\lambda\in(\gamma, \beta)$.

\bigskip

Next, we show that a solution to the Ricci flow gives rise to a solution to the Ricci-DeTurck flow via the harmonic map heat flow.

\begin{proposition}{\label{prop pullback Ric DeT}}
Let $\lambda\in(\gamma, \beta)$ be given. Let $\{\phi_t\}_{t\in[0,T]}$ be a one-parameter family of diffeomorphisms which evolves under the harmonic map heat flow (\ref{harmonic map equation}). For each $t\in [0,T]$, we define a metric $\tilde{g}(t)$ by $\tilde{g}(t) := (\phi_t^{-1})^*(g(t))$. Then, $\tilde{g}(t)\in \mathcal{X}_{k,\gamma}^{(\lambda)}(M\times[0,T])$ solves the Ricci-DeTurck system (\ref{Ricci DeT}). Moreover, $\tilde{g}$ satisfies the estimate
 \begin{align}
	\|\tilde{g}\|_{\mathcal{X}_{k,\gamma}^{(\lambda)}(M\times[0,T])} \leq K(\hat{g}, M, \|g\|_{\mathcal{X}_{k,\gamma}^{(\beta)}}, \|\psi\|_{C^{1,\beta}(M)}).
\end{align}
\end{proposition}

\begin{proof}
Since 
\begin{align*}
    \Delta_{g(t),\hat{g}}\phi_t(p) = \Delta_{(\phi_t)^*\tilde{g}(t),\ \hat{g}} \phi_t(p) = \Delta_{\tilde{g}(t),\ \hat{g}} \id(\phi_t(p)) =  -W(\phi_t(p))
\end{align*}
for all $(p,t)\in M\times(0,T]$, where $W^k =\tilde{g}^{ij}(\Gamma_{ij}^k(\tilde{g})-\Gamma_{ij}^k(\hat{g}))$, we have
	\begin{align*}
	\frac{\partial}{\partial t}(\tilde{\phi_t})^*\tilde{g}(t)\Big|_p &= (\phi_t)^*	\Big(\frac{\partial}{\partial t}\tilde{g}(t)\Big|_p + \frac{d}{ds}\Big|_{s=0}(\varphi_s)^*\tilde{g}(t)\Big|_p \Big)\\
	&= (\phi_t)^*\Big(\frac{\partial}{\partial t}\tilde{g}(t) - \mathcal{L}_W\tilde{g}(t)\Big)\Big|_p.
	\end{align*}
	But since $\frac{\partial}{\partial t}g(t) = -2Ric(g(t))$, by the diffeomorphism invariance of Ricci curvature we obtain
	\begin{align*}
	\frac{\partial}{\partial t}\tilde{g}(t) - \mathcal{L}_W\tilde{g}(t) = -2Ric(\tilde{g}(t)).	
	\end{align*}
	Hence $\tilde{g}$ satisfies the Ricci-DeTurck system with initial condition $\tilde{g}(0) = g_0$. Lastly, since Theorem \ref{thm harmonic map} and Lemma \ref{lem pullback harmonic} imply 
	\[
	\|\Phi_{\mu}^{ij}\|_{\lambda,\frac{\lambda}{2}; M\times [0,T]} + \|\hat{\nabla}\Phi_{\mu}^{ij}\|_{\mathcal{C}^{k-1,\gamma}_{\frac{1}{2}-\frac{\lambda}{2}}( M\times (0,T])} \leq K (\hat{g}, M,  \|g\|_{\mathcal{X}_{k,\gamma}^{(\beta)}}, \|\psi\|_{C^{1,\beta}(M)}),
	\]
 we have $\tilde{g}\in C^{\lambda,\frac{\lambda}{2}}$, and we can then proceed as in Proposition \ref{prop pullback g} to obtain that $\tilde{g}\in \mathcal{X}_{k,\gamma}^{(\lambda)}(M\times [0,T])$ and the desired estimates.
\end{proof}

\newpage
\section{Proof of the Main Theorems}

\bigskip

We first prove  the existence of a canonical solution to Ricci flow on the doubled manifold with $C^{\alpha}$ initial metric. The following result implies Main Theorem 1 and Main Theorem 2.

\begin{theorem}
	Let $\tilde{M}$ be a closed compact smooth manifold and $\tilde{g}_0\in C^{\alpha}(\tilde{M})$ be a Riemannian metric for some $\alpha\in (0,1)$. Let $k\geq 2$, $\gamma\in (0,\alpha)$ and $\beta\in (\gamma,\alpha)$ be given. Then there exists a $C^{1,\beta}$ diffeomorphism $\psi$ and $T = T(\tilde{M}, \hat{g}, \|\tilde{g}_0\|_{\alpha})$, $K = K(\tilde{M}, k, \hat{g},  \|\tilde{g}_0\|_{\alpha})$ such that the following holds:\\
	
	There is a solution $g(t)\in\mathcal{X}_{k,\gamma}^{(\beta)}(\tilde{M}\times[0,T])$ to the Ricci flow
	\begin{align*}
		\frac{\partial}{\partial t}g(t)=-2Ric(g(t))\quad&\text{on}\quad \tilde{M}\times(0,T]
	\end{align*}
	such that $g(0) = \psi^*\tilde{g}_0$ and
	 $$\|g\|_{\mathcal{X}_{k,\gamma}^{(\beta)}(\tilde{M}\times[0,T])}\leq K.$$	

Moreover, the solution is canonical in the sense that if $\bar{g}(t)\in\mathcal{X}_{k,\gamma}^{(\beta)}(\tilde{M}\times[0,T])$ is another solution to the Ricci flow and $\bar{\psi}\in C^{1,\beta}(\tilde{M})$ is a diffeomorphism such that $\bar{g}(0) = \bar{\psi}^*\tilde{g}_0$, then there is a $C^{k+1}$ diffeomorphism $\phi$ such that $\bar{g}(t) = \phi^*g(t)$ which satisfies $\bar{\psi} = \psi\circ\phi$.
\end{theorem}

\bigskip

\begin{proof}
By Proposition \ref{prop pullback g}, we can find  $T=T(\tilde{M}, \hat{g}, \|\tilde{g}_0\|_{\alpha})>0$ sufficiently small and $K = K(\tilde{M}, k, \hat{g},  \|\tilde{g}_0\|_{\alpha})$ and a $C^{1,\beta}$ diffeomorphism $\psi$ such that there is a solution $g(t)$ to the Ricci flow satisfying the estimate 
\begin{align*}
	\|g\|_{\mathcal{X}_{k,\gamma}^{(\beta)}}\leq K.
\end{align*}
and $g(0) = \psi^*\tilde{g}_0$. \\

Now, suppose that $\bar{g}(t)\in \mathcal{X}_{k,\gamma}^{(\beta)}(\tilde{M}\times[0,T])$ is another solution to the Ricci flow and $\bar{\psi}\in C^{1+\beta}(\tilde{M})$ is a diffeomorphism such that $\bar{g}(0) = \bar{\psi}^*\tilde{g}_0$. Let $\phi_t$ denote the solution to the harmonic map heat flow starting from $\phi_0 = \psi$ with respect to $g(t)$ and $\hat{g}$. Let $\bar{\phi}_t$ denote the solution to the harmonic map heat flow starting from $\bar{\phi}_0 = \psi$ with respect to $\bar{g}(t)$ and $\hat{g}$. By Proposition \ref{prop contraction harmonic} these solutions exist and are unique. Subsequently Theorem \ref{thm harmonic map} and Proposition \ref{prop pullback Ric DeT} imply that
\begin{align*}
	h(t) := (\phi_t^{-1})^*g(t) \quad\text{and}\quad \bar{h}(t) := (\bar{\phi}_t^{-1})^*\bar{g}(t)
\end{align*}
are both solutions to the Ricci-DeTurck flow (\ref{Ricci DeT}) such that $h, \bar{h}\in\mathcal{X}_{k,\gamma}^{(\lambda)}(\tilde{M}\times[0,T])$ for some $\lambda\in (\gamma, \beta)$ and $h(0) = \bar{h}(0) = \tilde{g}_0$. Since the solution to the Ricci-DeTurck flow in $\mathcal{X}_{k,\gamma}^{(\lambda)}$ are unique by Theorem \ref{thm existence Ricci DeT}, we have $h(t) = \bar{h}(t)$ on $M\times [0,T]$. Now observe that \begin{align*}
	\Delta_{g(t),\hat{g}} \phi_t \Big|_{\psi^{-1}(p)} = 	\Delta_{h(t),\hat{g}}\id \Big|_{\phi_t\circ\psi^{-1}(p)},
\end{align*}
hence $\phi_t\circ\psi^{-1}$ is a solution to the ODE
\begin{align}
\begin{cases}
	&\frac{\partial}{\partial t}\phi_t\circ\psi^{-1}(p) = 	\Delta_{h(t),\hat{g}}\id\big|_{\phi_t\circ\psi^{-1}(p)}\\
	&\varphi_0\circ\psi^{-1} = \id.
\end{cases}
\end{align}
Analogously, $\bar{\varphi}_t\circ\bar{\psi}^{-1}$ satisfies
\begin{align}
\begin{cases}
	&\frac{\partial}{\partial t}\bar{\varphi}_t\circ\bar{\psi}^{-1}(p) = 	\Delta_{\bar{h}(t),\hat{g}}\id\big|_{\bar{\varphi}_t\circ\bar{\psi}^{-1}(p)}\\
	&\bar{\varphi}_0\circ\bar{\psi}^{-1} = \id.
\end{cases}
\end{align}
Since we know that $h(t)=\bar{h}(t)$, it follows that $\phi_t\circ\psi^{-1}$ and $\bar{\phi}_t\circ\bar{\psi}^{-1}$ satisfy the same ODE with the same initial condition. Consequently $\phi_t\circ\psi^{-1} = \bar{\phi}_t\circ\bar{\psi}^{-1}$ on $M\times [0,T]$. In other words, $\phi_t^{-1}\circ\bar{\phi}_t = \psi^{-1}\circ\bar{\psi}$ is constant in $t$. Let us take the desired diffeomorphism $\phi$ to be $\phi := \phi_t^{-1}\circ\bar{\phi}_T$. Then from Theorem \ref{thm harmonic map} we know that $\phi$ is a $C^{k+1}$ diffeomorphism satisfying $\bar{\psi} = \psi\circ\phi$ and 
 \begin{align*}
 	\bar{g}(t) = (\phi_t^{-1}\circ\bar{\phi}_t)^*g(t) = \phi^*g(t).	
 \end{align*}	
\end{proof}

The previous theorem implies Main Theorem 3 and Main Theorem 4 via doubling:
\bigskip

\begin{theorem}[Main Theorem 3]
	Let $(M,g_0)$ be a compact smooth Riemannian manifold with boundary. Let $k\geq 2$, $\beta\in (0,1)$, $\gamma\in (0,\beta)$ and $\epsilon\in (0, 1-\beta)$ be given. Then there exists a $C^{1,\beta}$ diffeomorphism $\psi$ and $T = T(M, \hat{g}, \|g_0\|_{\beta+\varepsilon})$, $K = K(M, k, \hat{g},  \|g_0\|_{\beta+\varepsilon})$ such that the following holds:\\
	
	There is a solution $g(t)\in\mathcal{X}_{k,\gamma}^{(\beta)}(M\times[0,T])$ to the Ricci flow on manifold with boundary
\begin{align}{\label{Ricci flow boundary equation}}
    \begin{cases}
    \frac{\partial}{\partial t}g(t)=-2Ric(g(t))\quad&\text{on}\quad M\times(0,T]\\
    A_{g(t)}=0\quad\quad\quad\quad\quad &\text{on}\quad \partial M\times (0, T]  
    \end{cases}
\end{align}
	such that $g(0) = \psi^*g_0$ and
	 $$\|g\|_{\mathcal{X}_{k,\gamma}^{(\beta)}(M\times[0,T])}\leq K.$$	
For each $t>0$, the metric $g(t)$ extends smoothly to the doubled manifold $\tilde{M}$ of $M$, and the doubled metric lies in $\mathcal{X}_{k,\gamma}^{(\beta)}(\tilde{M}\times[0,T])$. The diffeomorphism $\psi$ also extends to a $C^{1,\beta}$ diffeomorphism on the doubled manifold. 
\end{theorem}

\bigskip

\begin{proof}
Let $\tilde{M}$ be the doubling of $M$, and $\tilde{g}_0$ be the extension of $g_0$ in $\tilde{M}$ via reflection about the boundary $\partial M$. Since $\tilde{g}_0$ is Lipschitz, we can find $\beta+\varepsilon\in (\beta,1)$ such that $\tilde{g}_0\in C^{\beta+\varepsilon}(M)$. Fix a smooth background metric $\hat{g}$ on $M$ such that in a small collar neighborhood of $\partial M$ the metric $\hat{g}$ is isometric to a product $\partial M\times [0,\varepsilon]$.

By Proposition \ref{prop pullback g}, we can find  $T=T(\tilde{M}, \hat{g}, \|\tilde{g}_0\|_{\beta+\varepsilon})>0$ sufficiently small and $K = K(\tilde{M}, k, \hat{g},  \|\tilde{g}_0\|_{\beta+\varepsilon})$ and a $C^{1,\beta}$ diffeomorphism $\tilde{\psi}$ such that there is a solution $\tilde{g}(t)$ to the Ricci flow on $\tilde{M}\times (0,T]$ satisfying the estimate 
\begin{align*}
	\|\tilde{g}\|_{\mathcal{X}_{k,\gamma}^{(\beta)}}\leq K.
\end{align*}
and $\tilde{g}(0) = \tilde{\psi}^*\tilde{g}_0$.

Note that by uniqueness and diffeomorphism invariance the solution to the Ricci-DeTurck flow on the doubled manifold $\tilde{M}$ with initial metric $\tilde{g}_0$ is invariant under the natural natural $\mathbb{Z}_2$ action given by the reflection about $\partial M$ which switches the two halves of $\tilde{M}$. Thus the diffeomorphism $\tilde{\psi}$ we obtained by solving the ODE (\ref{contraction w 4}) is equivariant under this $\mathbb{Z}_2$ action. This implies that the solution $\tilde{g}(t)$ to the Ricci flow is also invariant under the $\mathbb{Z}_2$ action. In particular, it follows from the $\mathbb{Z}_2$ symmetry that $g = \tilde{g}|_{M}\in \mathcal{X}_{k,\gamma}^{(\beta)}(M\times[0,T])$ has totally geodesic boundary on $M$, and thus a solution to (\ref{Ricci flow boundary equation}). Moreover, $\psi = \tilde{\psi}|_{M}:M\to M$ is a $C^{1,\beta}$ diffeomorphism such that $g(0) = \psi^*g_0$. 
\end{proof}

\bigskip

\begin{theorem}[Main Theorem 4]
	Let $(M,g_0)$ be a compact smooth Riemannian manifold with boundary. Suppose that the pairs $(g^1(t), \psi^1)$ and $(g^2(t), \psi^2)$ satisfy the conclusion of Theorem 6.2. Then there exists a $C^{k+1}$ diffeomorphism $\varphi:M\to M$ such that $\varphi$ extends to a $C^{k+1}$ diffeomorphism on the doubled manifold and
 	 $$g^2(t) = \varphi^*(g^1(t)),$$
 	 where $\psi^2 = \psi^1\circ\varphi$. In particular, $((\psi^1)^{-1})^*g^1(t) = ((\psi^2)^{-1})^*g^2(t)$.
\end{theorem}

\begin{proof}
	By assumptions, the pairs $(g^1(t), \psi^1)$ and $(g^2(t), \psi^2)$ can be extended to the doubled manifold so that Theorem 6.1 can be applied. Since the whole construction of the harmonic map heat flow is invariant under the natural $\mathbb{Z}_2$ action given by reflection, the assertion thus follows from restricting the conclusion of Theorem 6.1 to $M$.
\end{proof}

\bigskip

We can show that various important curvature conditions are preserved along the Ricci flow on manifolds with boundary in our formulation. We first prove a compactness result.
\bigskip
\begin{lemma}{\label{lem precompact}}
Let $\alpha,\beta\in (0,1)$ be given such that $\beta<\alpha$. Moreover, let $\gamma\in (0,\alpha)$ and $\delta\in (0,\beta)$ be given such that $\delta<\gamma$. Then the bounded subsets in the space $\mathcal{C}_{\frac{1}{2}-\frac{\alpha}{2}}^{k,\gamma}(M\times (0,T])$ are precompact in the space  $\mathcal{C}_{\frac{1}{2}-\frac{\beta}{2}}^{k,\delta}(M\times (0,T])$. 
\end{lemma}
\begin{proof}
Note that by Lemma \ref{properties of weighted space} we have the inclusion $\mathcal{C}_{\frac{1}{2}-\frac{\alpha}{2}}^{k,\gamma}\subset \mathcal{C}_{\frac{1}{2}-\frac{\beta}{2}}^{k,\delta} $.
Let $\{\eta_j\}$ be a bounded sequence in $\mathcal{C}_{\frac{1}{2}-\frac{\alpha}{2}}^{k,\gamma}(M\times (0,T])$ so that $\|\eta_j\|_{\mathcal{C}_{\frac{1}{2}-\frac{\alpha}{2}}^{k,\gamma}}\leq L$. In particular,
\begin{align*}
   \|\hat{\nabla}^k\eta_j\|_{0;M\times [\sigma/2,\sigma]}\leq \frac{L}{\sigma^{\frac{1}{2}-\frac{\alpha}{2}+\frac{k}{2}}}\quad\text{and}\quad  [\hat{\nabla}^k\eta_j]_{\gamma,\frac{\gamma}{2};M\times [\sigma/2,\sigma]}\leq \frac{L}{\sigma^{\frac{1}{2}-\frac{\alpha}{2}+\frac{k}{2}+\frac{\gamma}{2}}}
\end{align*}
for any $\sigma\in (0,T]$. Thus $\{\hat{\nabla}^k\eta_j\}$ is equicontinuous on $M\times [\sigma/2,\sigma]$ and uniformly bounded. By Arzela-Ascoli it contains a subsequence $\{\tilde{\eta}_j\}$ such that $\{\hat{\nabla}^k\tilde{\eta}_j\}$ is uniformly convergent on compact subsets of $M\times [\sigma/2,\sigma]$. By the same argument we can also proceed to find a subsequence of $\{\tilde{\eta}_j\}$ which is uniformly convergent on bounded subsets of $M\times [\sigma/2,\sigma]$ together with its lower order $x$-derivatives. Moreover, by a standard diagonal argument we can further find a subsequence which is uniformly convergent on bounded subsets of $M\times (0,T]$. We still denote this subsequence by $\{\tilde{\eta}_j\}$ and its limit by $\eta$. Thus $\eta$ satisfies the same bounds as $\eta_j$ given above.  Define $\zeta_j = \tilde{\eta}_j-\eta$. To finish the proof we will show that $\zeta_j\to 0$ in $\mathcal{C}_{\frac{1}{2}-\frac{\beta}{2}}^{k,\delta}(M\times (0,T])$. For any $0\leq r\leq k$, we have
\begin{align}{\label{zeta C^0}}
    \sigma^{\frac{1}{2}-\frac{\beta}{2}+\frac{r}{2}}\|\hat{\nabla}^r\zeta_j\|_{0;M\times [\sigma/2,\sigma]}\leq T^{\frac{\alpha-\beta}{2}}\sigma^{\frac{1}{2}-\frac{\alpha}{2}+\frac{r}{2}}\|\hat{\nabla}^r\zeta_j\|_{0;M\times [\sigma/2,\sigma]}.
\end{align}
Moreover,  
\begin{align}{\label{zeta gamma bound}}
    \sigma^{\frac{1}{2}-\frac{\alpha}{2}+\frac{r}{2} + \frac{\gamma}{2}}[\hat{\nabla}^r\zeta_j]_{\gamma,\frac{\gamma}{2};M\times [\sigma/2,\sigma]}\leq  \|\zeta_j\|_{\mathcal{C}_{\frac{1}{2}-\frac{\alpha}{2}}^{k,\gamma}(M\times (0,T])}\leq 2L
\end{align}
for any $\sigma\in (0,T]$. Now, we choose a local chart $U$ in $M$ and let $x,y\in U$ so that the (0,2)-tensors $\zeta_j$ are represented in local coordinates as $\zeta_j = \sum_{m=1}^N\zeta_{jm}\bold{e}_m$. Let $\epsilon>0$ be an arbitrary small number, we may assume without loss of generality that $\epsilon<\sigma$.\\
 If $d((x,t),(y,s))<\epsilon$, where the distance $d((x,t),(y,s)) = |x-y|+|t-s|^{1/2}$ is taken within the chart, then (\ref{zeta gamma bound}) implies
\begin{align}{\label{zeta Holder 1}}
    &\sigma^{\frac{1}{2} - \frac{\beta}{2}+ \frac{r}{2} + \frac{\delta}{2}}\frac{|\hat{\nabla}^r\zeta_{jm}(x,t)-\hat{\nabla}^r\zeta_{jm}(y,s)|}{d((x,t),(y,s))^{\delta}}\\
    &\notag = \sigma^{\frac{1}{2} - \frac{\alpha}{2}+ \frac{r}{2} + \frac{\gamma}{2}}\frac{|\hat{\nabla}^r\zeta_{jm}(x,t)-\hat{\nabla}^r\zeta_{jm}(y,s)|}{d((x,t),(y,s))^{\gamma}}\cdot d((x,t),(y,s))^{\gamma-\delta} \sigma^{\frac{\alpha}{2} - \frac{\beta}{2}+ \frac{\delta}{2} - \frac{\gamma}{2}} \\
    &\notag \leq 2L T^{\frac{\alpha-\beta}{2}}\epsilon^{\frac{\gamma-\delta}{2}}.
\end{align}
If $d((x,t),(y,s))\geq\epsilon$, then  
\begin{align}{\label{zeta Holder 2}}
     &\sigma^{\frac{1}{2} - \frac{\beta}{2}+ \frac{r}{2} + \frac{\delta}{2}}\frac{|\hat{\nabla}^r\zeta_{jm}(x,t)-\hat{\nabla}^r\zeta_{jm}(y,s)|}{d((x,t),(y,s))^{\delta}} \leq K(T)\|\hat{\nabla}^r\zeta_j\|_{0;M\times[\sigma/2,\sigma]}\epsilon^{-\delta}.\end{align}
Combining (\ref{zeta Holder 1}) and (\ref{zeta Holder 2}), we obtain
\begin{align}{\label{zeta Holder 3}}
   \sigma^{\frac{1}{2} - \frac{\beta}{2}+ \frac{r}{2} + \frac{\delta}{2}}[\hat{\nabla}^r\zeta_j]_{\delta,\frac{\delta}{2};M\times [\sigma/2,\sigma]}\leq  K(T)(2L\epsilon^{\frac{\gamma-\delta}{2}} + \|\hat{\nabla}^r\zeta_j\|_{0;M\times[\sigma/2,\sigma]}\epsilon^{-\delta}).
\end{align}
By taking $j\to\infty$ and then $\epsilon\to 0$ in (\ref{zeta Holder 3}), we obtain
\begin{align}{\label{zeta Holder 4}}
	\lim_{j\to\infty}\sigma^{\frac{1}{2} - \frac{\beta}{2}+ \frac{r}{2} + \frac{\delta}{2}}[\hat{\nabla}^r\zeta_j]_{\delta,\frac{\delta}{2};M\times [\sigma/2,\sigma]} = 0.
\end{align}
 for each $\sigma\in (0,T]$. Therefore, (\ref{zeta C^0}) and (\ref{zeta Holder 4}) conclude that
\begin{align}
    \lim_{j\to\infty}\|\zeta_j\|_{\mathcal{C}_{\frac{1}{2}-\frac{\beta}{2}}^{k,\delta}(M\times (0,T])} = 0
\end{align}
\end{proof}

\bigskip

\begin{theorem}[Main Theorem 5]
Suppose that $g(t)$ is a canonical solution to (\ref{Ricci flow boundary equation}) on $M\times [0,T]$ given by Theorem 6.2. Then the following holds:\\
 If $(M,g_0)$ has a convex boundary, then
\begin{itemize}
    \item[(i)] $(M,g_0)$ has positive curvature operator $\implies$ $(M,g(t))$ has positive curvature operator;
    \item[(ii)] $(M,g_0)$ is PIC1 $\implies$ $(M,g(t))$ is PIC1;
    \item[(iii)] $(M,g_0)$ is PIC2 $\implies$ $(M,g(t))$ is PIC2.
\end{itemize}
If $(M,g_0)$ has a two-convex boundary, then
\begin{itemize}
    \item[(iv)] $(M,g_0)$ is PIC $\implies$ $(M,g(t))$ is PIC.
\end{itemize}
Moreover, if $(M,g_0)$ has a mean-convex boundary, then
\begin{itemize}
    \item[(v)] $(M,g_0)$ has positive scalar curvature $\implies$ $(M,g(t))$ has positive scalar curvature.
\end{itemize}
\end{theorem}

\begin{proof}
By Theorem 2 in \cite{Chow}, there is a family of smooth Riemannian metrics $\{\hat{g}_{\lambda}\}_{\lambda>\lambda^*}$ on $M$ which converges to $g_0$ in $C^{\alpha}$ for any $\alpha\in [0,1)$ and  satisfies:
\begin{itemize}
\item[(i)] $(M, \hat{g}_{\lambda})$ has a totally geodesic boundary.
\item[(ii)] If $(M,g_0)$ has a convex boundary, then
	\begin{itemize}
	\item $(M,g_0)$ has positive curvature operator $\implies$ $(M,\hat{g}_{\lambda})$ has positive curvature operator;
    \item $(M,g_0)$ is PIC1 $\implies$ $(M,\hat{g}_{\lambda})$ is PIC1;
    \item $(M,g_0)$ is PIC2 $\implies$ $(M,\hat{g}_{\lambda})$ is PIC2. 
	\end{itemize}	
\item[(iii)] If $(M,g_0)$ has a two-convex boundary, then
	\begin{itemize}
	\item $(M,g_0)$ is PIC $\implies$ $(M,\hat{g}_{\lambda})$ is PIC.
	\end{itemize}
\item[(iv)] If $(M,g_0)$ has a mean-convex boundary, then
	\begin{itemize}
    \item $(M,g_0)$ has positive scalar curvature $\implies$ $(M,\hat{g}_{\lambda})$ has positive scalar curvature.
	\end{itemize}
\end{itemize}
Note that by Corollary 8 in \cite{Chow} the positivity conditions in the above statement are uniform in all sufficiently large $\lambda$. 
\bigskip

Let $\tilde{M}$ be the doubled manifold of $M$, and fix a background metric $\hat{g}$ such that in a small collar neighborhood of $\partial M$ the metric is isometric to $\partial M\times [0,\varepsilon]$. We extend the metrics $g_0$ to $\tilde{M}$ via reflection about the boundary $\partial M$. Moreover, from the construction in \cite{Chow} in a neighborhood of the boundary $\hat{g}_{\lambda}$ has the form of a product metic, thus the metric $\hat{g}_{\lambda}$ can be extended smoothly to the doubled manifold $\tilde{M}$.  Then $g_0$ is a Lipschitz metric on $\tilde{M}$ and $\hat{g}_{\lambda}$ is a smooth metric on $\tilde{M}$.\\

 Now, we assume that $(M,g_0)$ has convex boundary. If $(M,g_0)$ has positive curvature operator/ PIC1/ PIC2, then $(\tilde{M},\hat{g}_{\lambda})$ has positive curvature operator/ PIC1/ PIC2 for all sufficiently large $\lambda>0$ and these positivity conditions are uniformly bounded below for $\lambda$ large. Let $\tilde{g}_{\lambda}(t)$ be the solution to Ricci-DeTurck flow on $\tilde{M}\times (0,T]$ starting from $\tilde{g}_{\lambda}(0) = \hat{g}_{\lambda}$. Let $\phi_t$ solves the harmonic map heat flow so that $\tilde{g}(t) = (\phi_t^{-1})^*g(t)$ is a solution to Ricci-DeTurck flow on $\tilde{M}\times (0,T]$ starting from $\tilde{g}(0) = g_0$. Proposition \ref{prop pullback Ric DeT} then implies that $\tilde{g}(t)\in\mathcal{X}_{2,\gamma}^{(\alpha)}$ for some exponent $\alpha\in (0,1)$ and $\gamma\in (0,\alpha)$. Next we apply Theorem \ref{thm existence Ricci DeT} to obtain the estimates
 \begin{align*}
 	\|\tilde{g}_{\lambda}(t)\|_{\mathcal{X}_{2,\gamma}^{(\alpha)}(M\times[0,T])} \leq K(M, \hat{g}, \|\hat{g}_{\lambda}\|_{\alpha;M}).
 \end{align*}
Since $\hat{g}_{\lambda}\to g_0$ in $C^{\alpha}(M)$, we can find a uniform constant $K(M, \hat{g}, \|g_0\|_{\alpha;M})$ such that
 \begin{align}{\label{last estimate}}
 	\|\tilde{g}_{\lambda}(t)\|_{\mathcal{X}_{2,\gamma}^{(\alpha)}(M\times[0,T])} \leq K(M, \hat{g}, \|g_0\|_{\alpha;M}).
 \end{align}

Next, we pick some $\beta\in (0,\alpha)$ and $\delta\in (0,\gamma)$ such that $\delta<\beta$. Since bounded subsets in  $C^{\alpha,\frac{\alpha}{2}}(M\times[0,T])$ are precompact in $C^{\beta,\frac{\beta}{2}}(M\times[0,T])$, Lemma \ref{lem precompact} then implies that bounded subsets in the Banach space $\mathcal{X}_{2,\gamma}^{(\alpha)}(M\times[0,T])$ are precompact in $\mathcal{X}_{2,\delta}^{(\beta)}(M\times[0,T])$. Now, by the estimate (\ref{last estimate}) the metrics $\tilde{g}_{\lambda}$ are uniformly bounded in $\mathcal{X}_{2,\gamma}^{(\alpha)}(M\times[0,T])$ for all large $\lambda$. Upon passing to a subsequence we have $\tilde{g}_{\lambda}\to \tilde{g}_{\infty}$ in $\mathcal{X}_{2,\delta}^{(\beta)}(M\times[0,T])$  as $\lambda\to\infty$. By the continuity of coefficients in the Ricci-DeTurck flow, $\tilde{g}_{\infty}(t)$ is also a solution to the Ricci-DeTurck flow with $\tilde{g}_{\infty}(0) = g_0$ such that  
 \begin{align*}
 	\|\tilde{g}_{\infty}(t)\|_{\mathcal{X}_{2,\delta}^{(\beta)}(M\times[0,T])} \leq K(M, \hat{g}, \|g_0\|_{\alpha;M}).
 \end{align*}
Since $\delta<\beta$, Theorem \ref{thm existence Ricci DeT} implies that such a solution is unique, therefore we have $\tilde{g}_{\infty}(t) = \tilde{g}(t)$. On the other hand, since $\hat{g}_{\lambda}$ are smooth metrics, the curvature conditions are preserved under the Ricci flow, thus $\tilde{g}_{\lambda}(t)$ also has positive curvature operator/ PIC1/ PIC2 for all sufficiently large $\lambda$ by diffeomorphism invariance of curvatures. This in particular implies that $\tilde{g}(t) = \tilde{g}_{\infty}(t)$ also has positive curvature operator/ PIC1/ PIC2 for each $t>0$. Therefore, since $g(t) = (\phi_t)^*\tilde{g}(t)$, by diffeomorphism invariance of curvatures these curvature conditions also hold for $g(t)$ for each $t>0$. This proves statements (i) to (iii) in the theorem. Statements (iv) and (v) can be proved similarly. The theorem then follows.

\end{proof}

\end{document}